\documentclass[12pt]{amsart}

\usepackage{tikz}
\usepackage{bbm}
\usetikzlibrary{3d,arrows,calc,positioning,decorations.pathreplacing,matrix} %,arrows.meta}

\usepackage{calligra,mathrsfs}
\usepackage[all]{xy}
\usepackage{float, comment}
\usepackage{mathtools}
\usepackage{amsmath}
\usepackage{amsthm}
\usepackage{amssymb}
\usepackage{amsbsy}
\usepackage{amstext} 
\usepackage{amsopn}
\usepackage[mathscr]{eucal}
\usepackage{enumerate}
\usepackage{xcolor}
\usepackage{graphicx} % allows \includegraphics{}s
\usepackage{scalerel}
\usepackage{microtype} % improves formattings
\usepackage[margin=1in,marginparwidth=0.8in, marginparsep=0.1in]{geometry}
\usepackage[bookmarks=true, bookmarksopen=true, bookmarksdepth=3,bookmarksopenlevel=2, colorlinks=true, linkcolor=blue, citecolor=blue, filecolor=blue, menucolor=blue, urlcolor=blue]{hyperref}
\usepackage{tikz}
\usepackage{bbm}
\usepackage[all]{xy}
\usetikzlibrary{arrows,calc,positioning,decorations.pathreplacing} %,arrows.meta}

\DeclareMathOperator{\cHom}{\mathscr{H}\text{\kern -3pt {\calligra\large om}}\,}

\numberwithin{equation}{section}

\newtheorem{Proposition}[equation]{Proposition} 
\newtheorem{Lemma}[equation]{Lemma}

\newtheorem{Corollary}[equation]{Corollary}

\theoremstyle{definition}
\newtheorem{Remark}[equation]{Remark}
\newtheorem{Example}[equation]{Example}
\newtheorem{Definition}[equation]{Definition}

\numberwithin{figure}{section}

\def\cl{{\mathrm{cl}}}
\def\cn{\leq 0}

\newcommand{\A}{\mathbb{A}}
\newcommand{\C}{\mathbb{C}}

\renewcommand{\P}{\mathbb{P}}
\newcommand{\Q}{\mathbb{Q}}
\newcommand{\Z}{\mathbb{Z}}

\newcommand{\cF}{\mathcal{F}}
\newcommand{\cG}{\mathcal{G}}
\newcommand{\cH}{\mathcal{H}}
\newcommand{\cK}{\mathcal{K}}

\newcommand{\cN}{\mathcal{N}}
\newcommand{\cO}{\mathcal{O}}

\newcommand{\catC}{{\mathscr{C}}}
\newcommand{\catD}{{\mathscr{D}}}
\newcommand{\catE}{{\mathscr{E}}}

\newcommand{\bul}{\bullet}
\newcommand{\ga}{\gamma}

\newcommand{\eop}{\overline{e}}

\newcommand{\ot}{\otimes}

\newcommand{\wh}{\widehat}
\newcommand{\into}{\hookrightarrow}

\newcommand{\id}{{id}}

\newcommand{\hR}{{\mathcal{R}}}

\newcommand{\QCoh}{\mathrm{QCoh}}
\newcommand{\Coh}{\mathrm{Coh}}

\newcommand{\IndCoh}{\mathrm{IndCoh}}

\newcommand{\Fun}{\mathrm{Fun}}

\newcommand{\PreStkk}{\mathrm{PStk}_{\kk}}
\newcommand{\PreStkkleqn}{\mathrm{PStk}_{\kk, \leq n}}
\newcommand{\PreStkkconv}{\widehat{\mathrm{PStk}}_{\kk}}
\newcommand{\Stk}{\mathrm{Stk}}
\newcommand{\Stkk}{\mathrm{Stk}_\kk}
\newcommand{\Stkkleqn}{\mathrm{Stk}_{\kk, \leq n}}
\newcommand{\Stkkleqzero}{\mathrm{Stk}_{\kk, \leq 0}}

\newcommand{\Stkkconv}{\widehat{\mathrm{Stk}}_{\kk}}
\newcommand{\oneStkkconv}{1\text{-}\widehat{\mathrm{Stk}}_{\kk}}
\newcommand{\Spc}{\mathscr{S}}

\newcommand{\Fin}{\mathrm{Fin}}
\newcommand{\Ind}{\mathrm{Ind}}
\newcommand{\CAlg}{\mathrm{CAlg}}
\newcommand{\CAlgk}{\mathrm{CAlg}_\kk}
\newcommand{\CAlgkleqn}{\tau_{\leq n}\mathrm{CAlg}_{\kk}}
\newcommand{\CAlgkleqinfty}{\tau_{< \infty} \mathrm{CAlg}_{\kk}}
\newcommand{\CAlgvarleqn}[1]{\tau_{\leq n}\mathrm{CAlg}_{#1}}
\newcommand{\CAlgvarleqinfty}[1]{\tau_{< \infty}\mathrm{CAlg}_{#1}}

\newcommand{\Corr}{\mathrm{Corr}}
\newcommand{\CORR}{\mathrm{CORR}}

\newcommand{\pt}{\mathrm{pt}}
\newcommand{\op}{\mathrm{op}}

\newcommand{\Mod}{\mathrm{Mod}}
\newcommand{\kk}{k}
\newcommand{\indGStk}{\mathrm{indGStk}}
\newcommand{\indGStkk}{\mathrm{indGStk}_{\kk}}
\newcommand{\indGStkkreas}{\mathrm{indGStk}_{\kk}^{reas}}
\newcommand{\indGStkkcoh}{\mathrm{indGStk}_{\kk}^{coh}}
\newcommand{\indGStkkNoeth}{\mathrm{indGStk}_{\kk}^{Noeth}}
\newcommand{\indGStkkprop}{\indGStk_{\kk,\,prop}}
\newcommand{\indGStkkfcd}{\indGStk_{\kk,\,fcd}}
\newcommand{\indGStkkcl}{\indGStk_{\kk,\,cl}}
\newcommand{\indGStkkreasprop}{\indGStk_{\kk,\,prop}^{reas}}
\newcommand{\indGStkkreasclafp}{\indGStk_{\kk,\,cl,\,a\!f\!p}^{reas}}
\newcommand{\indGStkkcohclafp}{\indGStk_{\kk,\,cl,\,a\!f\!p}^{coh}}
\newcommand{\GStk}{\mathrm{GStk}}
\newcommand{\GStkk}{\mathrm{GStk}_{\kk}}

\newcommand{\GStkkplus}{\mathrm{GStk}_{\kk}^+}
\newcommand{\GStkkpluscoh}{\mathrm{GStk}_{\kk}^{coh,+}}
\newcommand{\GStkkplusprop}{\GStk_{\kk,\,prop}^+}
\newcommand{\GStkkpluscl}{\GStk_{\kk,\,cl}^+}
\newcommand{\GStkkcl}{\GStk_{\kk,\,cl}}
\newcommand{\GStkkplusclafp}{\GStk_{\kk,\,cl,\,a\!f\!p}^+}
\newcommand{\GStkkpluscohclafp}{\GStk_{\kk,\,cl,\,a\!f\!p}^{coh,+}}
\newcommand{\AlgSpk}{\mathrm{AlgSp}_{\kk}}
\newcommand{\PrSt}{\mathcal{P}\mathrm{r}^{\mathrm{St}}}

\newcommand{\LurieCathatinfty}{\widehat{\mathrm{Cat}}_{\infty}}
\newcommand{\Cathatinfty}{\mathrm{Cat}_{\infty}}

\newcommand{\Cathatinftysup}[1]{\mathrm{Cat}{}^{#1}_{\infty}}
\newcommand{\Cathat}{\mathrm{Cat}}
\newcommand{\Catinftytwo}{\mathrm{Cat}_{(\infty,2)}}
\newcommand{\Catinfty}{\mathrm{Cat}_{\infty}}
\newcommand{\CATinfty}{\mathrm{CAT}_{\infty}}

\newcommand{\LFun}{\mathrm{LFun}}
\newcommand{\al}{\alpha}
\newcommand{\be}{\beta}
\newcommand{\ep}{\epsilon}

\newcommand{\Affk}{\mathrm{Aff}_{\kk}}
\newcommand{\Affkftd}{\mathrm{Aff}_{\kk,ftd}}
\newcommand{\PrL}{\mathcal{P}\mathrm{r}^{\mathrm{L}}}
\newcommand{\PrR}{\mathcal{P}\mathrm{r}^{\mathrm{R}}}
\newcommand{\PrStk}{\mathcal{P}\mathrm{r}^{\mathrm{St}}_{\kk}}
\renewcommand{\Pr}{\mathcal{P}\mathrm{r}}
\newcommand{\Maps}{\mathrm{Maps}}
\newcommand{\Map}{\mathrm{Map}}
\newcommand{\Spec}{\mathrm{Spec}\,}

\newcommand{\clalg}[1]{H^0(#1)}

\newcommand{\whC}{\wh{\catC}}
\newcommand{\whD}{\wh{\catD}}
\newcommand{\wcC}{\wc{\catC}}
\newcommand{\wcD}{\wc{\catD}}
\newcommand{\LAd}{\mathrm{LAd}}
\newcommand{\RAd}{\mathrm{RAd}}
\renewcommand{\vert}{fwd}
\newcommand{\horiz}{bkwd}
\newcommand{\isom}{isom}

\newcommand{\PrStbacpl}{\mathcal{P}\mathrm{r}^{\mathrm{St,b}}_{\mathrm{acpl}}}

\newcommand{\PrStbcpl}{\mathcal{P}\mathrm{r}^{\mathrm{St,b}}_{\mathrm{cpl}}}
\newcommand{\PrStbrcpl}{\mathcal{P}\mathrm{r}^{\mathrm{St,b}}_{\mathrm{rcpl}}}

\newcommand{\PrSttexacpl}{\mathcal{P}\mathrm{r}^{\mathrm{St,t\text{-}ex}}_{\mathrm{acpl}}}
\newcommand{\PrSttexrcpl}{\mathcal{P}\mathrm{r}^{\mathrm{St,t\text{-}ex}}_{\mathrm{rcpl}}}

\newcommand{\Grothinftylexrcpl}{\mathrm{Groth}^{\mathrm{lex}}_{\infty}}

\DeclareFontFamily{U}{mathx}{\hyphenchar\font45}
\DeclareFontShape{U}{mathx}{m}{n}{
	<5> <6> <7> <8> <9> <10>
	<10.95> <12> <14.4> <17.28> <20.74> <24.88>
	mathx10
}{}
\DeclareSymbolFont{mathx}{U}{mathx}{m}{n}
\DeclareFontSubstitution{U}{mathx}{m}{n}
\DeclareMathAccent{\widecheck}{0}{mathx}{"71}
\newcommand{\wc}{\widecheck}
\newcommand{\Sp}{\mathrm{Sp}}
\DeclareMathSymbol{\shortminus}{\mathbin}{AMSa}{"39}

\DeclareRobustCommand{\SkipTocEntry}[5]{}

\DeclareMathOperator*{\colim}{colim}

\newcommand{\congto}{\xrightarrow{\sim}}

\newcommand{\arrtip}{latex'}

\begin{document}
\title{Ind-geometric stacks}

\author[Sabin Cautis]{Sabin Cautis}
\address[Sabin Cautis]{University of British Columbia \\ Vancouver BC, Canada}
\email{cautis@math.ubc.ca}

\author[Harold Williams]{Harold Williams}
\address[Harold Williams]{University of Southern California \\ Los Angeles CA, USA}
\email{hwilliams@usc.edu}

\begin{abstract}
	We develop the theory of ind-geometric stacks, in particular their coherent and ind-coherent sheaf theory. This provides a convenient framework for working with equivariant sheaves on ind-schemes, especially in derived settings. Motivating examples include the coherent Satake category, the double affine Hecke category, and related categories in the theory of Coulomb branches. 
\end{abstract}

\maketitle

\setcounter{tocdepth}{1}

\tableofcontents

\section{Introduction}
\thispagestyle{empty}

This paper and \cite{CWtm} provide a collection of foundational results on coherent sheaf theory in infinite-dimensional derived algebraic geometry. These results are applied in \cite{CW2} to the study of Coulomb branches of 4d $\cN=2$ gauge theories, whose mathematical theory was pioneered by Braverman-Finkelberg-Nakajima. In \cite{BFN} these authors associate to a reductive group $G$ and representation $N$ an ind-scheme $\hR_{G,N}$ with an action of the jet group $G_\cO$. The relevant Coulomb branch is the affine variety whose coordinate ring is the equivariant $K$-theory $K^{G_\cO}(\hR_{G,N})$, equipped with a certain convolution product. In \cite{CW2} we study the derived category $\Coh^{G_\cO}(\hR_{G,N})$ of equivariant coherent sheaves and construct a nonstandard t-structure on it, which in turn equips $K^{G_\cO}(\hR_{G,N})$ with a canonical basis.

A central difficulty in carrying this out is that $\hR_{G,N}$ is an object of both non-Noetherian and derived algebraic geometry. 
%both derived and non-Noetherian. 
%of infinite type and an object of derived algebraic geometry. 
The combination of these creates technical complications in coherent sheaf theory (as in Example \ref{ex:selfintAinfty}), and places $\Coh^{G_\cO}(\hR_{G,N})$ beyond the scope of most foundational treatments in the literature. Indeed, to our knowledge the only explicit accounts of equivariant coherent sheaf theory in this generality are \cite[Sec. 6]{Ras20} and \cite[Sec. 4]{Gai14}. However, our study of convolution requires a more comprehensive treatment of the different functorialities of coherent (and ind-coherent) sheaves than is provided by these references. 

For example, convolution in $\Coh^{G_\cO}(\hR_{G,N})$ involves the (derived) pullback of (bounded complexes of) coherent sheaves along certain morphisms of infinite Tor-dimension. We need to know that under applicable hypotheses such pullback commutes with pushforward, \mbox{!-pullback}, and sheaf Hom. In \cite[Sec. 6]{Ras20}, however, only a subset of these compatibilities are established and only for finite Tor-dimension pullbacks. Note that this is a purely infinite-type issue, since in finite type coherent sheaves \emph{only} admit pullback along morphisms of finite Tor-dimension. 

Roughly speaking, the present paper covers the needed foundational results which can be treated without coherence hypotheses (in the sense of relating to coherent rings), and \cite{CWtm} covers those which depend on these. For reasons indicated below, we find that the best language in which to formalize these results is that of ind-geometric stacks, a class of objects which includes the quotient $\hR_{G,N}/G_\cO$. Before turning to sheaf theory, a more basic task of the present paper is thus to document the theory of ind-geometric stacks, synthesizing the theory of geometric stacks as treated in \cite[Ch. 9]{LurSAG} with that of (derived) ind-schemes as treated in \cite{GR14}. In the rest of the introduction we give more details first on the main technical themes developed in the paper, and then on the relationship of our work with the literature. 

\subsection{Technical overview}

In approaching $\Coh^{G_\cO}(\hR_{G,N})$, a first point is that the language of sheaves on stacks will ultimately serve us better than the language of equivariant sheaves.  We want the flexibility to discuss the stack $\hR_{G,N}/G_\cO$ independently of $\hR_{G,N}$ and $G_\cO$ for the same reasons we would want to discuss a smooth manifold independently of an atlas: many results are stated and proved most clearly in coordinate-free terms, and fixing coordinates can obscure symmetries. For example, the atlas $\hR_{G,N}$ obscures the convolution product on $\Coh^{G_\cO}(\hR_{G,N})$, whose existence becomes more manifest after identifying $\hR_{G,N}/G_\cO$ with the fiber product $N_\cO/G_\cO \times_{N_\cK/G_\cK} N_\cO/G_\cO$ of stacks. 

While one can develop a version of ind-coherent sheaf theory on an arbitrary stack (or prestack), it does not appear that outside finite type this can be done in a way which gives the version relevant to applications (Remark \ref{rem:IndCohnaive}). Instead, the natural scope of the theory we develop is that of ind-geometric stacks. We use the term geometric stack in the sense of \cite[Sec. 9]{LurSAG}, i.e. a geometric stack is an fpqc stack (on the site of nonpositively graded commutative dg algebras, if our ground ring is a field of characteristic zero) with affine diagonal and which admits a flat cover by an affine scheme. 

We caution that the term geometric stack is used flexibly in the literature \cite{Sim96,TV08}, and other variants are more natural in other contexts. Algebraic (i.e. Artin) stacks with affine diagonal are examples of geometric stacks in the above sense, but we avoid this term since the geometric stacks of interest to us are not algebraic (for example, $BG$ is algebraic since $G$ is smooth, but $BG_\cO$ is merely geometric since $G_\cO$ is merely pro-smooth). 

An ind-geometric stack is a filtered colimit of geometric stacks along closed immersions, though following the treatment of ind-schemes in \cite{GR14} this must be interpreted with care in derived settings (where convergence issues arise). While the ind-geometric stacks of interest to us can be expressed as quotients of ind-schemes, this perspective is technically inconvenient --- it is better to systematically treat such objects as direct limits of quotients of schemes, rather than as quotients of direct limits of schemes. This makes for a more compact overall theory, as the separate treatment of non-affine schemes becomes redundant once geometric stacks are treated. More pointedly, this avoids the need for multiple (ultimately redundant) renormalization/anticompletion steps in the treatment of ind-coherent sheaves. 

After setting up the needed geometric prerequisites in Sections \ref{sec:reasgeomstacks} and \ref{sec:indgeom}, we turn to our main topic, coherent and ind-coherent sheaf theory on ind-geometric stacks, in Section \ref{sec:cohandindcohsheaves}. Our applications require this theory to be developed without Noetherian hypotheses, and in fact it is convenient to develop it as far as possible without coherence hypotheses. This is because the category of coherent rings has poor formal properties, such as failing to be closed under tensor products \cite[Sec. 7.3.13]{Gla89}. With this in mind, for an arbitrary geometric stack $X$ we write $\Coh(X) \subset \QCoh(X)$ for the subcategory of bounded almost perfect objects (i.e. bounded pseudocoherent complexes when $X$ is a classical scheme \cite{Ill71}). When $X$ is ind-geometric but reasonable (i.e. it is a colimit along almost finitely presented closed immersions), $\Coh(X)$ is then defined as a colimit over reasonable geometric substacks. 

Our development of ind-coherent sheaf theory in this generality is based on the theory of anticomplete t-structures from \cite[Sec. C.5.5]{LurSAG}. That is, when $X$ is an arbitrary geometric stack we define $\IndCoh(X)$ as the (left) anticompletion of the category $\QCoh(X)$. This definition characterizes $\IndCoh(X)$ by a bounded, colimit-preserving functor $\IndCoh(X) \to \QCoh(X)$ satisfying a universal property. When $X$ is ind-geometric, $\IndCoh(X)$ is then defined as a colimit over closed geometric substacks. 

We caution that, while we have kept the terminology of \cite{Gai13a}, this notion of $\IndCoh(X)$ does not always coincide with the ind-completion of $\Coh(X)$ outside of finite type, and indeed this ind-completion can be poorly behaved in general. By Proposition \ref{prop:IndCohisIndofCoh} the two notions coincide when $X$ is coherent in the following sense. If $X$ is geometric, it is coherent if (1) there is a flat cover $\Spec A \to X$ with $A$ a coherent ring, and (2) the abelian category $\QCoh(X)^\heartsuit$ is compactly generated (i.e. there are ``enough'' coherent sheaves in the abelian sense). By \cite[Prop. 9.5.2.3]{LurSAG} and Lemma \ref{lem:affcgen} the second condition follows in particular if $X$ is admissible, meaning it admits an affine morphism to a geometric stack with a Noetherian flat cover. If $X$ is ind-geometric, it is coherent if it is reasonable and every reasonable geometric substack is coherent. 

A further caution is that $\IndCoh$ does not satisfy descent with respect to the kinds of quotients we are interested in. For example, though in \cite{CW2} we often write $\IndCoh^{G_\cO}(\hR_{G,N})$ for $\IndCoh(\hR_{G,N}/G_\cO)$, this category is not equivalent to the category $\IndCoh_{naive}^{G_\cO}(\hR_{G,N})$  of $G_\cO$-equivariant objects of $\IndCoh(\hR_{G,N})$. Indeed, this is already the case if we replace $\hR_{G,N}$ with a point, since $\IndCoh_{naive}^{G_\cO}(\pt) \cong \QCoh(BG_\cO)$ is not anticomplete (concretely, the structure sheaf is not compact in $\QCoh(BG_\cO)$). This is to be contrasted with the finite type situation, where e.g. $\IndCoh(BG) \cong \QCoh(BG)$ since algebraic groups are of finite cohomological dimension in characteristic zero. While we can recover $\IndCoh(\hR_{G,N}/G_\cO)$ from $\IndCoh_{naive}^{G_\cO}(\hR_{G,N})$ by anticompletion, this is needlessly circuitous: defining $\IndCoh(\hR_{G,N})$ already requires anticompleting $\QCoh$ of the closed subschemes of $\hR_{G,N}$ (and then taking a colimit), so one may as well directly construct $\IndCoh(\hR_{G,N}/G_\cO)$ by anticompleting $\QCoh$ of the closed substacks of $\hR_{G,N}/G_\cO$ (and then taking a colimit). 

Sections \ref{sec:upper!} and \ref{sec:sheafHom} study $!$-pullback and sheaf Hom for ind-coherent sheaves on ind-geometric stacks, in particular establishing their continuity properties and compatibility with pushforward and flat $*$-pullback. While ind-coherent sheaves do not have an internal tensor product in general, they do admit external tensor products. This is sufficient to have a useful notion of ind-coherent sheaf Hom, which satisfies various compatibilities with the usual quasi-coherent notion. We emphasize that even though coherent sheaves are our primary interest, it is necessary to introduce ind-coherent sheaves in order to define these adjoint functorialities. These are in turn necessary to study the duality properties of the convolution product on $\Coh^{G_\cO}(\hR_{G,N})$. 

\subsection{Relations to existing literature}\label{sec:literature}

Let us briefly survey the literature on which we build most directly. As already mentioned, our most immediate antecedent is \cite[Sec. 6]{Ras20}, which develops the ind-coherent sheaf theory of reasonable (dg) ind-schemes and renormalizable prestacks (flat quotients of ind-finite-type ind-schemes). The quotient $\hR_{G,N}/G_\cO$ is not renormalizable in the strict sense ($\hR_{G,N}$ is not ind-finite-type), but is close enough that the treatment of \cite[Sec. 6]{Ras20} largely applies. However, the additional functorialities needed for \cite{CW2} are more efficiently developed in the ind-geometric formalism (for reasons already alluded to, such as the absence of multiple renormalization steps). 

Looking further back, we build on the ind-coherent sheaf theory of (locally almost of) finite-type prestacks developed in \cite{Gai13a,GR17}, whose language we largely follow. These in turn build on the treatment of quasi-coherent sheaf theory in derived algebraic geometry in \cite{TV08} or \cite{LurSAG}. Some discussion of ind-coherent sheaves on infinite-type schemes appears in \cite[Sec. 2]{Gai13a} and of equivariant ind-coherent sheaves on placid ind-schemes in \cite[Sec. 4]{Gai14}, though the parts relevant to us are subsumed by \cite[Sec. 6]{Ras20} (in the language of \cite[Sec. 4]{Gai14}, it is ind-coherent *-sheaves which are of direct interest to us, rather than the dual theory of !-sheaves). 

The theory of ind-coherent sheaves on classical Noetherian schemes is considered from a different point of view (and with different terminology) in \cite{Kra05}, and our use of anticompletions implicitly prioritizes this point of view in generalizing to the non-Noetherian setting. That is, when $X$ is a classical geometric stack our $\IndCoh(X)$ is the enhanced homotopy category of injective complexes in $\QCoh(X)^\heartsuit$, and may not be compactly generated when $X$ is not locally Noetherian. 

In derived algebraic geometry the theory of ind-schemes is complicated by distinction between convergent and arbitrary stacks (i.e. stacks on the site of all commutative dg algebras or of those with bounded cohomology). These issues are treated in \cite{GR14}, with issues more specifically related to the derived notion of reasonableness treated in \cite{Ras20}. Our treatment of the corresponding ind-geometric issues largely follows this, though some deformation-theoretic tools are no longer available, and as needed we appeal to Tannaka duality (in the sense of \cite[Ch. 9]{LurSAG}) in their place. 

In the constructible setting, some themes parallel to ours are treated in \cite{BKV22}. A placid 1-geometric stack in the terminology of \cite{BKV22} is an example of classical geometric stack in the present terminology. We omit ``1-'' since unlike \cite{BKV22} we do not explicitly consider higher geometric stacks, while conversely \cite{BKV22} does not explicitly consider derived geometric stacks (constructible sheaves being insensitive to derived structures). Restricting to the placid setting is not needed for any constructions in the present paper, though in \cite{CWtm} we restrict to the related setting of tamely presented stacks. The notion of ind-geometric stack plays a similar role to the notion of placidly stratified stack in \cite{BKV22}, i.e. formalizing a class of stacks built from geometric ones. While we do not elaborate on this here, ind-geometric stacks also provide a convenient framework for treating the functorialities of ind-constructible sheaves in this generality (these being renormalizations of the ``locally ind-constructible'' presentable sheaf categories considered in \cite{BKV22}). 

Finally, let us mention the treatment of equivariant coherent sheaves on classical, coherent ind-schemes in \cite[Sec. 1]{VV10}, which is motivated by similar applications (the spherical DAHA being a particular example of a quantized K-theoretic Coulomb branch). On the other hand, while the categories of $G_\cO$-equivariant coherent sheaves on $\hR_{G,N}$ and its underlying classical ind-scheme $\hR_{G,N}^{\cl}$ have the same Grothendieck group, they are not themselves equivalent, and for the constructions of \cite{CW2} it is not sufficient to work with~$\hR_{G,N}^{\cl}$. 

\addtocontents{toc}{\SkipTocEntry}
\subsection*{Acknowledgements}
We are deeply grateful to Sam Raskin, Hiro Lee Tanaka, Aaron Mazel-Gee, and Chang-Yeon Chough for taking the time to discuss numerous technical issues that arose in the preparation of this paper and its companions \cite{CWtm, CW2}. S.C. was supported by NSERC Discovery Grant 2019-03961 and H. W. was supported by NSF grants DMS-1801969 and DMS-2143922. 

\section{Conventions}\label{sec:convnot}

We collect here our notational and terminological conventions. Our default references for categorical and geometric background are \cite{LurHTT,LurHA,LurSAG}, and we follow their conventions up to a few exceptions noted below. 

\begin{itemize}
\item We use the terms category and $\infty$-category interchangeably, and say ordinary category when we specifically mean a category in the traditional sense. We write $\Map_\catC(X,Y)$ for the mapping space between $X, Y \in \catC$, and regard ordinary categories as $\infty$-categories with discrete mapping spaces. 

\item We write $\CAlg(\catC)$ for the category of commutative algebra objects of a monoidal category $\catC$. We write $\CAlg$ for $\CAlg(\Sp^{cn})$, where $\Sp^{cn}$ is the category of connective spectra (this would be $\CAlg^{cn}$ in \cite{LurSAG}). 

\item We fix once and for all a Noetherian base $\kk \in \CAlg$.

\item We use cohomological indexing for t-structures. If $\catC$ has a t-structure $(\catC^{\leq 0}, \catC^{\geq 0})$ with heart $\catC^\heartsuit$ we write $\tau^{\leq n} : \catC \to \catC^{\leq n}$, $H^n: \catC \to \catC^\heartsuit$, etc., for the associated functors. In this notation, the condition that $\kk$ is Noetherian is the condition that $H^0(\kk)$ is an ordinary Noetherian ring and $H^n(\kk)$ is finitely generated over $H^0(\kk)$ for all $n < 0$. We use the terms left bounded and right bounded interchangeably with (cohomologically) bounded below and bounded above. 

\item We write $\tau_{\leq n} \catD$ for the subcategory of $n$-truncated objects in an $\infty$-category $\catD$. In particular, $\tau_{\leq 0} \CAlg$ is the ordinary category of ordinary commutative rings. Note the distinction between subscripts and superscripts in this and the previous convention, e.g. $\tau_{\leq n}(\catC^{\leq 0})$ and $\catC^{[-n, 0]}$ refer to the same subcategory of $\catC$. 

\item Given $A \in \CAlg$, we write $\Mod_A$ for the category of $A$-modules (i.e. $A$-module objects in the category of spectra). If $A$ is an ordinary ring this is the (enhanced) unbounded derived category of ordinary $A$-modules (i.e. of $\Mod_A^{\heartsuit}$).   

\item An $A$-module $M$ is coherent if it is bounded and almost perfect (i.e. $\tau^{\geq n}M$ is compact in $\Mod_A^{\geq n}$ for all $n$). If $A$ is coherent (i.e. $\clalg{A}$ is a coherent ordinary ring and $H^n(A)$ is a finitely presented $\clalg{A}$-module for all $n$), then $M$ is coherent if and only if it is bounded and $H^n(M)$ is a finitely presented $\clalg{A}$-module for all $n$. We write $\Coh_A \subset \Mod_A$ for the full subcategory of coherent modules. 

\item Given $A \in \CAlg$, we write $\CAlg_A := \CAlg_{A/} \cong \CAlg(\Mod_A^{\leq 0})$, and write $\CAlgvarleqinfty{A} := \cup_n \CAlgvarleqn{A}$ for the subcategory of truncated $A$-algebras (i.e. $n$-truncated for some~$n$). If $A$ is an ordinary ring containing $\Q$, then $\CAlg_A$ is equivalently the (enhanced homotopy) category of nonpositively graded commutative dg $A$-algebras, or of simplicial/animated commutative $A$-algebras.  

\item  We implicitly fix two universes and associated category sizes: small and large. We write $\Cathatinfty$ for the $\infty$-category of large $\infty$-categories (in \cite{LurHTT} this would be $\LurieCathatinfty$, and $\Cathatinfty$ would be its subcategory of small $\infty$-categories). We write $\PrL \subset \Cathatinfty$ for the subcategory of presentable $\infty$-categories and left adjoints, and $\PrSt \subset \PrL$ for the further subcategory of presentable stable $\infty$-categories. 

\item Given categories $\catC$ and $\catD$, we write both $\Fun(\catC, \catD)$ and $\catD^\catC$ for the category of functors from $\catC$ to $\catD$. 

\item All limit or colimit diagrams are implicitly small unless otherwise stated. Thus in ``let $X \cong \colim X_\al$ be a filtered colimit'' the indexing diagram is assumed to be small. By extension, $\Ind(\catC)$ will refer to the category freely generated by $\catC$ under small filtered colimits even if $\catC$ is large (as in \cite[Def. 21.1.2.5]{LurSAG}). 

\item If $\catC$ admits filtered colimits, a functor $F: \catC \to \catD$ is continuous if it preserves them. Suppose further that $\catC$, $\catD$ are presentable, stable, and equipped with t-structures that are compatible with filtered colimits, and that $F$ is exact. Then $F$ is almost continuous if its restriction to $\catC^{\geq n}$ is continuous for all $n$ (equivalently, for $n = 0$). 

\item A prestack (implicitly over $\Spec \kk$) is a functor from $\CAlgk$ to the category of (possibly large) spaces. We write $\PreStkk$ for the category of prestacks, and $\PreStkkconv$, $\PreStkkleqn$ for the variants with $\CAlgkleqinfty$, $\CAlgkleqn$  in place of $\CAlgk$. We write $\Spec: \CAlgk \to \PreStkk$ for the Yoneda embedding.  

\item A stack is a prestack which is a sheaf for the fpqc topology \cite[Prop. B.6.1.3]{LurSAG}. We write $\Stkk \subset \PreStkk$ for the category of stacks, and $\Stkkconv \subset \PreStkkconv$, $\Stkkleqn \subset \PreStkkleqn$ for its variants. (Note that $\CAlgkleqinfty$ does not admit arbitrary pushouts, but the use of \cite[Prop. A.3.2.1]{LurSAG} in defining the fpqc topology only requires closure under flat pushouts.)

\item If $\catC$ admits finite limits, we write $\Corr(\catC)$ for the $\infty$-category of correspondences in $\catC$ (e.g. \cite[Def. 3.3]{Bar13}, \cite[Sec. 7.1.2]{GR17}). This has the same objects as $\catC$, but a morphism from $X$ to $Z$ in $\Corr(\catC)$ is a diagram $X \xleftarrow{h} Y \xrightarrow{f} Z$ in $\catC$. 

\item Let $\vert$ and $\horiz$ be classes of morphisms in $\catC$ which contain all isomorphisms and are stable under composition, and under base change along each other. Suppose also that $\catC' \subset \catC$ is a full subcategory such that $Y \in \catC'$ whenever $h: Y \to X$ is in $\horiz$ and $X \in \catC'$. Then we write $\Corr(\catC')_{\vert,\horiz}$ for the 1-full subcategory of $\Corr(\catC)$ which only includes correspondences $X \xleftarrow{h} Y \xrightarrow{f} Z$ such that $h \in \horiz$, $f \in \vert$, and $X, Z \in \catC'$ (hence $Y \in \catC'$). Note that $\catC'$ need not be closed under arbitrary pullbacks. The subcategory $\Corr(\catC')_{\vert,\isom} \subset \Corr(\catC')_{\vert,\horiz}$ which only includes correspondences in which $h$ is an isomorphism is equivalent to the 1-full subcategory $\catC'_{\vert} \subset \catC'$, likewise for the subcategory where $f$ is an isomorphism and $\catC'^{\op}_{\horiz} \subset \catC'^{\op}$. 

\item We presume our constructions and results remain valid if we replace $\CAlg$ with the category $\CAlg^\Delta$ of simplicial/animated commutative rings. We work with $\CAlg$ mostly to make some references easier to pinpoint. However, we do appeal to Tannaka duality in the proofs of Propositions~\ref{prop:immisaffforgeo} and \ref{prop:gstkconvergent}, and we do not explicitly know how to adapt this to the simplicial setting. On the other hand, any derived prestack has an underlying spectral prestack, and by definition these share the same category of quasi-coherent sheaves. Since our focus is on sheaves, it is in this sense more natural to work in the spectral setting. This distinction is also irrelevant to our intended applications, in which our base is $\C$ and we have $\CAlg_\C \cong \CAlg_\C^\Delta$. 
\end{itemize}

\section{Geometric stacks}\label{sec:reasgeomstacks}

In this section we recall the basic theory of geometric stacks, and collect some results we will need in later sections. Most of these are extensions of documented results about Artin stacks, but which seem to lack references in the needed generality. In particular we collect the basic properties of the main classes of morphisms we will need: morphisms of finite Tor-dimension or finite cohomological dimension, proper morphisms, and closed immersions. Two important technical results are that geometric stacks are convergent (Proposition \ref{prop:gstkconvergent}) and are compact in the category of convergent 1-stacks (Proposition \ref{prop:oneStkkconvcompactness}). These will play an important role in the next section, but require a different approach than that used for Artin stacks in \cite[Prop. 3.4.4.9]{GR17}.

\subsection{Definitions}
Recall our convention that a stack refers to a functor $\CAlgk \to \Spc$ satisfying fpqc descent (here $\kk$ is our fixed Noetherian base), and that the category of stacks is denoted by $\Stkk$. Our terminology follows \cite[Ch. 9]{LurSAG} (up to the presence of the base $\kk$). 

\begin{Definition}\label{def:1}
	A stack $X$ is geometric if its diagonal $X \to X \times X$ is affine and there exists faithfully flat morphism $\Spec B \to X$ in $\Stkk$. A morphism $X \to Y$ in $\Stkk$ is geometric if for any morphism $\Spec A \to Y$, the fiber product $X \times_Y \Spec A$ is geometric. We write $\GStkk \subset \Stkk$ for the full subcategory of geometric stacks. 
\end{Definition}

Note here that products are taken in $\Stkk$, hence are implicitly over $\Spec \kk$. 
Also note that affineness of $X \to X \times X$ implies that any morphism $\Spec B \to X$ is affine. In particular, (faithful) flatness of such a morphism is defined by asking that its base change to any affine scheme is such. More generally, a morphism $X \to Y$ in $\GStkk$ is (faithfully) flat if its composition with any faithfully flat $\Spec A \to X$ is (faithfully) flat. A faithfully flat morphism of geometric stacks will also be called a flat cover.

\begin{Proposition}\label{prop:gstkprops}
	 Geometric morphisms are stable under composition and base change in~$\Stkk$. If $f: X \to Y$ is a morphism in $\Stkk$, then $f$ is geometric if $X$ and $Y$ are, and $X$ is geometric if $f$ and $Y$ are. In particular, $\GStkk$  is closed under fiber products in $\Stkk$. 
\end{Proposition}
\begin{proof}
Over the sphere spectrum $S$ this is \cite[Prop. 9.3.1.2, Ex. 9.3.1.10]{LurSAG}. Note that by \cite[Cor. 5.1.6.12]{LurHTT}, \cite[Prop. A.3.3.1]{LurSAG} we have an equivalence $\Stkk \cong \Stk_{/\Spec \kk}$, where $\Stk := \Stk_S$. It then suffices to show $\GStkk$ is the preimage of $\GStk$ under the forgetful functor to $\Stk$. 
Clearly $X \in \Stkk$ has a flat cover $\Spec A \to X$ in $\Stkk$ if and only if its image in $\Stk$ does. 
%over $\Spec \kk$ if and only if it does over $\Spec S$. 
Now note that if $f: Y \to Z$, $g: Z \to W$ are morphisms in $\Stk$ and $g$ is affine, then $f$ is affine if and only if $g \circ f$ is (since any $\Spec B \to Z$ factors through $Z \times_W \Spec B \to Z$). The morphism $X \times_{\Spec \kk} X \to X \times_{\Spec S} X$ is affine since it is a base change of $\Spec \kk \to \Spec (\kk \ot_S \kk)$, hence $X \to X \times_{\Spec \kk} X$ is affine if and only if $X \to X \times_{\Spec S} X$ is. 
\end{proof}

Recall that an ordinary commutative ring $A$ is coherent if every finitely generated ideal is finitely presented. More generally, $A \in \CAlgk$ is coherent (resp. Noetherian) if $\clalg{A}$ is coherent (resp. Noetherian) and $H^n(A)$ is a finitely presented $\clalg{A}$-module for all $n$. 

\begin{Definition}\label{def:coherentGstack}
	A geometric stack $X$ is locally coherent (resp. locally Noetherian) if there exists a flat cover $\Spec A \to X$ such that $A$ is coherent (resp. Noetherian). It coherent if it is locally coherent and $\QCoh(X)^\heartsuit$ is compactly generated. 
\end{Definition}

A locally Noetherian geometric stack is coherent by \cite[Prop. 9.5.2.3]{LurSAG}. 

\subsection{Truncated and classical geometric stacks}
The definition of ind-geometric stack will involve the following class of geometric stacks. 

\begin{Definition}\label{def:truncgstack}
	A geometric stack $X$ is $n$-truncated if it admits a flat cover $\Spec A \to X$ such that $A$ is $n$-truncated. We say $X$ is classical if it is zero-truncated, and truncated if it is $n$-truncated for some $n$. We denote by $\GStkkplus \subset \GStkk$ the full subcategory of truncated geometric stacks.  
\end{Definition}

Alternatively, note that the restriction functor $(-)_{\leq n}: \PreStkk \to \PreStkkleqn$ takes $\Stkk$ to $\Stkkleqn$ \cite[Prop. A.3.3.1]{LurSAG}. Write $i_{\leq n}: \Stkkleqn \to \Stkk$ for the left adjoint of this restriction and $\tau_{\leq n}: \Stkk \to \Stkk$ for their composition. Then if $X$ is geometric, $\tau_{\leq n} X$ is an $n$-truncated geometric stack called the $n$-truncation of $X$, and $X$ is $n$-truncated if and only if the natural map $\tau_{\leq n} X \to X$ is an isomorphism \cite[Cor. 9.1.6.8, Prop. 9.1.6.9]{LurSAG}. 

In particular, if $\kk$ is a field and $X \in \Stkkleqzero$ is an ordinary algebraic variety, then $i_{\leq 0} X$ is a zero-truncated geometric stack. The functor $i_{\leq 0}: \Stkkleqzero \to \Stkk$ embeds the category of ordinary varieties (more generally, ordinary quasi-compact, semi-separated schemes, or quasi-compact Artin stacks with affine diagonal) as a full subcategory of $\GStkkplus$, and by default we will identify these categories with their images in $\GStkkplus$. 

Our terminology follows \cite[Def. 9.1.6.2]{LurSAG}, but we caution  that what we call $n$-truncatedness is called $n$-coconnectedness in \cite{GR17}. We also note that this use of the symbol $\tau_{\leq n}$ and of the term truncation are different from their usual meaning in terms of truncatedness of mapping spaces, but in practice no ambiguity will arise (and this abuse has the feature that $\tau_{\leq n} \Spec A \cong \Spec \tau_{\leq n} A$). 

\subsection{Coherent sheaves}
Recall that for any stack $X$, the category $\QCoh(X)$ of quasi-coherent sheaves on $X$ is the limit of the categories $\Mod_A$ over all maps $\Spec A \to X$. If $X$ is geometric $\QCoh(X)$ is presentable and is equivalent to the corresponding limit over the Cech nerve of any flat cover \cite[Prop. 9.1.3.1]{LurSAG}. We say an $A$-module $M$ is coherent if it is bounded and almost perfect (i.e. $\tau^{\geq n} M$ is compact in $\Mod_A^{\geq n}$ for all $n$). 

\begin{Definition}
	If $X \in \GStkkplus$, then $\cF \in \QCoh(X)$ is coherent if~$f^*(\cF)$ is a coherent $A$-module for some (equivalently, any) flat cover $\Spec A \to X$. We write $\Coh(X) \subset \QCoh(X)$ for the full subcategory of coherent sheaves. 
\end{Definition}

While the above definition makes sense when $X$ is not truncated, without additional hypotheses the resulting category $\Coh(X)$ may be degenerate (for example, it may contain no nonzero objects). It will be convenient to exclude such degenerate cases from our discussion, though our treatment of coherent sheaves on ind-geometric stacks will include well-behaved non-truncated geometric stacks within its scope. 

If $X$ is locally coherent, the standard t-structure on $\QCoh(X)$ restricts to one on $\Coh(X)$. If $X$ is coherent, it follows from \cite[Prop. 9.1.5.1]{LurSAG} that specifically $\QCoh(X)^{\heartsuit}$ is compactly generated by $\Coh(X)^\heartsuit$. If $X$ is zero-truncated but not locally coherent, our use of the term coherent sheaf corresponds to the notion of bounded pseudocoherent complex in \cite{Ill71}, see \cite[Rem. 2.8.4.6]{LurSAG}. 

Coherent sheaves have the following basic functoriality. A morphism $f: X \to Y$ in $\GStkk$ is of Tor-dimension~$\leq n$ if $f^*(\QCoh(Y)^{\geq 0}) \subset \QCoh(Y)^{\geq n}$, and is of finite Tor-dimension if it is of Tor-dimension~$\leq n$ for some $n$. We have the following variant of standard results. 

\begin{Proposition}\label{prop:ftdprops}
	Morphisms of Tor-dimension $\leq n$ are stable under base change in $\GStkk$, and morphisms of finite Tor-dimension are also stable under composition. A morphism $f: X \to Y$ of geometric stacks is of Tor-dimension $\leq n$ if and only if its base change along any given flat cover $h: \Spec A \to Y$ is. In this case $f^*: \QCoh(Y) \to \QCoh(X)$ takes $\Coh(Y)$ to $\Coh(X)$. 
\end{Proposition}
\begin{proof}
	Stability under composition is immediate. Flat locality on the target follows since if $h'$ and $f'$ are defined by base change, then $h'^*$ is conservative, t-exact, and satisfies $h'^* f^* \cong f'^* h^*$. If $h: \Spec A \to Y$ is arbitrary, then $h'$ is affine, hence $h'_*$ is conservative, t-exact, and satisfies $f^*h_* \cong h'_* f'^*$ \cite[Prop. 9.1.5.7]{LurSAG}. Stability under base change along affine morphisms follows, and arbitrary base change now follows by composing an arbitrary $Y' \to Y$ with a flat cover $\Spec B \to Y'$. The last claim then follows since almost perfect modules are stable under extension of scalars. 
\end{proof}

\subsection{Pushforward and base change}\label{sec:geompropsofmors}
Given a morphism $f: X \to Y$ in $\Stkk$, the pushforward $f_*: \QCoh(X) \to \QCoh(Y)$ is defined as the right adjoint of $f^*$. In general $f_*$ is poorly behaved, but in the geometric case we have the following results. A different proof in a related context is sketched in \cite[Prop. 5.5.6]{LurDAG}, \cite[Lem. A.1.3]{HLP23}. The template used below will be used again in proving Propositions \ref{prop:upper!almostcont} and \ref{prop:cHomalmostcont}, and follows a parallel result about $!$-pullback in \cite[Prop. 6.4.1.4]{LurSAG}. Recall that $f_*$ being almost continuous means its restriction to $\QCoh(X)^{\geq n}$ is continuous for all $n$. 

\begin{Proposition}\label{prop:lower*almostcont}
If  $f: X \to Y$ is a morphism of geometric stacks, then $f_*: \QCoh(X) \to \QCoh(Y)$ is almost continuous. 
\end{Proposition} 

\begin{Proposition}\label{prop:lower*ftdbasechange}
	Let the following be a Cartesian diagram of geometric stacks.
	\begin{equation*}
		\begin{tikzpicture}
			[baseline=(current  bounding  box.center),thick,>=\arrtip]
			\node (a) at (0,0) {$X'$};
			\node (b) at (3,0) {$Y'$};
			\node (c) at (0,-1.5) {$X$};
			\node (d) at (3,-1.5) {$Y$};
			\draw[->] (a) to node[above] {$f' $} (b);
			\draw[->] (b) to node[right] {$h $} (d);
			\draw[->] (a) to node[left] {$h' $}(c);
			\draw[->] (c) to node[above] {$f $} (d);
		\end{tikzpicture}
	\end{equation*}
	If $h$ is of finite Tor-dimension, then the Beck-Chevalley map $h^* f_*(\cG) \to f'_* h'^*(\cG)$ is an isomorphism for all $\cG \in \QCoh(X)^+$.  
\end{Proposition}

\begin{Lemma}\label{lem:lower*almostcontaffine}
	Proposition \ref{prop:lower*almostcont} is true when $Y$ is affine. 
\end{Lemma}
\begin{proof}
	In this case $\QCoh(Y)$ is compactly generated by perfect sheaves. If $\cF \in \QCoh(Y)$ is perfect so is $f^*(\cF)$, hence the claim follows by applying \cite[Prop. 5.5.7.2]{LurHTT} to $\tau^{\geq n} f^*: \QCoh(Y) \to \QCoh(X)^{\geq n}$. 
\end{proof}

\begin{Lemma}\label{lem:lower*affineftdbasechange}
	Proposition \ref{prop:lower*ftdbasechange} is true when $Y$ and $Y'$ are affine. 
\end{Lemma}
\begin{proof}
	Let $Y \cong \Spec A$ and $Y' \cong \Spec B$. Since $h$ is affine $h_*$ conservative, hence it suffices to show $ h_* h^* f_*(\cG) \to h_* f'_* h'^*(\cG)$ is an isomorphism. Rewriting the second term using $h_* f'_* \cong f_* h'_*$, one sees this is the specialization of the Beck-Chevalley map $\theta_M: f_*(\cG) \ot M \to f_*(\cG \ot f^*(M))$ in the case $M = B$.
	
	Write $\catC$ for the full subcategory of  $M \in \Mod_A$ such that $\theta_{M}$ is an isomorphism. The assignment $M \mapsto \theta_{M}$ extends to a functor $\Mod_A \to \Mod_A^{\Delta^1}$, which is exact since the source and target of $\theta_M$ are exact in $M$. It follows that $\catC$ is a stable subcategory closed under retracts, as isomorphisms form such a subcategory of $\Mod_A^{\Delta^1}$. Clearly $A \in \catC$, hence $\catC$ contains all perfect $A$-modules. If $M$ is of Tor-dimension $\leq n$, then we can write it as a filtered colimit $M \cong \colim_\al M_\al$ of perfect $A$-modules of Tor-dimension $\leq n$ \cite[Prop. 9.6.7.1]{LurSAG}. The claim now follows since tensoring is continuous, since the $\cG \ot f^*(M_\al)$ are uniformly bounded below, and since $f_*$ is almost continuous by Lemma~\ref{lem:lower*almostcontaffine}. 
\end{proof}

\begin{Lemma}\label{lem:lower*flatcoverbasechange}
	Proposition \ref{prop:lower*ftdbasechange} is true when $Y'$ is affine and $h$ is faithfully flat. 
\end{Lemma}
\begin{proof}
	Let $Y_\bul$ denote the Cech nerve of $h$ (so $Y_0 = Y'$) and $f_k: X_k \to Y_k$ the base change of $f$. 
	Given a morphism $p: i \to j$ in $\Delta_s$, let $h_p: Y_j \to Y_i$ denote the associated map and $h'_p: X_j \to X_i$ its base change. The categories $\QCoh(X_k)^{\geq 0}$ and $\QCoh(Y_k)^{\geq 0}$, together with the functors $h^*_p$, $h'^*_p$, and $\tau^{\geq 0} \circ f^*_k$, form a diagram $\Delta^1 \times \Delta_s \to \Cathatinfty$. By Lemma \ref{lem:lower*affineftdbasechange} the Beck-Chevalley transformation $h_p^* f_{i*} \to f_{j*} h'^*_p$ restricts to an isomorphism of functors $\QCoh(X_i)^{\geq 0} \to \QCoh(Y_j)^{\geq 0}$ for any $p$. Since $h$ is faithfully flat we have $\QCoh(X)^{\geq 0} \cong \lim_{\Delta_s} \QCoh(X_i)^{\geq 0}$ and $\QCoh(Y)^{\geq 0} \cong \lim_{\Delta_s} \QCoh(Y_i)^{\geq 0}$, and the claim follows from \cite[Cor. 4.7.5.18]{LurHA}. 
\end{proof}

\begin{proof}[Proof of Proposition \ref{prop:lower*almostcont}]
	Let $\cG \cong \colim \cG_\al$ be a filtered colimit in $\QCoh(X)^{\geq 0}$, let $h: X' \cong \Spec A \to Y$ be a flat cover, and define $f': X' \to Y'$, $h': X' \to X$ by base change. Since $h^*$ is continuous and conservative it suffices to show $\colim h^* f_*(\cG_\al) \to h^* f_*(\cG)$ is an isomorphism. By Lemma~\ref{lem:lower*flatcoverbasechange} and left t-exactness of $f_*$ this is equivalent to $\colim  f'_* h'^*(\cG_\al) \to f'_* h'^*(\cG)$ being an isomorphism. Since $h'^*$ is  t-exact this follows from Lemma \ref{lem:lower*almostcontaffine}. 
\end{proof}

\begin{proof}[Proof of Proposition \ref{prop:lower*ftdbasechange}]
	Let $\phi: U \cong \Spec A \to Y$ and $\theta: U' \cong \Spec A' \to U \times_Y Y'$ be flat covers. We obtain a diagram 
	\begin{equation*}%\label{eq:up!compcube}
		\begin{tikzpicture}[baseline=(current  bounding  box.center),thick,>=\arrtip]
			\newcommand*{\ha}{1.5}; \newcommand*{\hb}{1.5}; \newcommand*{\hc}{1.5};
\newcommand*{\va}{-.9}; \newcommand*{\vb}{-.9}; \newcommand*{\vc}{-.9}; 
			\node (ab) at (\ha,0) {$Z'$};
			\node (ad) at (\ha+\hb+\hc,0) {$U'$};
			\node (ba) at (0,\va) {$X'$};
			\node (bc) at (\ha+\hb,\va) {$Y'$};
			\node (cb) at (\ha,\va+\vb) {$Z$};
			\node (cd) at (\ha+\hb+\hc,\va+\vb) {$U$};
			\node (da) at (0,\va+\vb+\vc) {$X$};
			\node (dc) at (\ha+\hb,\va+\vb+\vc) {$Y$};
			%\draw[->] (ab) to node[above] {$\phi' $} (ad);
			\draw[->] (ab) to node[above] {$g' $} (ad);
			%\draw[->] (ab) to node[above left] {$\theta' $} (ba);
			\draw[->] (ab) to node[above left, pos=.25] {$\psi' $} (ba);
			%\draw[->] (ab) to node[left,pos=.8] {$\psi' $} (cb);
			\draw[->] (ab) to node[right,pos=.2] {$\xi' $} (cb);
			\draw[->] (ad) to node[below right] {$\psi $} (bc);
			\draw[->] (ad) to node[right] {$\xi $} (cd);
			%\draw[->] (ba) to node[left] {$h' $} (da);
			\draw[->] (ba) to node[left] {$ $} (da);
			\draw[->] (cb) to node[above,pos=.25] {$g $} (cd);
			%\draw[->] (cb) to node[above left] {$g' $} (da);
			\draw[->] (cb) to node[above left, pos=.25] {$\phi' $} (da);
			\draw[->] (cd) to node[below right] {$\phi $} (dc);
			\draw[->] (da) to node[above,pos=.75] {$ $} (dc);
			
			\draw[-,line width=6pt,draw=white] (ba) to  (bc);
			%\draw[->] (ba) to node[above,pos=.75] {$f' $} (bc);
			\draw[->] (ba) to node[above,pos=.75] {$ $} (bc);
			\draw[-,line width=6pt,draw=white] (bc) to  (dc);
			\draw[->] (bc) to node[right,pos=.2] {$ $} (dc);
		\end{tikzpicture}
	\end{equation*}
	in which all but the left and right faces are Cartesian. Note that $\psi$ is faithfully flat and $\xi$ is of finite Tor-dimension, since they are the compositions of $\theta$ with the base changes of $\phi$ and $h$, respectively. Since $\psi^*$ is conservative, it suffices to show the top left arrow in
	\begin{equation*}
		\begin{tikzpicture}
			[baseline=(current  bounding  box.center),thick,>=\arrtip]
			\newcommand*{\ha}{3.5}; \newcommand*{\hb}{3.5};
			\newcommand*{\va}{-1.5};
			\node (aa) at (0,0) {$\psi^* h^* f_*(\cG)$};
			\node (ab) at (\ha,0) {$\psi^* f'_* h'^*(\cG)$};
			\node (ac) at (\ha+\hb,0) {$g'_*\psi'^* h'^*(\cG)$};
			\node (ba) at (0,\va) {$\xi^* \phi^*f_*(\cG)$};
			\node (bb) at (\ha,\va) {$\xi^* g_* \phi'^*(\cG)$};
			\node (bc) at (\ha+\hb,\va) {$g'_* \xi'^* \phi'^*(\cG)$};
			\draw[->] (aa) to node[above] {$ $} (ab);
			\draw[->] (ab) to node[above] {$  $} (ac);
			\draw[->] (ba) to node[above] {$ $} (bb);
			\draw[->] (bb) to node[above] {$ $} (bc);
			\draw[->] (aa) to node[below,rotate=90] {$\sim $} (ba);
			%\draw[->] (ab) to node[right] {$h' $} (bb);
			\draw[->] (ac) to node[below,rotate=90] {$\sim $} (bc);
		\end{tikzpicture} 
	\end{equation*}
	is an isomorphism. This follows since the bottom left and top right arrows are isomorphisms by Lemma \ref{lem:lower*flatcoverbasechange}, and the bottom right is by Lemma \ref{lem:lower*affineftdbasechange}. 
\end{proof}

Following \cite{GR17}, we can encode the coherence properties of base change isomorphisms using correspondence categories. Let $\catC$ be a category with finite limits, and let $\vert$ and $\horiz$ be classes of morphisms which are stable under composition and under base change along each other. Recall that we have an associated category $\Corr(\catC)_{\vert,\horiz}$ whose morphisms are correspondences $X \xleftarrow{h} Y \xrightarrow{f} Z$ such that $f$ is in $\vert$ and $h$ is in $\horiz$. 

We say a functor $\Phi: \catC^{\op} \to \Catinfty$ is left $\vert$-adjointable if for every Cartesian square
\begin{equation*}
	\begin{tikzpicture}
		[baseline=(current  bounding  box.center),thick,>=\arrtip]
		\node (a) at (0,0) {$X'$};
		\node (b) at (3,0) {$Y'$};
		\node (c) at (0,-1.5) {$X$};
		\node (d) at (3,-1.5) {$Y$};
		\draw[->] (a) to node[above] {$f' $} (b);
		\draw[->] (b) to node[right] {$h $} (d);
		\draw[->] (a) to node[left] {$h' $}(c);
		\draw[->] (c) to node[above] {$f $} (d);
	\end{tikzpicture}
\end{equation*}
with $f \in \vert$, the associated 
Beck-Chevalley transformation 
$ \Phi(f)^L \Phi(h) \to \Phi(h') \Phi(f')^L $ 
is an isomorphism. When $\horiz = all$ contains all morphisms, we have the following universal property, where $\CORR(\catC)$ and $\CATinfty$ are the $(\infty,2)$-categorical enhancements of $\Corr(\catC)$ and $\Catinfty$. 

\begin{Proposition}\label{prop:Corrextension}
	Restriction along $\catC^\op \subset \CORR(\catC)_{\vert,all}$ induces a monomorphism $$\Map_{\Catinftytwo}(\CORR(\catC)_{\vert,all}, \CATinfty) \to \Map_{\Catinftytwo}(\catC^\op, \CATinfty)$$ with essential image the left $\vert$-adjointable functors. In particular, any left $\vert$-adjointable functor $\Phi: \catC^\op \to \Catinfty$ extends canonically to a functor $\Corr(\catC)_{\vert,all} \to \Catinfty$ whose value on a correspondence $X \xleftarrow{h} Y \xrightarrow{f} Z$ is $\Phi(f)^L \Phi(h): \Phi(X) \to \Phi(Z)$. 
\end{Proposition}

This property was identified in \cite[Thm. 7.3.2.2]{GR17}. We have followed the formulation of \cite[Thm. 2.2.7]{EH20}, which closely follows the proof due to \cite[Thm. 4.2.6]{Mac20}. We will refer to this proposition also to invoke any of its variants in which left is replaced with right and/or the roles of $\vert$ and $\horiz$ are swapped \cite[Var. 2.2.9-2.2.10]{EH20}. 

Returning to the case at hand, let $ftd$ denote the class of morphisms of finite Tor-dimension in $\GStkk$, so that morphisms in $\Corr(\GStkk)_{all,ftd}$ are correspondences $X \xleftarrow{h} Y \xrightarrow{f} Z$ such that $h$ is of finite Tor-dimension. By Propositions \ref{prop:Corrextension} and \ref{prop:lower*ftdbasechange} the assignment $X \mapsto \QCoh(X)^+$ extends to a functor 
\begin{equation}\label{eq:CorrQCoh+}
	\QCoh^+: \Corr(\GStkk)_{all,ftd} \to \Cathatinfty
\end{equation}
whose value on the above correspondence is $f_* h^*: \QCoh(X)^+ \to \QCoh(Z)^+$. 

\subsection{Cohomological dimension}
We have better control of $f_*$ in the following case.  Recall that a morphism $f: X \to Y$ in $\GStkk$ is of cohomological dimension $\leq n$ if $f_*(\QCoh(X)^{\leq 0}) \subset \QCoh(Y)^{\leq n}$, and is of finite cohomological dimension if it is of cohomological dimension $\leq n$ for some $n$. We caution that morphisms of infinite cohomological dimension are ubiquitous in our motivating context. For example, if $G$ is a complex reductive group, $BG_\cO \to \Spec \C$ is of infinite cohomological dimension. 

\begin{Proposition}\label{prop:fcdprops}
Morphisms of finite cohomological dimension are stable under composition and base change in $\GStkk$. A morphism $f: X \to Y$ of geometric stacks is of finite cohomological dimension if and only if its base change along any given flat cover $\Spec A \to Y$ is. In this case $f_*: \QCoh(X) \to \QCoh(Y)$ is continuous, and for any Cartesian square
	\begin{equation}\label{eq:properbasechange}
	\begin{tikzpicture}
		[baseline=(current  bounding  box.center),thick,>=\arrtip]
		\node (a) at (0,0) {$X'$};
		\node (b) at (3,0) {$Y'$};
		\node (c) at (0,-1.5) {$X$};
		\node (d) at (3,-1.5) {$Y$};
		\draw[->] (a) to node[above] {$f' $} (b);
		\draw[->] (b) to node[right] {$h $} (d);
		\draw[->] (a) to node[left] {$h' $}(c);
		\draw[->] (c) to node[above] {$f $} (d);
	\end{tikzpicture}
\end{equation}
in $\GStkk$ the Beck-Chevalley transformation $h^*f_*(\cF) \to f'_* h'^*(\cF)$ is an isomorphism for all $\cF \in \QCoh(X)$. 
\end{Proposition}
\begin{proof}
Stability under composition is immediate, while stability under base change and flat locality on the target follow from \cite[Prop. A.1.9]{HLP23} (whose proof applies to geometric stacks, not just Artin stacks with affine diagonal). The remaining properties then follow from \cite[Prop. 9.1.5.3, Prop. 9.1.5.7]{LurSAG} or \cite[Prop. A.1.5]{HLP23}. 
\end{proof}

Let $\Corr(\GStkk)_{fcd,all}$ denote the 1-full subcategory of $\Corr(\GStkk)$ which only includes correspondences $X \xleftarrow{h} Y \xrightarrow{f} Z$ such that $f$ is of finite cohomological dimension. By Propositions~\ref{prop:Corrextension} and \ref{prop:fcdprops} the assignment $X \mapsto \QCoh(X)$ extends to a functor 
\begin{equation}\label{eq:CorrQCohGStk}
	\QCoh: \Corr(\GStkk)_{fcd,all} \to \Cathatinfty
\end{equation}
whose value on the above correspondence is $f_* h^*: \QCoh(X) \to \QCoh(Z)$. 

\subsection{Proper and almost finitely presented morphisms}\label{sec:propafp}
Recall following \cite[Def. 17.4.1.1]{LurSAG} that a morphism $f: X \to Y$ of stacks is (locally) almost of finite presentation if, for any $n$ and any filtered colimit $A \cong \colim A_\al$ in $\CAlgkleqn$, the canonical map
\begin{equation}\label{eq:functorafp}
	\colim X(A_\al) \to X(A) \times_{Y(A)} \colim Y(A_\al)
\end{equation} 
is an isomorphism (we omit the word locally by default, as all morphisms we consider will be quasi-compact or effectively so). 

\begin{Proposition}[{\cite[Rem. 17.4.1.3, Rem. 17.4.1.5]{LurSAG}}]\label{prop:afp2of3prop} 
	Almost finitely presented morphisms are stable under composition and base change in $\Stkk$. 
	If $f$ and $g$ are composable morphisms in $\Stkk$ such that $g \circ f$ and $g$ are almost of finite presentation, then so is $f$. 
\end{Proposition}

We mostly consider this condition together with properness. We say $f: X \to Y$ is proper if for any $\Spec A \to Y$, the fiber product $X \times_Y \Spec A$ is proper over $\Spec A$ in the sense of \cite[Def. 5.1.2.1]{LurSAG}. In particular, this requires that $X \times_Y \Spec A$ be a quasi-compact, separated (spectral) algebraic space. 

\begin{Proposition}\label{prop:propprops}
Proper morphisms are stable under composition and base change in $\Stkk$. If $f$ and $g$ are morphisms in $\Stkk$ such that $g \circ f$ and $g$ are proper, then so is $f$. Proper morphisms of geometric stacks are of finite cohomological dimension. If $f$ is a proper, almost finitely presented morphism of truncated geometric stacks, then $f_*$ takes $\Coh(X)$ to $\Coh(Y)$.
\end{Proposition}

\begin{proof}
Stablity under base change is immediate, and the claims about composition follow from \cite[Prop. 5.1.4.1, Prop. 6.3.2.2]{LurSAG}. 
By Proposition \ref{prop:fcdprops} finiteness of cohomological dimension can be checked after base change along a flat cover $h: \Spec A \to Y$, where it follows from \cite[Prop. 2.5.4.4, Prop. 3.2.3.1]{LurSAG}. If $f$ is almost of finite presentation and $f'$ is its base change along $h$, then $f'_*$ preserves coherence by \cite[Thm. 5.6.0.2]{LurSAG}. Then so does $h^* f_*$ by Proposition~\ref{prop:fcdprops}, hence so does~$f_*$ by definition. 
\end{proof}

Let $\Corr(\GStkkplus)_{prop,ftd}$ denote the 1-full subcategory of $\Corr(\GStkk)$ which only includes correspondences $X \xleftarrow{h} Y \xrightarrow{f} Z$ such that $h$ is of finite Tor-dimension, $f$ is proper and almost finitely presented, and $X$ and $Z$ are truncated (hence so is $Y$). By Propositions~\ref{prop:ftdprops} and \ref{prop:propprops}, we can restrict the domain and values of either (\ref{eq:CorrQCoh+}) or (\ref{eq:CorrQCohGStk}) to obtain a functor 
\begin{equation}\label{eq:CorrCohGStkprop}
	\Coh: \Corr(\GStkkplus)_{prop,ftd} \to \Catinfty
\end{equation}
whose value on the above correspondence is $f_* h^*: \Coh(X) \to \Coh(Z)$. 

\subsection{Closed immersions}
A morphism $f: X \to Y$ in $\Stkk$ is a closed immersion if for any $\Spec A \to Y$, the morphism 
%$(X \times_Y \Spec A)^{\cl} \to (\Spec A)^{\cl}$ 
$\tau_{\leq 0}(X \times_Y \Spec A) \to \tau_{\leq 0} \Spec A$ 
is a closed immersion of ordinary affine schemes. The following statement follows immediately from its classical counterpart. 

\begin{Proposition}\label{prop:cl2of3prop}
	Closed immersions are stable under composition and base change in $\Stkk$. If $f$ and $g$ are composable morphisms in $\Stkk$ such that $g \circ f$ and $g$ are closed immersions, then so is $f$. 
\end{Proposition}

A closed immersion need not be affine (for example, the inclusion of a closed subscheme of a classical ind-scheme is typically not affine in the derived sense), but in $\GStkk$ we have the following. 

\begin{Proposition}\label{prop:immisaffforgeo}
	Closed immersions between geometric stacks are affine. 
\end{Proposition}
\begin{proof}
	Let $f: X \to Y$ be a closed immersion between geometric stacks, and let $\Spec A \to Y$ be arbitrary. The base change $f': X' \to \Spec A$ of $f$ factors through a map $f'': X' \to \Spec B$, where $B := f'_*(\cO_{X'}) \in \CAlg_A$. We claim $f''$ is an isomorphism. By Tannaka duality \cite[Thm. 9.2.0.2]{LurSAG} it suffices to show $f''_*: \QCoh(X') \to \Mod_B$ is an equivalence. By Barr-Beck-Lurie \cite[Thm. 4.7.4.5]{LurHA} it further suffices to show $f'_*: \QCoh(X') \to \Mod_A$ is conservative and preserves small colimits.
	
	Since $f$ is a closed immersion $\tau_{\leq 0} f'': \tau_{\leq 0} X' \to \tau_{\leq 0} \Spec B$ is an isomorphism, hence $f''_*$ restricts to an equivalence $\QCoh(X')^\heartsuit \cong \Mod_B^\heartsuit$. It follows that the restriction of $f'_*$ to  $\QCoh(X')^\heartsuit$ is conservative, continuous, and factors through $\Mod_A^{\heartsuit}$. It follows as in \cite[Lem. A.1.6]{HLP23} that $f'_*$ is t-exact. 
	
	Suppose now that $f'_*(\cF) \cong 0$ for some $\cF \in \QCoh(X')$. It follows from the preceding paragraph that $\cH^n (\cF) \cong 0$ for all $n \in \Z$. Since $\QCoh(X')$ is t-complete we then have $\cF \cong 0$, hence $f'_*$ is conservative. Similarly, suppose $\cF \cong \colim \cF_\al$ is a filtered colimit in $\QCoh(X')$, and let $\phi: \colim f'_*(\cF_\al) \to f'_*(\cF)$ denote the natural map. Since the t-structures on $\QCoh(X')$ and $\Mod_A$ are compatible with filtered colimits, it follows from the preceding paragraph that $\cH^n(\phi)$ is an isomorphism for all $n \in \Z$. That $\phi$ is an isomorphism, hence that~$f'_*$ is continuous, now follows from the t-completeness of $\Mod_A$. But then $f'_*$ preserves all small colimits since it is exact \cite[Prop. 5.5.1.9]{LurHTT}.	
\end{proof}

\begin{Proposition}\label{prop:classicaliso}
	Let $f: X \to Y_\al$, $h: Y' \to Y_\al$, and $i: Y_\al \to Y_\be$ be morphisms of geometric stacks. If $i$ is a closed immersion, the map $\tau_{\leq 0}(X \times_{Y_\al} Y') \to \tau_{\leq 0}(X \times_{Y_\be} Y')$ is an isomorphism. 
\end{Proposition}
\begin{proof}
	Let $X'_\al := X \times_{Y_\al} Y'$, $X'_\be := X \times_{Y_\be} Y'$, and suppose first that $X$ and $Y'$ are affine. Fixing a flat cover $\Spec B_\be \to Y_\be$ we obtain flat covers $\Spec A'_\al \to X'_\al$, $\Spec A'_\be \to X'_\be$ by base change. It suffices to show the induced morphism $\tau_{\leq 0} \Spec A'_\al \to \tau_{\leq 0} \Spec A'_\be$ is an isomorphism \cite[Lem. 6.2.3.16]{LurHTT}. If $\Spec A \to X$, $\Spec B' \to Y'$, and $\Spec B_\al \to Y_\al$ are also obtained by base change from $\Spec B_\be \to Y_\be$, then $\tau_{\leq 0} \Spec B_\al \to \tau_{\leq 0} \Spec B_\be$ is a closed immersion since $i$ is. But then $\tau_{\leq 0} \Spec A'_\al \to \tau_{\leq 0} \Spec A'_\be$ is an isomorphism since $A'_\al \cong A \ot_{B_\al} B'$ and $A'_\be \cong A \ot_{B_\be} B'$. 
	
	In the general case, fix flat covers $\Spec C \to X$ and $\Spec D' \to Y'$, and let $\Spec C'_\al := \Spec C \times_{Y_\al} \Spec D'$ and $\Spec C'_\be := \Spec C \times_{Y_\be} \Spec D'$. By the preceding paragraph $\tau_{\leq 0} \Spec C'_\al \to \tau_{\leq 0} \Spec C'_\be$ is an isomorphism. But since it is the base change of $\tau_{\leq 0} X'_\al \to \tau_{\leq 0} X'_\be$ along the flat cover $\tau_{\leq 0} \Spec C'_\be \to \tau_{\leq 0} X'_\be$, the claim follows. 
\end{proof}

\subsection{Convergence}
Recall that $\PreStkkconv$ denotes the category of functors $\CAlgkleqinfty \to \Spc$, and $\Stkkconv \subset \PreStkkconv$ the subcategory of functors satisfying fpqc descent. The restriction functor $(-)_{< \infty}: \PreStkk \to \PreStkkconv$ has a fully faithful right adjoint, by which we generally regard $\PreStkkconv$ as a subcategory of $\PreStkk$. One says a prestack is convergent (or nilcomplete) if it is contained in this subcategory. Explicitly, $X \in \PreStkk$ is convergent if for all $A \in \CAlgk$ the natural morphism $X(A) \to \lim X(\tau_{\leq n} A)$ is an isomorphism. We have an induced notion of convergent stack, which is unambiguous in the following sense. 

\begin{Lemma}\label{lem:convstacks}
	The inclusion $\PreStkkconv \subset \PreStkk$ identifies $\Stkkconv$ with $\Stkk \cap \PreStkkconv$. 
\end{Lemma}
\begin{proof}
	That $\Stkk \cap \PreStkkconv \subset \Stkkconv$ follows from  the definition of the fpqc topology and from \cite[Prop. A.3.3.1]{LurSAG} (closure of $\CAlgkleqinfty$ under pushouts along flat morphisms is sufficient to apply this to $\PreStkkconv$). Now suppose $X \in \Stkkconv \subset \PreStkkconv \subset \PreStkk$. Note that $\tau_{\leq n}: \Sp^{\cn} \to \tau_{\leq n} \Sp^{\cn}$ preserves finite products since it preserves colimits and $\Sp^{\cn}$ is additive, hence $\tau_{\leq n}: \CAlgk \to \CAlgkleqn$ preserves finite products by \cite[Cor. 3.2.2.5]{LurHA}. 
	Then if $A \cong \prod_{i=1}^n A_i$ is a finite product in $\CAlgk$, we have
	$ X(A) \cong \lim_n X(\tau_{\leq n} A) \cong \lim_{n,i} X(\tau_{\leq n} A_i) \cong \lim_{i} X(A_i). $
	Similarly, if $A \to A^0$ in $\CAlgk$ is faithfully flat and $A^\bullet$ its Cech nerve, then  
	$ X(A) \cong \lim_n X(\tau_{\leq n} A) \cong \lim_{n,i} X(\tau_{\leq n} A^i) \cong \lim_{i} X(A^i). $
	Thus $X \in \Stkk$ by \cite[Prop. A.3.3.1]{LurSAG}. 
\end{proof}

We now have the following result in the geometric case. 

\begin{Proposition}\label{prop:gstkconvergent}
Geometric stacks are convergent. 
\end{Proposition}
\begin{proof}
	Suppose $X \in \GStkk$ and $A \in \CAlgk$. Recall again that using \cite[Cor. 5.1.6.12]{LurHTT}, \cite[Prop. A.3.3.1]{LurSAG} we have an equivalence $\Stkk \cong \Stk_{/\Spec \kk}$, where $\Stk := \Stk_S$ for $S$ the sphere spectrum. We define convergent objects in $\Stk$ the same way as in $\Stkk$. By \cite[Lem. 5.5.5.12]{LurHTT} we have that $\Map_{\Stkk}(\Spec A, X)$ is the fiber of the map $\Map_{\Stk}(\Spec A, X) \to \Map_{\Stk}(\Spec A, \Spec \kk)$ induced by composition with $X \to \Spec \kk$ over the point corresponding to $\Spec A \to \Spec \kk$. Since the same holds for each $\tau_{\leq n}A$, and since $\Spec \kk$ is convergent in $\Stk$, it follows that $X$ is convergent in $\Stkk$ if its image in $\Stk$ is convergent. 
	
	Consider the natural diagram
	\begin{equation}\label{eq:convergencediagram}
		\begin{tikzpicture}
			[baseline=(current  bounding  box.center),thick,>=\arrtip]
			\node (a) at (0,0) {$\Map_{\Stk}(\Spec A, X)$};
			\node (b) at (7,0) {$\Map_{\CAlg(\Cathatinfty)}( \QCoh(X)^{\leq 0}, \Mod_A^{\leq 0})$};
			\node (c) at (0,-1.5) {$\lim \Map_{\Stk}(\Spec \tau_{\leq n} A, X)$};
			\node (d) at (7,-1.5) {$\lim \Map_{\CAlg(\Cathatinfty)}( \QCoh(X)^{\leq 0}, \Mod_{\tau_{\leq n}A}^{\leq 0})$.};
			\draw[->] (a) to node[above] {$ $} (b);
			\draw[->] (b) to node[right] {$ $} (d);
			\draw[->] (a) to node[left] {$ $}(c);
			\draw[->] (c) to node[above] {$ $} (d);
		\end{tikzpicture}
	\end{equation}
	By \cite[Lem. 5.4.6]{LurDAGXII} the top map is a monomorphism if the corresponding map from $\Map_{\Stk}(\Spec A, X)$ to $\Map_{\CAlg(\Cathatinfty)}( \QCoh(X), \Mod_A)$ is a monomorphism, which in turn follows from Tannaka duality \cite[Prop. 9.3.0.3]{LurSAG}. Since this also holds replacing $A$ with $\tau_{\leq n}A$, and since monomorphisms are stable under limits, it follows similarly that the bottom map is a monomorphism. The right map is an isomorphism, since we have
	$$ \Mod_A^{\leq 0} \cong \lim_m \Mod_A^{[m,0]} \cong \lim_{m,n} \Mod_{\tau_{\leq n} A}^{[m,0]} \cong \lim_n \Mod_{\tau_{\leq n} A}^{\leq 0}, $$
	in $\CAlg(\Cathatinfty)$. Here the first and third equivalences follow from the relevant t-structures being left complete and $\CAlg(\Cathatinfty) \to \Cathatinfty$ preserving limits \cite[Cor. 3.2.2.5]{LurHA}, the second from $\Mod_{A}^{[m,0]} \to \Mod_{\tau_{\leq n} A}^{[m,0]}$ being an equivalence for $m \leq - n$. It follows that the left map in (\ref{eq:convergencediagram}) is a monomorphism, so we must show it is essentially surjective. 
	
	Since the horizontal maps are monomorphisms, it suffices to show the right isomorphism restricts to an essential surjection between their essential images. By Tannaka duality \cite[Prop. 9.3.0.3]{LurSAG} the essential image of the bottom map consists of systems $\{G_n: \QCoh(X)^{\leq 0} \to \Mod_{\tau_{\leq n}A}^{\leq 0}\}$ of symmetric monoidal functors which preserve small colimits and flat objects, similarly for the top map. Let $G: \QCoh(X)^{\leq 0} \to \Mod_{A}^{\leq 0}$ be the functor associated to some such system $\{G_n\}$ under the right isomorphism. Then $G$ preserves small colimits 
	since $\PrL$ is closed under limits in $\Cathatinfty$ \cite[Prop. 5.5.3.13]{LurHTT}. Moreover, if $G_n(\cF) \cong \tau_{\leq n}A \otimes_A  G(\cF)$ is flat over $\tau_{\leq n}A$ for all $n$, then since $\tau_{\leq n}G(\cF) \cong \tau_{\leq n}(\tau_{\leq n}A \otimes_A  G(\cF))$ by \cite[Prop. 7.1.3.15]{LurHA} it follows from the definition of flatness that $G(\cF)$ is flat over $A$. Thus $G$ is in the image of the top map in (\ref{eq:convergencediagram}), establishing the claim. 
\end{proof}

Since $\CAlgkleqn$ is closed under products and targets of flat morphisms in $\CAlgkleqinfty$, the restriction functor $(-)_{\leq n}: \PreStkkconv \to \PreStkkleqn$ takes $\Stkkconv$ to $\Stkkleqn$ \cite[Prop. A.3.3.1]{LurSAG}. We write $\wh{i}_{\leq n}: \Stkkleqn \to \Stkkconv$ for the resulting left adjoint and $\wh{\tau}_{\leq n}: \Stkkconv \to \Stkkconv$ for their composition. 

The functors $(-)_{\leq n}$ induce an equivalence $\Stkkconv \cong \lim \Stkkleqn$ in $\Cathatinfty$ \cite[Ex. A.7.1.6]{LurSAG}. Moreover, for any $X \in \Stkkconv$ the natural map $\colim \wh{\tau}_{\leq n} X \to X$ is an isomorphism. This is a special case of Lemma~\ref{lem:colimpres}, but more explicitly if $\wh{\tau}_{\leq n}^{\,pre}: \PreStkkconv  \to \PreStkkconv$ denotes the composition of $(-)_{\leq n}: \PreStkkconv \to \PreStkkleqn$ and its left adjoint, then $\colim \wh{\tau}_{\leq n}^{\,pre} X \to X$ is an isomorphism since $\CAlgkleqinfty = \cup_n \CAlgkleqn$ and since colimits in $\PreStkkconv$ are computed objectwise. But $\wh{\tau}_{\leq n}$ is the sheafification of $\wh{\tau}_{\leq n}^{\,pre}$, so $\colim \wh{\tau}_{\leq n} X \to X$ is an isomorphism since sheafification is continuous. In particular, if $X$ is a geometric stack Proposition \ref{prop:gstkconvergent} implies $\tau_{\leq n} X \cong \wh{\tau}_{\leq n} X$, hence we obtain the following corollary. 

\begin{Proposition}\label{prop:gstacktruncapprox}
	For any geometric stack $X$ we have $X \cong \colim \tau_{\leq n} X$ in $\Stkkconv$.
\end{Proposition}

Now let $\oneStkkconv \subset \Stkkconv$ denote the full subcategory consisting of $X$ such that $X(A)$ is an $(n+1)$-truncated space for all $A \in \CAlgkleqn$. Then Proposition \ref{prop:gstkconvergent} is refined by the following result, the first half of which is a variant of \cite[Cor. I.2.4.3.4]{GR17}, the second of \cite[Lem. 1.3.6]{GR14}. 

\begin{Proposition}\label{prop:oneStkkconvcompactness}
	Geometric stacks are objects of $\oneStkkconv$. Moreover, $\oneStkkconv$ is closed under filtered colimits in $\PreStkkconv$, and truncated geometric stacks are compact as objects of $\oneStkkconv$. 
\end{Proposition}
\begin{proof}
	Let $X$ be a geometric stack. To show $X \in \oneStkkconv$, it suffices to show that $X_{\leq n}$ belongs to $\tau_{\leq n+1} \Stkkleqn$, the category of $(n+1)$-truncated objects of $\Stkkleqn$, since for $A \in \CAlgkleqn$ we have $X(A) \cong \Map_{\Stkkleqn}((\Spec A)_{\leq n}, X_{\leq n})$. 

	Let $\Spec B \to X$ be a flat cover, so that $(\Spec B)_{\leq n} \to X_{\leq n}$ is a flat cover in $\Stkkleqn$. If $(\Spec B)_{\leq n}^\bullet$ is its Cech nerve, we have $X_{\leq n} \cong \colim (\Spec B)_{\leq n}^m$ in $\Stkkleqn$. Equivalently, $X_{\leq n}$ is the sheafification of the same colimit taken in $\PreStkkleqn$. Since sheafification is left exact it preserves $(n+1)$-truncated objects \cite[Prop. 5.5.6.16]{LurHTT}, hence it suffices to show the colimit taken in $\PreStkkleqn$ is $(n+1)$-truncated. For any $m$ an object of $\PreStkkleqn$ is $m$-truncated if its values on $\CAlgkleqn$ are $m$-truncated spaces \cite[Rem. 5.5.8.26]{LurHTT}. Each $(\Spec B)_{\leq n}^m$ is affine since $X$ is geometric, hence is then $n$-truncated in $\PreStkkleqn$ since $\CAlgkleqn$ is an $(n+1)$-category. The claim then follows since values of colimits in $\PreStkkleqn$ are computed objectwise, and since the geometric realization of a groupoid of $n$-truncated spaces is $(n+1)$-truncated. 
	
	We now claim $X_{\leq n}$ is compact in $\tau_{\leq n+1} \Stkkleqn$. By the argument of \cite[Lem. 1.3.3]{GR14} we have that $\tau_{\leq n+1} \Stkkleqn$ is closed under filtered colimits in $\PreStkkleqn$ (noting that by \cite[Prop. 5.5.6.16]{LurHTT} we have $\tau_{\leq n+1} \Stkkleqn = \Stkkleqn \cap \tau_{\leq n+1} \PreStkkleqn$ since $\Stkkleqn$ is closed under limits in $\PreStkkleqn$). It follows that each $(\Spec B)_{\leq n}^m$ is compact in $\tau_{\leq n+1} \Stkkleqn$ since it is so in $\PreStkkleqn$. Moreover, $X_{\leq n}$ is their colimit in $\tau_{\leq n+1} \Stkkleqn$ since it is their colimit in $\Stkkleqn$. But $\tau_{\leq n+1} \Stkkleqn$ is an $(n+2)$-category, so by the proof of \cite[Lem. 1.3.3.10]{LurHA} $X_{\leq n}$ is also the colimit in $\tau_{\leq n+1} \Stkkleqn$ of the $(\Spec B)_{\leq n}^m$ over the finite subdiagam $\Delta_{s, \leq n+2}^{op} \subset \Delta_{s}^{op}$. It follows that $X_{\leq n}$ is compact in $\tau_{\leq n+1} \Stkkleqn$ \cite[Cor. 5.3.4.15]{LurHTT}. 
	
	Next note that \cite[Prop. A.3.3.1]{LurSAG} also implies $\Stkkconv$ is the full subcategory of $X \in \PreStkkconv$ such that $X_{\leq n} \in \Stkkleqn$ for all $n$. But then $\tau_{\leq n+1} \Stkkleqn = \Stkkleqn \cap \tau_{\leq n+1} \PreStkkleqn$ implies $\oneStkkconv$ is the full subcategory of $X \in \PreStkkconv$ such that $X_{\leq n} \in \tau_{\leq n+1} \Stkkleqn$ for all $n$. Closure of $\oneStkkconv$ under filtered colimits then follows since $(-)_{\leq n}: \PreStkkconv \to \PreStkkleqn$ is continuous and since as recalled above $\tau_{\leq n+1} \Stkkleqn$ is closed under filtered colimits in $\PreStkkleqn$. 
	
	Now suppose $X$ is an $n$-truncated geometric stack, let $Y \cong \colim Y_\al$ be a filtered colimit in $\oneStkkconv$, and consider the following diagram. 
\begin{equation*}
\begin{tikzpicture}
	[baseline=(current  bounding  box.center),thick,>=\arrtip]
	\newcommand*{\ha}{6.5}; \newcommand*{\hb}{4.5};
	\newcommand*{\va}{-1.5};
	\newcommand*{\vb}{-1.5};
	\node (aa) at (0,0) {$\colim \Map_{\oneStkkconv}(X,  Y_\al)$};
	\node (ab) at (\ha,0) {$\colim \Map_{\tau_{\leq n+1} \Stkkleqn}(X_{\leq n}, (Y_\al)_{\leq n})$};
	\node (bb) at (\ha,\va) {$\Map_{\tau_{\leq n+1} \Stkkleqn}(X_{\leq n}, \colim (Y_\al)_{\leq n})$};
	\node (ca) at (0,\va+\vb) {$\Map_{\oneStkkconv}(X, Y)$};
	\node (cb) at (\ha,\va+\vb) {$\Map_{\tau_{\leq n+1} \Stkkleqn}(X_{\leq n}, Y_{\leq n})$};
	\draw[->] (aa) to node[above] {$ $} (ab);
	\draw[->] (ca) to node[above] {$ $} (cb);
	\draw[->] (ab) to node[above] {$ $} (bb);
	\draw[->] (bb) to node[left] {$ $} (cb);
	\draw[->] (aa) to node[right] {$ $} (ca);
\end{tikzpicture}
\end{equation*}
The bottom right map is an isomorphism since by the preceding paragraph $(-)_{\leq n}$ restricts to a continuous functor $\oneStkkconv \to \tau_{\leq n+1} \Stkkleqn$. Since $X$ is $n$-truncated it is the image of $X_{\leq n}$ under the left adjoint $\wh{i}_{\leq n}: \Stkkleqn \to \Stkkconv$ of $(-)_{\leq n}$, hence the horizontal maps are isomorphisms. The top right map is an isomorphism since  $X_{\leq n}\in \tau_{\leq n+1} \Stkkleqn$ is compact, hence the left map is an isomorphism and $X$ is compact in $\oneStkkconv$. 
\end{proof}

\section{Ind-geometric stacks}\label{sec:indgeom}

We now define our main objects of study and establish their basic properties. Key technical results include the consistency of the notions of truncated and reasonable geometric substack (Proposition~\ref{prop:igprestermsarereas}), and the closure of ind-geometric stacks under fiber products (Proposition~\ref{prop:indgeomfiberprods}) and ind-closed filtered colimits (Proposition~\ref{prop:indclosedclosure}). 

\subsection{Definitions}

Recall from Proposition \ref{prop:gstkconvergent} that $\GStkk$ is contained in the full subcategory $\Stkkconv \subset \Stkk$ of convergent stacks. This subcategory plays a central role in our discussion because colimits in $\Stkkconv$ are typically more natural than colimits in $\Stkk$. For example, any $\Spec A$ is the colimit of its truncations $\Spec \tau_{\leq n} A$ in $\Stkkconv$, but not in $\Stkk$ unless $A$ is itself truncated. In particular, the inclusion $\Stkkconv \subset \Stkk$ does not preserve colimits in general. 

\begin{Definition}\label{def:indgeomstack}
An \emph{ind-geometric stack} is a convergent stack $X$ which admits an expression $X \cong \colim_{\al} X_\al$ as a filtered colimit in $\Stkkconv$ of truncated geometric stacks along closed immersions. We call such an expression an \emph{ind-geometric presentation} of $X$. 
\end{Definition}

This is the natural extension of the derived notion of ind-scheme introduced in \cite{GR14}, see Proposition \ref{prop:indschdefcomp}.  We write $\indGStkk \subset \Stkkconv$ for the full subcategory of ind-geometric stacks. Any geometric stack $X$ is ind-geometric since $X \cong \colim \tau_{\leq n} X$ in $\Stkkconv$ by Proposition~\ref{prop:gstacktruncapprox}.  

\begin{Definition}\label{def:reasindgeomstack}
A \emph{reasonable presentation} is an ind-geometric presentation $X \cong \colim_{\al} X_\al$ in which the structure maps are almost finitely presented. An ind-geometric stack is \emph{reasonable} if it admits a reasonable presentation, and \emph{coherent} if it admits a reasonable presentation whose terms are coherent geometric stacks. 
\end{Definition}

This is the natural extension of the derived notion of reasonableness introduced in \cite{Ras20}, which in turn extends the classical notion of reasonableness in \cite{BD}. 
We write $\indGStkkreas \subset \indGStkk$ (resp. $\indGStkkcoh \subset \indGStkk$) for the full subcategory of reasonable (resp. coherent) ind-geometric stacks. A non-truncated geometric stack need not be reasonable as an ind-geometric stack, though we do have the following result. 

\begin{Proposition}\label{prop:loccohisreasonable}
	If a geometric stack $X$ is locally coherent (resp. coherent), then it is reasonable (resp. coherent) as an ind-geometric stack. 
\end{Proposition}
\begin{proof}
	Let $\Spec A \to X$ be a flat cover such that $A$ is coherent. Then $\Spec \tau_{\leq n} A \to \Spec \tau_{\leq n+1 } A$ is almost of finite presentation for all $n$ \cite[Cor. 5.2.2.2]{LurSAG}. This is the base change of the affine morphism $\tau_{\leq n} X \to \tau_{n+1} X$ along a flat cover, and it follows from \cite[Prop. 4.1.4.3]{LurSAG} that $\tau_{\leq n} X \to \tau_{n+1} X$ is also almost finitely presented. Thus $X$ is reasonable by Proposition \ref{prop:gstacktruncapprox}. If additionally $\QCoh(X)^\heartsuit$ is compactly generated, hence $X$ is coherent as a geometric stack, then $X$ is coherent as an ind-geometric stack since $\QCoh(X)^\heartsuit \cong \QCoh(\tau_{\leq n}X)^\heartsuit$ for all $n$. 
\end{proof}

The following illustrates that in infinite type unreasonable geometric stacks are plentiful. 

\begin{Example}\label{ex:selfintAinfty}
Consider the self-intersection $X = \{0\} \times_{\A^\infty} \{0\}$ of the origin in $\A^\infty$. We have $X \cong \Spec A$, where $A$ is a symmetric algebra on countably many generators in degree $\shortminus 1$. In particular, $H^n (A)$ is not finite dimensional for any $n < 0$, the map $\tau_{\leq n} X \to \tau_{\leq n+1} X$ is not almost finitely presented for any $n \geq 0$, and $X$ is not reasonable as an ind-geometric stack. A related pathology is that the category $\Mod_{A}$ contains no nonzero bounded almost perfect objects, hence $\Coh(X)$ could only be sensibly interpreted as the zero category. It is to avoid such degenerate behavior that our treatment of $\Coh(-)$ in later sections is confined to the reasonable setting. 
\end{Example}

\subsection{Truncated and classical ind-geometric stacks}
The following extends Definition~\ref{def:truncgstack}. 

\begin{Definition}
An ind-geometric stack $X$ is $n$-truncated if it admits an ind-geometric presentation $X \cong \colim X_\al$ in which each $X_\al$ is an $n$-truncated geometric stack. We say $X$ is classical if it is zero-truncated. 
\end{Definition}

A typical example of a classical ind-geometric stack is the following. Suppose $X \cong \colim X_\al$ is a presentation of a classical ind-scheme as a filtered colimit of quasi-compact, semi-separated ordinary schemes (regarded as objects of $\Stkkconv$) along closed immersions, and that $G$ is a classical affine group scheme acting on $X$. For each $\al$ the induced map $X_\al \times G \to X$ factors through some $X_\be$. The closure $X'_\al \subset X_\be$ of its image is a $G$-invariant closed subscheme of $X$. We obtain a presentation $X \cong \colim X'_\al$ by closed $G$-invariant subschemes, and it follows that the quotient $X/G$ (in $\Stkkconv$) is a classical ind-geometric stack with ind-geometric presentation $X/G \cong \colim X'_\al/G$. Note that the quotients $X'_\al/G$ taken in $\Stkk$ are again geometric \cite[Prop. 9.3.1.3]{LurSAG}, hence convergent (Proposition \ref{prop:gstkconvergent}), hence they coincide with the quotients $X'_\al/G$ taken in $\Stkkconv$. 

Recall that $(-)_{\leq n}: \Stkkconv \to \Stkkleqn$ identifies the category of $n$-truncated geometric stacks with a full subcategory of $\Stkkleqn$, the inverse equivalence being given by the left adjoint $\wh{i}_{\leq n}: \Stkkleqn \to \Stkkconv$. Note that $(-)_{\leq n}$ takes filtered colimits in $\oneStkkconv$ to filtered colimits in $\tau_{\leq n+1} \Stkkleqn$, since it restricts from a continuous functor $\PreStkkconv \to \PreStkkleqn$, and since as in Proposition \ref{prop:oneStkkconvcompactness} and its proof $\oneStkkconv \subset \PreStkkconv$ and $\tau_{\leq n+1} \Stkkleqn \subset \PreStkkleqn$ are closed under filtered colimits. Since $\wh{i}_{\leq n}$ is continuous it follows that $(-)_{\leq n}$ identifies the category of $n$-truncated ind-geometric stacks with the obvious subcategory of $\Stkkleqn$.  Again letting $\wh{\tau}_{\leq n}: \Stkkconv \to \Stkkconv$ denote the composition of $(-)_{\leq n}$ and $\wh{i}_{\leq n}$, the following result states in particular that Definition \ref{def:indgeomstack} is indeed the obvious extension of \cite[Def. 1.4.2]{GR14} from schemes to geometric stacks (it is stated differently so that reasonableness may be introduced more easily).  

\begin{Proposition}\label{prop:indschdefcomp}
	A convergent stack $X$ is ind-geometric if and only if  $\wh{\tau}_{\leq n} X$ is an $n$-truncated ind-geometric stack for all $n$. 
\end{Proposition}
\begin{proof}
	The only if direction follows since $\wh{\tau}_{\leq n}$ preserves closed immersions of geometric stacks, and since by the above discussion its restriction to $\oneStkkconv$ is continuous. The if direction follows from Proposition \ref{prop:indclosedclosure} below (which does not depend on the current result), since each $\wh{\tau}_{\leq n}X \to \wh{\tau}_{\leq n+1}X$ is classically an isomorphism, hence is an ind-closed immersion.
\end{proof}

\subsection{Geometric substacks} 
Truncatedness plays an essential role in our discussion due to the following variant of \cite[Lem. 1.3.6]{GR14} (though see Remark~\ref{rem:nontruncindgdef}).  The claim follows immediately from Proposition \ref{prop:oneStkkconvcompactness}, but would fail if $Y$ were not truncated: in this case $Y$ is not compact in $\oneStkkconv$, since e.g. $\id_Y$ does not factor through any truncation of $Y$. 

\begin{Proposition}\label{prop:factorthroughgeometric}
	Let $X \cong \colim X_\al$ be an ind-geometric presentation. Then for any truncated geometric stack $Y$, the natural map
	$$\colim \Map_{\Stkk}(Y, X_\al) \to \Map_{\Stkk}(Y, X)$$
	is an isomorphism. 
\end{Proposition}

To discuss ind-geometric stacks more intrinsically, without referring to particular ind-geometric presentations, the following notion is useful. 

\begin{Definition}\label{def:reassubstack}
	Let $X$ be an ind-geometric stack. 
	A \emph{truncated} (resp. \emph{reasonable}) \emph{geometric substack} of $X$ is a truncated geometric stack $X'$ equipped with a closed immersion $X' \to X$ (resp. an almost finitely presented closed immersion $X' \to X$).  
\end{Definition}

\begin{Proposition}\label{prop:igprestermsarereas}
Let $X \cong \colim X_\al$ be an ind-geometric (resp. reasonable) presentation. Then for all $\al$, the structure morphism $i_\al: X_\al \to X$ realizes $X_\al$ as a truncated (resp. reasonable) geometric substack of $X$. Any other truncated (resp. reasonable) geometric substack $X' \to X$ can be factored as $X' \xrightarrow{j_\al} X_\al \xrightarrow{i_\al} X$ for some $\al$, and in any such factorization $j_\al$ is a closed immersion (resp. almost finitely presented closed immersion), hence affine. 
\end{Proposition}

\begin{proof}
	To show $i_\al$ is a closed immersion, fix $\Spec A \to X$ and let $Z := X_\al \times_{X} \Spec A$. Since $\tau_{\leq 0} Z \cong \tau_{\leq 0}(X_\al \times_{X} \tau_{\leq 0} \Spec A)$ we may assume $A$ is classical.  By Proposition~\ref{prop:factorthroughgeometric} we can then factor $\Spec A \to X$ through some $X_{\al'}$, which we may assume satisfies $\al' \geq \al$. For $\al'' \geq \al'$ let $Z_{\al''} := X_\al \times_{X_{\al''}} \Spec A$. We have $Z \cong \colim_{\al'' \geq \al'} Z_{\al''}$ since filtered colimits in $\Stkkconv$ are left exact \cite[Ex. 7.3.4.7]{LurHTT}. Moreover, $\tau_{\leq 0} Z \cong \colim_{\al'' \geq \al'} \tau_{\leq 0} Z_{\al''}$ since by Proposition \ref{prop:oneStkkconvcompactness} and its the proof all terms are in $\oneStkkconv$ and $Y \mapsto \tau_{\leq 0} Y$ preserves filtered colimits in $\oneStkkconv$. To show $\tau_{\leq 0} Z \to \tau_{\leq 0} \Spec A$ is a closed immersion it then suffices to show $\tau_{\leq 0} Z_{\al'} \to \tau_{\leq 0} Z_{\al''}$ is an isomorphism for any $\al'' \geq \al'$, since $\tau_{\leq 0} Z_{\al'} \to \tau_{\leq 0} \Spec A$ is a closed immersion by hypothesis. But this follows from Proposition~\ref{prop:classicaliso}. 

	Now suppose the given presentation is reasonable, and let $A \cong \colim A_\be$ be a filtered colimit in $\CAlgkleqn$ for some $n$. By Proposition \ref{prop:factorthroughgeometric} we have $X(A_\be) \cong \colim_{\ga \geq \al} X_\ga(A_\be)$ for all~$\beta$, and likewise $X(A) \cong \colim_{\ga \geq \al} X_\ga(A)$. We then have
	$$ X_\al(A) \times_{X(A)} \colim_\be X(A_\be) \cong \colim_{\ga \geq \al} \left(X_\al(A) \times_{X_{\ga}(A)} \colim_\be X_{\ga}(A_\be)\right)$$
	since filtered colimits of spaces are left exact \cite[Prop. 5.3.3.3]{LurHTT}. But the right hand colimit is isomorphic to $\colim_\be X_\al(A_\be)$ since each individual term is by hypothesis. 
	
	Finally, by Proposition \ref{prop:factorthroughgeometric} we can factor $X' \to X$ through a morphism $j_\al: X' \to X_\al$ for some~$\al$. Then $j$ is a closed immersion by Proposition \ref{prop:cl2of3prop}, hence is affine by Proposition~\ref{prop:immisaffforgeo}, and is almost of finite presentation in the reasonable case by Proposition \ref{prop:afp2of3prop}. 
\end{proof}

\subsection{Properties of morphisms} Notions such as ind-properness extend from ind-schemes to ind-geometric stacks in the obvious way. 

\begin{Proposition}\label{prop:indPequivconditions}
	Let $f: X \to Y$ be a morphism of ind-geometric stacks, and let $X \cong \colim_\al X_\al$ be an ind-geometric presentation. The following conditions are equivalent.
	\begin{enumerate}
		\item 
		\begin{minipage}[t]{\linewidth-\labelwidth-\labelsep}
		For every diagram
		\begin{equation}\label{eq:indPdef}
			\begin{tikzpicture}
				[baseline=(current  bounding  box.center),thick,>=\arrtip]
				\node (a) at (0,0) {$X'$};
				\node (b) at (3,0) {$X$};
				\node (c) at (0,-1.5) {$Y'$};
				\node (d) at (3,-1.5) {$Y$};
				\node (e) at (5.0,-1.5) {$ $};
				\draw[->] (a) to node[above] {$ $} (b);
				\draw[->] (b) to node[right] {$f $} (d);
				\draw[->] (a) to node[left] {$f' $}(c);
				\draw[->] (c) to node[above] {$ $} (d);
			\end{tikzpicture}
		\end{equation}
		in which $X' \to X$ and $Y' \to Y$ are truncated geometric substacks, the map $f'$ is proper (resp. a closed immersion, of finite cohomological dimension). 
		\end{minipage}
	\vspace{2mm}
		\item 
		\begin{minipage}[t]{\linewidth-\labelwidth-\labelsep}
			For every $X_\al$ there exists a diagram
		\begin{equation}\label{eq:cancheckpresforindPdiag}
			\begin{tikzpicture}
				[baseline=(current  bounding  box.center),thick,>=\arrtip]
				\node (a) at (0,0) {$X_\al$};
				\node (b) at (3,0) {$X$};
				\node (c) at (0,-1.5) {$Y_\al$};
				\node (d) at (3,-1.5) {$Y$};
				\node (e) at (5.0,-1.5) {$ $};
				\draw[->] (a) to node[above] {$ $} (b);
				\draw[->] (b) to node[right] {$f $} (d);
				\draw[->] (a) to node[left] {$f_\al $}(c);
				\draw[->] (c) to node[above] {$ $} (d);
			\end{tikzpicture}
		\end{equation}
	in which $Y_\al$ is a truncated geometric substack of $Y$ and $f_\al$ is proper (resp. a closed immersion, of finite cohomological dimension). 
		\end{minipage}
	\end{enumerate}	
\end{Proposition}

\begin{Definition}\label{def:indP}
	A morphism $f: X \to Y$ of ind-geometric stacks is ind-proper (resp. an ind-closed immersion, of ind-finite cohomological dimension) if it satisfies the equivalent conditions of Proposition \ref{prop:indPequivconditions}.
\end{Definition}

\begin{proof}[Proof of Proposition \ref{prop:indPequivconditions}]
	Fix an ind-geometric presentation $Y \cong \colim Y_\be$. That (1) implies (2) follows since $f \circ i_\al$ factors through some $Y_\be$ by Proposition~\ref{prop:factorthroughgeometric}. To show (2) implies (1), fix a diagram (\ref{eq:indPdef}). By hypothesis and Proposition~\ref{prop:factorthroughgeometric} there exists a diagram of the left-hand form for some $\al$,  
	\begin{equation*}
		\begin{tikzpicture}[baseline=(current  bounding  box.center),thick,>=\arrtip]
			\node[matrix] at (0,0) {
				\newcommand*{\ha}{1.5}; \newcommand*{\hb}{1.5}; \newcommand*{\hc}{1.5};
				\newcommand*{\va}{-1.25}; \newcommand*{\vb}{-1.25};
				\node (aa) at (0,0) {$X'$};
				\node (ac) at (\ha+\hb,0) {$Y'$};
				\node (bb) at (\ha,\va) {$X$};
				\node (bd) at (\ha+\hb+\hc,\va) {$Y$};
				\node (ca) at (0,\va+\vb) {$X_\al$};
				\node (cc) at (\ha+\hb,\va+\vb) {$Y_\al$};
				\draw[->] (aa) to node[above] {$f' $} (ac);
				\draw[->] (bb) to node[above] {$f $} (bd);
				\draw[->] (ca) to node[above] {$f_\al $} (cc);
				\draw[->] (aa) to node[above] {$ $} (bb);
				\draw[->] (aa) to node[above] {$ $} (ca);
				\draw[->] (ca) to node[above] {$ $} (bb);
				\draw[->] (ac) to node[above] {$ $} (bd);
				\draw[->] (cc) to node[above] {$ $} (bd);\\
			};
			
			\node[matrix] at (7.0,0) {
				\newcommand*{\ha}{1.5}; \newcommand*{\hb}{1.5}; \newcommand*{\hc}{1.5}; \newcommand*{\hd}{1.5};
				\newcommand*{\va}{-1.25}; \newcommand*{\vb}{-1.25};
				\node (aa) at (0,0) {$X'$};
				\node (ac) at (\ha+\hb,0) {$Y'$};
				\node (bd) at (\ha+\hb,\va) {$Y_\be$};
				\node (be) at (\ha+\hb+\hc,\va) {$Y$,};
				\node (ca) at (0,\va+\vb) {$X_\al$};
				\node (cc) at (\ha+\hb,\va+\vb) {$Y_\al$};
				\draw[->] (aa) to node[above] {$f' $} (ac);
				\draw[->] (ca) to node[above] {$f_\al $} (cc);
				\draw[->] (aa) to node[above] {$ $} (ca);
				\draw[->] (ac) to node[above] {$ $} (bd);
				\draw[->] (cc) to node[above] {$ $} (bd);
				\draw[->] (ac) to node[above] {$ $} (be);
				\draw[->] (cc) to node[above] {$ $} (be);
				\draw[->] (bd) to node[above] {$ $} (be);\\
			};
		\end{tikzpicture}
	\end{equation*}
	where $f_\al$ is proper and $Y_\al \to Y$ is a truncated geometric substack. We claim this extends to a diagram of the right-hand form for some $Y_\be$. 
	
	To see this, note that for any finite diagram $p: K \to \GStk^+$, the natural map
	$$\colim \Map_{\Fun(K,\Stkk)}(p, Y_\be) \to \Maps_{\Fun(K,\Stkk)}(p, Y)$$
	is an isomorphism, where we let $Y$ and $Y_\be$ denote the associated constant diagrams. This follows since $p$ is compact in $\Fun(K,\oneStkkconv)$ by \cite[Prop. 5.3.4.13]{LurHTT} and Proposition \ref{prop:oneStkkconvcompactness}. The claim at hand follows by taking $p$ to be the subdiagram on the left spanned by $X'$, $Y'$, $X_\al$, and $Y_\al$. 
	
	In the right-hand diagram, the vertical maps are closed immersions by Proposition \ref{prop:igprestermsarereas}, hence $f'$ is proper by Proposition \ref{prop:propprops}. The other classes of morphisms are treated the same way, using Propositions \ref{prop:fcdprops} and \ref{prop:cl2of3prop}, and the following observation: if $f$ and $g$ are composable morphisms in $\GStkk$ such that $g \circ f$ is of finite cohomological dimension and $g$ is affine (hence $g_*$ conservative and t-exact), then $f$ is of finite cohomological dimension. 
\end{proof}

If $f: X \to Y$ is of ind-finite cohomological dimension and there exists an $n$ such that any morphism $f'$ as in (\ref{eq:indPdef}) is of cohomological dimension $\leq n$, then we say $f$ is of finite cohomological dimension. For example, an ind-closed immersion is of finite cohomological dimension (with $n= 0$), while the projection $\P^\infty := \bigcup \P^n \to \Spec \kk$ is of ind-finite, but not finite, cohomological dimension. 

\begin{Proposition}\label{prop:indPcompprops}
Ind-proper morphisms, ind-closed immersions, and morphisms of ind-finite cohomological dimension are stable under composition in $\indGStkk$.   
\end{Proposition}

\begin{proof} 
Let $X$, $Y$, and $Z$ be ind-geometric stacks, $f: X \to Y$ and $g: Y \to Z$ ind-proper morphisms, and $X \cong \colim X_\al$ an ind-geometric presentation. By definition there exist truncated geometric substacks $Y_\al \to Y$ and $Z_\al \to Z$ such that the restrictions of $f$ and $g$ factor through proper morphisms $f_\al: X_\al \to Y_\al$ and $g_\al: Y_\al \to Z_\al$ (note that $Y_\al \to Y$ may always be extended to a reasonable presentation, but the existence of the desired $Z_\al \to Z$ doesn't depend on this). But then $g_\al \circ f_\al$ is proper, hence $g \circ f$ is ind-proper. The other classes of morphisms are treated the same way. 
\end{proof}

\begin{Proposition}\label{prop:indPtwoofthreeprops}
	Let $f: X \to Y$ and $g: Y \to Z$ be morphisms of ind-geometric stacks. If $g \circ f$ and $g$ are ind-proper (resp. ind-closed immersions), then so is $f$. 
\end{Proposition}
\begin{proof}
Let $X \cong \colim X_\al$  be a reasonable presentation. By Proposition~\ref{prop:factorthroughgeometric} there exist truncated geometric substacks $Y_\al \to Y$ and $Z_\al \to Z$ such that the restrictions of $f$ and $g$ factor through morphisms $f_\al: X_\al \to Y_\al$ and $g_\al: Y_\al \to Z_\al$. By hypothesis $g_\al \circ f_\al$ and $g_\al$ are proper (resp. closed immersions), hence so is $f_\al$ by Proposition~\ref{prop:afp2of3prop} (resp. Proposition~\ref{prop:cl2of3prop}). 
\end{proof}

\begin{Proposition}\label{prop:indPconsistency}
	Let $f: X \to Y$ be a morphism of geometric stacks. Then $f$ is ind-proper (resp. an ind-closed immersion, of ind-finite cohomological dimension) if and only if it is proper (resp. a closed immersion, of finite cohomological dimension). 
\end{Proposition}
\begin{proof}
Recall that $X \cong \colim \tau_{\leq n} X$ and $X \cong \colim \tau_{\leq n} Y$ are ind-geometric presentations. Properness of $f$ is equivalent to properness of $\tau_{\leq 0} f$ \cite[Rem. 5.1.2.2]{LurSAG}, hence to properness of each $\tau_{\leq n} f$, hence to ind-properness. The corresponding claim for closedness is immediate, while for finiteness of cohomological dimension it follows from \cite[Lem. A.1.6]{HLP23} and the fact that $\QCoh(X)^\heartsuit \cong \QCoh(\tau_{\leq 0} X)^\heartsuit$. 
\end{proof}

We will often say a morphism of reasonable ind-geometric stacks is almost ind-finitely presented if it is almost finitely presented (i.e. in the sense of (\ref{eq:functorafp})). This is justified by the following result. 

\begin{Proposition}\label{prop:afpdefconsistent}
Let $f: X \to Y$ be a morphism of reasonable ind-geometric stacks, and let $X \cong \colim_\al X_\al$ be a reasonable presentation. The following conditions are equivalent.
\begin{enumerate}
	\item The morphism $f$ is almost finitely presented.
	\item For every diagram (\ref{eq:indPdef}) in which $X' \to X$ and $Y' \to Y$ are reasonable geometric substacks, the map $f'$ is almost finitely presented.
	\item For every $X_\al$ there exists a diagram  (\ref{eq:cancheckpresforindPdiag}) 
	in which $Y_\al$ is a reasonable geometric substack of $Y$ and $f_\al$ is almost finitely presented. 
\end{enumerate}	
\end{Proposition}
\begin{proof}
	That (1) implies (2) follows from Propositions~\ref{prop:afp2of3prop} and \ref{prop:igprestermsarereas}, and that (2) implies (3) is immediate. To show (3) implies (1) let $A \cong \colim A_\be$ be a filtered colimit in $\CAlgkleqn$ for some $n$. Then we have
	\begin{equation*}
		\colim_\be X(A_\be) \cong \colim_{\al,\be} X_\al(A_\be) 
		\cong \colim_\al \left( X_\al(A) \times_{Y(A)} \colim_\be Y(A_\be) \right),
	\end{equation*}
	the first isomorphism using Proposition~\ref{prop:factorthroughgeometric} and the second Proposition~\ref{prop:igprestermsarereas}. But the last expression is then isomorphic to $X(A) \times_{Y(A)} \colim_\be Y(A_\be)$ by the left exactness of filtered colimits of spaces. 
\end{proof} 

Ind-closed immersions have the following closure property. Here $\indGStkkcl \subset \indGStkk$  denotes the 1-full subcategory which only includes ind-closed immersions, similarly for $\GStkkpluscl \subset \GStkkplus$. Recall that a subcategory is 1-full if for $n>1$ it includes all $n$-simplices whose edges belong to the indicated class of morphisms. 

\begin{Proposition}\label{prop:indclosedclosure} 
	The canonical functor $\Ind(\GStkkpluscl) \to \Stkkconv$ factors through an equivalence $\Ind(\GStkkpluscl) \cong \indGStkkcl$. 
	In particular, ind-geometric stacks are closed under filtered colimits along ind-closed immersions  in $\Stkkconv$. 
\end{Proposition}
\begin{proof}
	By definition $\indGStkk$ is the essential image of $\Ind(\GStkkpluscl)$. Let $X \cong \colim X_\al$, $Y \cong \colim Y_\be$ be ind-geometric presentations. By abuse we denote the corresponding objects of $\Ind(\GStkkpluscl)$ by $X$ and $Y$ as well, so that
	$$\Map_{\Ind(\GStkkpluscl)}(X, Y) \cong \lim_\al \colim_\be \Map_{\GStkkpluscl}(X_\al, Y_\be).$$ 
	Now the natural map
	$$\lim_\al \colim_\be \Map_{\GStkkpluscl}(X_\al, Y_\be) \to \lim_\al \colim_\be \Map_{\GStkkplus}(X_\al, Y_\be) \cong \Map_{\indGStkk}(X, Y)$$
	is a monomorphism since monomorphisms are stable under limits and filtered colimits (note that the isomorphism on the right follows from Proposition \ref{prop:oneStkkconvcompactness}). 
	It thus suffices to show its image is exactly the subspace of ind-closed immersions, but this follows from the definitions. 
\end{proof}

\begin{Remark}\label{rem:nontruncindgdef}
	Note that a closed immersion of non-truncated geometric stacks is also an ind-closed morphism of ind-geometric stacks. It follows from Proposition \ref{prop:indclosedclosure} that $\indGStkk$ is the essential image of the (not fully faithful) functor $\Ind(\GStkkcl) \to \Stkkconv$, where $\GStkkcl \subset \GStkk$ is the 1-full subcategory which only includes closed immersions.  In other words, we obtain the same class of objects if in Definition \ref{def:indgeomstack} we do not require the $X_\al$ to be truncated. 
\end{Remark}

The following variant of Proposition \ref{prop:indclosedclosure} is proved the same way. Here $\indGStkkreasclafp \subset \indGStkkreas$ denotes the 1-full subcategory which only includes almost ind-finitely presented ind-closed immersions, similarly for $\indGStkkcohclafp \subset \indGStkkcoh$, $\GStkkplusclafp \subset \GStkkplus$, and $\GStkkpluscohclafp \subset \GStkkpluscoh$.
  
  \begin{Proposition}\label{prop:cohfiltcolims}
  	The canonical functor $\Ind(\GStkkplusclafp) \to \Stkkconv$  factors through an equivalence $\Ind(\GStkkplusclafp) \cong \indGStkkreasclafp$, and $\Ind(\GStkkpluscohclafp) \to \Stkkconv$ factors through an equivalence $\Ind(\GStkkpluscohclafp) \cong \indGStkkcohclafp$. In particular, reasonable (resp. coherent) ind-geometric stacks are closed under filtered colimits along almost ind-finitely presented ind-closed immersions  in $\Stkkconv$. 
  \end{Proposition}

\subsection{Fiber Products} 

Now we consider fiber products of ind-geometric stacks, and the base change properties of the classes of morphisms considered above. 

\begin{Proposition}\label{prop:indgeomfiberprods}
Ind-geometric stacks are closed under finite limits in $\Stkkconv$ (and $\Stkk$). 
\end{Proposition}
\begin{proof}
Note that $\Stkkconv$ is closed under limits in $\Stkk$, so the two claims are equivalent. Since $\indGStkk$ contains the terminal object $\Spec \kk$, it suffices to show closure under fiber products \cite[Prop. 4.4.2.6]{LurHTT}. 
	
Let $f: X \to Y$ and $h: Y' \to Y$ be morphisms of ind-geometric stacks, and let $X' := X \times_Y Y'$. Suppose first that $X$ and $Y'$ are truncated geometric stacks, and let $Y \cong \colim Y_\al$ be an ind-geometric presentation. By Proposition \ref{prop:factorthroughgeometric} we can factor $f$ and $h$ through $Y_\al$ for some $\al$. We have $X' \cong \colim_{\be \geq \al} X'_\be$, where $X'_\be := X \times_{Y_\be} Y'$, by left exactness of filtered colimits in $\Stkkconv$ \cite[Ex. 7.3.4.7]{LurHTT}. The transition maps are closed immersions of not necessarily truncated geometric stacks by Proposition~\ref{prop:classicaliso}. It follows they are ind-closed as morphisms of ind-geometric stacks, hence $X'$ is ind-geometric by Proposition~\ref{prop:indclosedclosure}. Now suppose $X \cong \colim X_\al$ and $Y' \cong \colim Y'_\be$ are ind-geometric presentations. Then as above $X' \cong \colim X_\al \times_Y Y'_\be$ expresses $X'$ as a filtered colimit in $\Stkkconv$ of ind-geometric stacks along ind-closed immersions, so again $X'$ is ind-geometric by Proposition~\ref{prop:indclosedclosure}. 
\end{proof}

\begin{Proposition}\label{prop:indPbasechange}
	Ind-proper morphisms (resp. ind-closed immersions, morphisms of ind-finite cohomological dimension) are stable under base change in $\indGStkk$. 
\end{Proposition}
\begin{proof}
	Let $f: X \to Y$ and $h: Y' \to Y$ be morphisms in $\indGStkk$ such that $f$ is ind-proper. If $X \cong \colim X_\al$ is an ind-geometric presentation, we have for all $\al$ a diagram
	\begin{equation*}
		\begin{tikzpicture}[baseline=(current  bounding  box.center),thick,>=\arrtip]
			\newcommand*{\ha}{1.5}; \newcommand*{\hb}{1.5}; \newcommand*{\hc}{1.5};
\newcommand*{\va}{-.9}; \newcommand*{\vb}{-.9}; \newcommand*{\vc}{-.9}; 
			\node (ab) at (\ha,0) {$X_\al'$};
			\node (ad) at (\ha+\hb+\hc,0) {$Y_\al'$};
			\node (ba) at (0,\va) {$X'$};
			\node (bc) at (\ha+\hb,\va) {$Y'$};
			\node (cb) at (\ha,\va+\vb) {$X_\al$};
			\node (cd) at (\ha+\hb+\hc,\va+\vb) {$Y_\al$};
			\node (da) at (0,\va+\vb+\vc) {$X$};
			\node (dc) at (\ha+\hb,\va+\vb+\vc) {$Y$};
			\draw[->] (ab) to node[above] {$\phi' $} (ad);
			\draw[->] (ab) to node[above] {$ $} (ba);
			\draw[->] (ab) to node[left,pos=.8] {$\psi' $} (cb);
			\draw[->] (ad) to node[above] {$ $} (bc);
			\draw[->] (ad) to node[right] {$\psi $} (cd);
			\draw[->] (ba) to node[above] {$ $} (da);
			\draw[->] (cb) to node[above,pos=.2] {$\phi $} (cd);
			\draw[->] (cb) to node[above] {$ $} (da);
			\draw[->] (cd) to node[above] {$ $} (dc);
			\draw[->] (da) to node[above,pos=.6] {$f $} (dc);
			
			\draw[-,line width=6pt,draw=white] (ba) to  (bc);
			\draw[->] (ba) to node[above,pos=.75] {$f' $} (bc);
			\draw[-,line width=6pt,draw=white] (bc) to  (dc);
			\draw[->] (bc) to node[right,pos=.2] {$h $} (dc);
		\end{tikzpicture}
	\end{equation*}
	in $\indGStkk$ such that all but the top and bottom faces are Cartesian, $Y_\al \to Y$ is a truncated geometric substack, and $\phi$ is proper. Let $Y'_\al \cong \colim_\be Y'_{\al\be}$ be an ind-geometric presentation. Then, letting $X'_{\al\be} := X'_\al \times_{Y'_\al} Y'_{\al\be}$, we have $X'_\al \cong \colim_\be X'_{\al\be}$ by left exactness of filtered colimits in $\Stkkconv$. Note that for all $\be$ the morphisms $X'_{\al\be} \to X'$ and $Y'_{\al\be} \to Y'$ are closed immersions since $Y'_{\al\be} \to Y'_\al$ and $Y_\al \to Y$ are, and in particular $Y'_{\al\be}$ is a truncated geometric substack of $Y'$. 
	
	Now let $X' \cong \colim X'_{\ga}$ be an ind-geometric presentation and fix some $\ga$. By Proposition~\ref{prop:factorthroughgeometric} we can choose $\al$ so that $X'_\ga \to X$ factors through $X_\al$, hence so that $X'_\ga \to X'$ factors through $X'_\al$.  Proposition \ref{prop:oneStkkconvcompactness}  then implies that $X'_\ga \to X'$ factors through $X'_{\al\be}$ for some $\be$. This map $X'_\ga \to X'_{\al\be}$ is a closed immersion since $X'_\ga \to X'$ and $X'_{\al\be} \to X'$ are, while $X'_{\al\be} \to Y'_{\al\be}$ is proper since it is a base change of $\phi$. Thus the composition  $X'_\ga \to X'_{\al\be} \to Y'_{\al\be}$ is proper, hence $f'$ is ind-proper. The other classes of morphisms are treated the same way. 
\end{proof}

Already the self-intersection of the origin in $\A^\infty$ illustrates that reasonable ind-geometric stacks are not closed under arbitrary fiber products. To formulate a more limited result, we say a morphism $h: X \to Y$ of ind-geometric stacks is of Tor-dimension~$\leq n$ (resp. of finite Tor-dimension) if it is geometric and its base change to any geometric stack is of Tor-dimension~$\leq n$ (resp. of finite Tor-dimension) in the sense of sense of Section~\ref{sec:geompropsofmors}. If $X$ and $Y$ are geometric, this is consistent with our previous terminology  by Proposition \ref{prop:ftdprops}, which also implies the following stability properties. 

\begin{Proposition}\label{prop:indftdcompprop}%{prop:indcohpullcompprop}
	Morphisms of finite Tor-dimension are stable under composition and base change in $\indGStkk$.
\end{Proposition}

We then have the following closure result in the reasonable case. 

\begin{Proposition}\label{prop:indPbaseprops}
Let $h: X \to Y$ be a morphism of finite Tor-dimension between ind-geometric stacks. If $Y$ is reasonable, so is $X$. 
In particular, let the following be a Cartesian diagram of ind-geometric stacks. 
\begin{equation*}
	\begin{tikzpicture}
		[baseline=(current  bounding  box.center),thick,>=\arrtip]
		\node (a) at (0,0) {$X'$};
		\node (b) at (3,0) {$Y'$};
		\node (c) at (0,-1.5) {$X$};
		\node (d) at (3,-1.5) {$Y$};
		\draw[->] (a) to node[above] {$f' $} (b);
		\draw[->] (b) to node[right] {$h $} (d);
		\draw[->] (a) to node[left] {$h' $}(c);
		\draw[->] (c) to node[above] {$f $} (d);
	\end{tikzpicture}
\end{equation*}
If $X$, $Y$, and $Y'$ are reasonable and $h$ is of finite Tor-dimension, then $X'$ is reasonable.  
\end{Proposition}

\begin{proof}
Let $Y \cong \colim Y_\al$ be a reasonable presentation. Each $X_\al := X \times_Y Y_\al$ is a truncated geometric stack since $h$ is of finite Tor-dimension. We have $X \cong \colim X_\al$ since filtered colimits are left exact in $\Stkkconv$, and this is a reasonable presentation since almost finitely presented closed immersions are stable under base change. The last claim now follows by Proposition \ref{prop:indftdcompprop}. 
\end{proof}

The coherent case is more delicate, as even coherent affine schemes are not closed under fiber products \cite[Sec. 7.3.13]{Gla89}. We give two positive results in this setting. 

\begin{Proposition}\label{prop:cohbasechangeunderafpclimm}
	Let the following be a Cartesian diagram of ind-geometric stacks. 
	\begin{equation*}
		\begin{tikzpicture}
			[baseline=(current  bounding  box.center),thick,>=\arrtip]
			\node (a) at (0,0) {$X'$};
			\node (b) at (3,0) {$Y'$};
			\node (c) at (0,-1.5) {$X$};
			\node (d) at (3,-1.5) {$Y$};
			\draw[->] (a) to node[above] {$f' $} (b);
			\draw[->] (b) to node[right] {$h $} (d);
			\draw[->] (a) to node[left] {$h' $}(c);
			\draw[->] (c) to node[above] {$f $} (d);
		\end{tikzpicture}
	\end{equation*}
	Suppose that $X$ and $Y$ are reasonable, that $Y'$ is coherent, and that $f$ is an almost ind-finitely presented ind-closed immersion. Then $X'$ is coherent and $f'$ is an almost ind-finitely presented ind-closed immersion. 
\end{Proposition}

\begin{Lemma}\label{lem:affcgen}
	Let $f: X \to Y$ be an affine morphism of geometric stacks. If $\QCoh(Y)^\heartsuit$ is compactly generated, then so is $\QCoh(X)^\heartsuit$. 
\end{Lemma}
\begin{proof}
	Since $f$ is affine $f_*: \QCoh(X) \to \QCoh(Y)$ is t-exact and conservative, hence restricts to a conservative functor $\QCoh(X)^\heartsuit \to \QCoh(Y)^\heartsuit$. This restriction is continuous and has a left adjoint, the restriction of $\tau^{\geq 0} \circ f^*$. Thus compact generation of $\QCoh(Y)^\heartsuit$ implies that of $\QCoh(X)^\heartsuit$ by \cite[Prop. 7.1.4.12]{LurHA} (whose proof applies to compact generation, not just compact projective generation). 
\end{proof}

\begin{Lemma}\label{lem:cohsourcegeomcase}
	Let $f: X \to Y$ be an almost finitely presented closed immersion of geometric stacks. If $Y$ is locally coherent (resp. coherent), then so is $X$.
\end{Lemma}
\begin{proof}
	Let $\Spec A \to Y$ be a flat cover with $A$ coherent, and $f': \Spec B \to \Spec A$ the base change of $f$ (recall that $f$ is affine by Proposition \ref{prop:immisaffforgeo}). By \cite[Cor. 5.2.2.2]{LurSAG} $B$~is almost perfect as an $A$-module, hence $H^n(B)$ is finitely presented over $H^0(A)$ for all $n \leq 0$. Moreover $H^0(B)$ is a quotient of $H^0(A)$ by a finitely generated ideal, so $H^0(B)$ is coherent and the $H^n(B)$ are finitely presented over $H^0(B)$ \cite[Thm. 2.4.1]{Gla89}. If $\QCoh(Y)^\heartsuit$ is compactly generated, then so is $\QCoh(X)^\heartsuit$ by Lemma \ref{lem:affcgen}. 
\end{proof}

\begin{proof}[Proof of Proposition \ref{prop:cohbasechangeunderafpclimm}]
	Suppose first that $X$ and $Y'$ are truncated geometric stacks, and let $Y \cong \colim Y_\al$ be a reasonable presentation. We may assume $f$ and $h$ factor through maps $f_\al: X \to Y_\al$, $h_\al: Y' \to Y_\al$ for all $\al$. Letting $X'_\al := X \times_{Y_\al} Y'$, each $f'_\al: X'_\al \to Y'$ is an almost finitely presented closed immersion by base change and Proposition~\ref{prop:afpdefconsistent}, hence $X'_\al$ is coherent by Lemma~\ref{lem:cohsourcegeomcase}. For any $\be \geq \al$ the induced map $i'_{\al \be}: X'_\al \to X'_\be$ is an almost finitely presented closed immersion since $f'_\be \circ i'_{\al \be} \cong f'_\al$ (Proposition \ref{prop:afp2of3prop}). Since $X' \cong \colim X \times_{Y_\al} Y'$ in $\Stkkconv$ by left exactness of filtered colimits, it follows that $X'$ is coherent by Proposition~\ref{prop:cohfiltcolims}. 
	
	In general, fix reasonable presentations $X \cong \colim X_\al$ and $Y' \cong \colim Y'_\be$. Then as above $X' \cong \colim X_\al \times_Y Y'_\be$ presents $X'$ as a filtered colimit of coherent ind-geometric stacks along almost ind-finitely presented ind-closed immersions, hence $X'$ is coherent by Proposition~\ref{prop:cohfiltcolims}. That $f'$ is an almost ind-finitely presented ind-closed immersion follows from Propositions~\ref{prop:afpdefconsistent} and \ref{prop:indPbasechange}. 
\end{proof}

We say an ind-geometric stack $X$ is locally Noetherian if it has a reasonable presentation $X \cong \colim X_\al$ in which each $X_\al$ is locally Noetherian (as noted before, this implies $X$ is coherent). The proof of Proposition \ref{prop:cohfiltcolims} extends to show locally Noetherian ind-geometric stacks are closed under almost finitely ind-closed immersions. If $f: X \to Y$ is a proper, almost finitely presented morphism of geometric stacks and $Y$ is locally Noetherian, it follows that $X$ is as well by base changing $f$ to a Noetherian flat cover $\Spec A \to Y$. The proof of Proposition \ref{prop:cohbasechangeunderafpclimm} then extends to show the following result (which will be strengthened in \cite{CWtm} once we have developed the notion of a tamely presented morphism). 

\begin{Proposition}\label{prop:locNoethCartesian}
	Let the following be a Cartesian diagram of ind-geometric stacks. 
	\begin{equation*}
		\begin{tikzpicture}
			[baseline=(current  bounding  box.center),thick,>=\arrtip]
			\node (a) at (0,0) {$X'$};
			\node (b) at (3,0) {$Y'$};
			\node (c) at (0,-1.5) {$X$};
			\node (d) at (3,-1.5) {$Y$};
			\draw[->] (a) to node[above] {$f' $} (b);
			\draw[->] (b) to node[right] {$h $} (d);
			\draw[->] (a) to node[left] {$h' $}(c);
			\draw[->] (c) to node[above] {$f $} (d);
		\end{tikzpicture}
	\end{equation*}
	Suppose that $X$ and $Y$ are reasonable, that $Y'$ is locally Noetherian, and that $f$ is ind-proper and almost ind-finitely presented. Then $X'$ is locally Noetherian. 
\end{Proposition}

\section{Coherent and ind-coherent sheaves}\label{sec:cohandindcohsheaves}

In this section we consider coherent and ind-coherent sheaves on ind-geometric stacks. We begin with the former, establishing the basic functorialities of ind-proper pushforward and finite Tor-dimension pullback. We then extend our discussion to ind-coherent sheaves, which are needed to discuss adjoint functorialities such as ind-proper $!$-pullback and sheaf Hom. 

The most significant complication compared to the treatment of ind-schemes in \cite{GR14}, \cite{Ras20} lies in the definition of ind-coherent sheaves. If $X$ is a geometric stack we define $\IndCoh(X)$ as the left anticompletion of $\QCoh(X)$, following a construction of \cite[App. C]{LurSAG}. This characterizes $\IndCoh(X)$ in terms of a universal property satisfied by bounded colimit-preserving functors out of it. If $X$ is classical, $\IndCoh(X)$ is (the dg nerve of) the category of injective complexes in $\QCoh(X)^\heartsuit$, introduced in \cite{Kra05}. However, in full generality we do not have $\IndCoh(X) \cong \Ind(\Coh(X))$, and the latter may be poorly behaved.  

Nonetheless, our notation (perhaps abusively) reflects that in most cases of interest these categories do coincide. Specifically, they agree when $X$ is coherent (Proposition \ref{prop:IndCohisIndofCoh}), so in this case one may safely define $\IndCoh(X)$ via ind-completion and bypass the discussion of anticompletion. But it is often convenient to have $\IndCoh(X)$ defined in greater generality, in particular since the class of coherent ind-geometric stacks (or even of coherent affine schemes) is not closed under fiber products. 

\subsection{Coherent sheaves}\label{sec:cohsubsec}
We first define $\Coh(X)$ for a reasonable ind-geometric stack $X$ (see Example \ref{ex:selfintAinfty} for an illustration of why we only consider the reasonable setting). A posteriori, this category will be computed by the formula 
\begin{equation}\label{eq:cohviaindgeopres}
	\Coh(X) \cong \colim \Coh(X_\al)
\end{equation} 
in $\Catinfty$, where $X \cong \colim X_\al$ is any reasonable presentation. In particular, any $\cF \in \Coh(X)$ can be written as $\cF \cong i_{\al*}(\cF_\al)$ for some $\al$ and $\cF_\al \in \Coh(X_\al)$. 

Suppose that $Y$ is another reasonable ind-geometric stack and that $f: X \to Y$ is ind-proper and almost ind-finitely presented. The pushforward $f_*: \Coh(X) \to \Coh(Y)$ will be defined so that $f_*(\cF) \cong j_{\al*} f_{\al*}(\cF_\al)$, where 
$X_\al \xrightarrow{f_\al} Y_\al \xrightarrow{j_\al} Y$ is any factorization of $f\circ i_\al$ through a reasonable geometric substack of $Y$. If $h: Y \to X$ is of finite Tor-dimension, the pullback $h^*: \Coh(Y) \to \Coh(X)$ will be defined so that $h^*(\cF) \cong i'_{\al*} h_{\al}^*(\cF_\al)$, where $i'_\al$ and $h_\al$ are defined by base change from $i_\al$ and $h$. 

Recall from (\ref{eq:CorrCohGStkprop}) that the corresponding functorialities for coherent sheaves on truncated geometric stacks were packaged as a functor 
\begin{equation}\label{eq:CohCorrpropcoh2}
	\Coh: \Corr(\GStkkplus)_{prop,ftd} \to \Catinfty.
\end{equation}
We now let $\Corr(\indGStkkreas)_{prop,ftd}$ denote the 1-full subcategory of $\Corr(\indGStkk)$ which only includes correspondences $X \xleftarrow{h} Y \xrightarrow{f} Z$ such that $h$ is of finite Tor-dimension, $f$ is ind-proper and almost ind-finitely presented, and $X$ and $Z$ are reasonable (hence so is $Y$ by Proposition \ref{prop:indPcompprops}). These are indeed stable under composition of correspondences by Propositions \ref{prop:indPcompprops}, \ref{prop:indPbasechange}, \ref{prop:indftdcompprop}, and \ref{prop:indPbaseprops}, and we note that $\Corr(\GStkkplus)_{prop,ftd}$ is a full subcategory of $\Corr(\indGStkk)_{prop,ftd}$. 

\begin{Definition}\label{def:cohonindgstks}
	We define a functor
	\begin{equation}\label{eq:cohcorrindgstkfunctor}
		\Coh: \Corr(\indGStkkreas)_{prop,ftd} \to \Catinfty
	\end{equation}
	by left Kan extending (\ref{eq:CohCorrpropcoh2}) along $\Corr(\GStk^+)_{prop,ftd} \subset \Corr(\indGStkkreas)_{prop,ftd}$.   
\end{Definition}

This Kan extension exists by \cite[Cor. 4.3.2.16, Cor. 4.2.4.8]{LurHTT}. Taking either $h$ or $f$ to be the identity, the values of (\ref{eq:cohcorrindgstkfunctor}) on the correspondence $X \xleftarrow{h} Y \xrightarrow{f} Z$ define functors $h^*: \Coh(X) \to \Coh(Y)$ and $f_*: \Coh(Y) \to \Coh(Z)$. The formula (\ref{eq:cohviaindgeopres}) is a consequence of the following result. 

\newcommand{\Cohgeom}{\Coh_{geom}}
\newcommand{\Cohind}{\Coh_{ind}}
\newcommand{\CohInd}{\Coh_{Ind}}

\begin{Proposition}\label{prop:cohonindgstks}
	The restriction of $\Coh$ to $\indGStkkprop$ preserves filtered colimits along almost ind-finitely presented ind-closed immersions. 
\end{Proposition}

To show this we need the following variant of a special case of \cite[Thm. 9.6.1.5]{GR17}, whose proof we include since the cited statement is significantly more general. In the statement $\catC$ is a category with finite limits, and $\vert$ and $\horiz$ are classes of morphisms in $\catC$ which contain all isomorphisms and are stable under composition and under base change along each other. Recall that if $\catC' \subset \catC$ is a full subcategory such that $Y \in \catC'$ whenever $h: Y \to X$ is in $\horiz$ and $X \in \catC'$, then $\Corr(\catC')_{\vert,\horiz}$ is the 1-full subcategory of $\Corr(\catC)$ which only includes correspondences $X \xleftarrow{h} Y \xrightarrow{f} Z$ such that $h \in \horiz$, $f \in \vert$, and $X, Z \in \catC'$ (hence $Y \in \catC'$). 

\begin{Proposition}\label{prop:corrKanext}
	\cite[Thm. 9.6.1.5]{GR17} Let $\catC$, $\vert$, and $\horiz$ be as above, let $\catC'' \subset \catC' \subset \catC$ be full subcategories which both satisfy the above condition with respect to $\horiz$, and let $\catD$ be another category. Let $F'': \Corr(\catC'')_{\vert,\horiz} \to \catD$ be a functor, $F': \Corr(\catC')_{\vert,\horiz} \to \catD$ a left Kan extension of $F''$, and $G': \catC'_{\vert} \to \catD$ a left Kan extension of  $F''|_{\catC''_{\vert}}$. Then the canonical transformation $G' \to F'|_{\catC'_{\vert}}$ is an isomorphism. 
\end{Proposition}

\begin{proof}
	It suffices to show $(\catC''_{\vert})_{/Z} \to (\Corr(\catC'')_{\vert,\horiz})_{/Z}$ is left cofinal for all $Z \in \catC'$, since the canonical morphism $G'(Z) \to F'(Z)$ is given by taking colimits over the restrictions of $F'$ to these diagrams. It further suffices to show $((\catC''_{\vert})_{/Z})_{(h,f)/}$ is weakly contractible for any object $X \xleftarrow{h} Y \xrightarrow{f} Z$ of $(\Corr(\catC'')_{\vert,\horiz})_{/Z}$ \cite[Thm. 4.1.3.1]{LurHTT}. The category $((\catC''_{\vert})_{/Z})_{(h,f)/}$ can be identified with the category of diagrams
	\begin{equation*}
		\begin{tikzpicture}
			[baseline=(current  bounding  box.center),thick,>=\arrtip]
			\newcommand*{\ha}{3}; \newcommand*{\hb}{2.5};
			\newcommand*{\va}{-1.3}; \newcommand*{\vb}{-1.5};
			\node (aa) at (0,2*\va) {$X$};
			\node (ab) at (.5*\ha,\va) {$Y'$};
			\node (ac) at (\ha,0) {$Y$};
			%\node (ba) at (0,\va) {$X_\al$};
			\node (bb) at (\ha,2*\va) {$W'$};
			\node (bc) at (1.5*\ha,\va) {$W$};
			%\node (ca) at (0,\va+\vb) {$X_\al$};
			%\node (cb) at (\ha,\va+\vb) {$X$};
			\node (cc) at (2*\ha,2*\va) {$Z$};
			\draw[<-] (aa) to node[above left] {$h $} (ab);
			\draw[<-] (ab) to node[above left] {$\xi'  $} (ac);
			%\draw[->] (ba) to node[above] {$  $} (bb);
			\draw[<-] (bb) to node[above left] {$\xi  $} (bc);
			\draw[->] (ac) to node[above right] {$\phi  $} (bc);
			\draw[->] (ab) to node[above right] {$\phi'  $} (bb);
			\draw[->] (bc) to node[above right] {$\psi $} (cc);
		\end{tikzpicture}
	\end{equation*} 
	in which $W \in \catC''$, $f \cong \psi \circ \phi$, $\xi$ and $\xi'$ are isomorphisms, and $\phi$ and $\psi$ belong to $\vert$. Up to contractible choices such a diagram is determined by the subdiagram spanned by the top right edges, hence this category further identifies with $((\catC''_{\vert})_{/Z})_{f/}$. Since $X \in \catC''$, the hypothesis on $h$ implies that $Y \in \catC''$. Thus $f$ itself belongs to $(\catC''_{\vert})_{/Z}$, hence $((\catC''_{\vert})_{/Z})_{f/}$ has an initial object and is weakly contractible. 
\end{proof}

\begin{proof} [Proof of Proposition \ref{prop:cohonindgstks}]
	Write $\GStkkplusprop := \GStkkplus \cap \indGStkkreasprop$ for the category of truncated geometric stacks and proper, almost finitely presented morphisms, and let $\Cohind$ and $\Cohgeom$ denote the restrictions of $\Coh$ to $\indGStkkreasprop$ and  $\GStkkplusprop$, respectively. It follows from Proposition \ref{prop:corrKanext} that $\Cohind$ is the left Kan extension of $\Cohgeom$ along the inclusion $\GStkkplusprop \subset \indGStkkreasprop$. 
	
	The proof of Proposition \ref{prop:indclosedclosure} adapts to show that the canonical continuous functor $\Ind(\GStkkplusprop) \to \Stkkconv$ identifies $\Ind(\GStkkplusprop)$ with a subcategory of $\Stkkconv$, and that $\indGStkkreasprop$ is the intersection of $\indGStkkreas$ with $\Ind(\GStkkplusprop)$ in $\Stkkconv$. In particular, $\indGStkkreasprop$ identifies with a full subcategory of $\Ind(\GStkkplusprop)$. 
	By the transitivity of left Kan extensions \cite[Prop. 4.3.2.8]{LurHTT}, $\Cohind$ is the restriction to $\indGStkkreasprop$ of $\CohInd$, the left Kan extension of $\Cohgeom$ to $\Ind(\GStkkplusprop)$. 
	
	Now write $\indGStkkreasclafp \subset \indGStkkreas$ for the subcategory which only includes almost ind-finitely presented ind-closed immersions. Proposition \ref{prop:indclosedclosure} states that $\indGStkkreasclafp$ admits filtered colimits and its inclusion into $\Stkkconv$, hence into $\Ind(\GStkkplusprop)$, is continuous. But $\CohInd$ is continuous \cite[Lem. 5.3.5.8]{LurHTT}, hence so is its restriction to $\indGStkkreasclafp$. 
\end{proof}

\begin{Proposition}\label{prop:cohidcomp}
The category $\Coh(X)$ is small, stable, and idempotent complete for any reasonable ind-geometric stack $X$.
\end{Proposition}
\begin{proof}
If $X$ is geometric this follows from \cite[Prop. 7.2.4.11]{LurHA}.
Given (\ref{eq:cohviaindgeopres}), the general case follows from \cite[Prop. 1.1.4.6, Lem. 7.3.5.10]{LurHA}. 
\end{proof}

\subsection{Anticompletion}\label{sec:anticompletion} Before turning to ind-coherent sheaves, we review the notion of anticompleteness from \cite[App. C]{LurSAG} in slightly adapted form. Recall that a t-structure on a stable $\infty$-category $\catC$ is left complete if the natural functor $\catC \to \lim_n \catC^{\geq n}$ is an equivalence, and is right complete if $\catC \to \lim_n \catC^{\leq n}$ is an equivalence. The category $\wh{\catC} := \lim_n \catC^{\geq n}$ is called the left completion of $\catC$.
It has a canonical t-structure such that $\catC \to \wh{\catC}$ is t-exact and restricts to an equivalence $\catC^{\geq 0} \congto \wh{\catC}^{\geq 0}$ \cite[Prop. 1.2.1.17]{LurHA}. 

 If $\catD$ is another stable $\infty$-category with a t-structure, an exact functor $F: \catC \to \catD$ is bounded if there exist $m, n$ such that $F(\catC^{\geq 0}) \subset \catD^{\geq m}$ and  $F(\catC^{\leq 0}) \subset \catD^{\leq n}$. If $\catC$ and $\catD$ are presentable, we write $\LFun^b(\catC, \catD) \subset \LFun(\catC, \catD)$ for the full subcategory of bounded colimit-preserving functors. In the presentable case, a t-structure on $\catC$ is accessible if $\catC^{\geq 0}$ is also presentable, and is compatible with filtered colimits if $\catC^{\geq 0}$ is closed under filtered colimits in $\catC$. 
 
We let $\PrStbrcpl$ denote the $\infty$-category whose objects are presentable stable $\infty$-categories equipped with accessible t-structures which are right complete and compatible with filtered colimits, and whose morphisms are bounded colimit-preserving functors. Explicitly, given $\catC \in \PrSt$ we consider the set of cores $\catC^{\leq 0}$ (subcategories closed under small colimits and extensions), partially ordered by $\catC^{\leq 0}_1 < \catC^{\leq 0}_2$ if $\catC^{\leq 0}_1 \subset \catC^{\leq 0}_2[n]$ for some $n$. These posets are contravariantly functorial under taking preimages along exact functors, and $\PrStbrcpl$ is a full subcategory of the associated Cartesian fibration over $\PrSt$. 

\begin{Definition}\label{def:anticomplete}
We say $\catC \in \PrStbrcpl$ is left anticomplete if composition with $\catD \to \wh{\catD}$ induces an equivalence
$$ \LFun^{b}(\catC, \catD) \congto \LFun^{b}(\catC, \wh{\catD})$$ 
for any $\catD \in \PrStbrcpl$. 
\end{Definition}

We further let $\PrStbacpl$ and $\PrStbcpl$ denote the full subcategories of $\PrStbrcpl$ defined by only including t-structures which are respectively left anticomplete and left complete. We then have the following variant of \cite[Cor. C.3.6.4, Cor. C.5.5.11, Prop. C.5.9.2]{LurSAG}. 

\begin{Proposition}\label{prop:acplcplequiv}
	The inclusion $\PrStbcpl \into \PrStbrcpl$ admits a left adjoint, which acts on objects by $\catC \mapsto \wh{\catC}$. The inclusion $\PrStbacpl \into \PrStbrcpl$ admits a right adjoint, which we denote by $\catC \mapsto \wc{\catC}$.
	The restrictions of these adjoints define inverse equivalences between $\PrStbcpl$ and $\PrStbacpl$. 
\end{Proposition}
\begin{proof}	
Given $\catC, \catD \in \PrStbrcpl$, write $\LFun^{b, \leq n}(\catC, \catD) \subset \LFun^{b}(\catC, \catD)$ for the full subcategory of functors which take $\catC^{\leq 0}$ to $\catD^{\leq n}$. 
Since $\catC \to \wh{\catC}$ is t-exact, composition with it induces a functor 
$ \LFun^{b, \leq n}(\wh{\catC}, \catD) \to \LFun^{b, \leq n}(\catC, \catD).$ 
The existence of the desired adjoint follows if this is an equivalence for all $n$ and for all $\catD \in \PrStbcpl$ \cite[Prop. 5.2.4.2]{LurHTT}. By shifting we can reduce to $n = 0$. That composition with $\catC \to \wh{\catC}$ induces an equivalence between right t-exact functors follows from \cite[Prop. C.3.1.1., Prop. C.3.6.3]{LurSAG} (noting that $\catD \cong \Sp(\catD^{\leq 0})$ \cite[Rem. C.3.1.5]{LurSAG}). But this further identifies right t-exact functors which are bounded since $\catC^{\geq 0} \congto \wh{\catC}^{\geq 0}$.

Now let $\wc{\catC} := \Sp(\wc{\catC}^{\leq 0}) \in \PrStbrcpl$, where $\wc{\catC}^{\leq 0}$ is the anticompletion of ${\catC}^{\leq 0}$ in the sense of \cite[Prop. C.5.5.9]{LurSAG}. We claim $\wc{\catC}$ is left anticomplete in the sense of Definition~\ref{def:anticomplete}. It suffices to show that for any $\catD \in \PrStbrcpl$, the functor
$ \LFun^{b, \leq n}(\wc{\catC}, \catD) \to \LFun^{b, \leq n}(\wc{\catC}, \wh{\catD}) $
given by composition with $\catD \to \wh{\catD}$ is an equivalence for all $n$. By shifting we can reduce to the case $n = 0$, which follows from \cite[Prop. C.3.1.1, Def. C.5.5.4]{LurSAG}. 

The left-exact functor $\wc{\catC}^{\leq 0} \to \catC^{\leq 0}$ of \cite[Prop. C.5.5.9]{LurSAG} induces a t-exact functor $\wc{\catC} \to \catC$ \cite[Prop. C.3.1.1, Prop. C.3.2.1]{LurSAG}.  
The existence of the desired adjoint follows if the induced functor 
$ \LFun^{b}(\catD, \wc{\catC}) \to \LFun^{b}(\catD, \catC)$
is an equivalence for all $\catD \in \PrStbacpl$ \cite[Prop. 5.2.4.2]{LurHTT}. But this follows from Definition~\ref{def:anticomplete}, given that the left completions of $\wc{\catC}$ and $\catC$ are equivalent \cite[Prop. C.5.5.9]{LurSAG}.

Finally, it follows by adjunction and Definition~\ref{def:anticomplete} that $\catC \mapsto \wh{\catC}$ has fully faithful restriction to $\PrStbacpl$, likewise for $\catC \mapsto \wc{\catC}$ and $\PrStbcpl$. That these restrictions are inverse equivalences now follows since $\catC \in \PrStbcpl$ implies $\catC$ is the left completion of $\wc{\catC}$ \cite[Prop. C.5.5.9]{LurSAG}.  
\end{proof}

\subsection{Ind-coherent sheaves} Recall that if $X$ is a geometric stack, the standard t-structure on $\QCoh(X)$ is accessible, left and right complete, and compatible with filtered colimits \cite[Cor. 9.1.3.2]{LurSAG}. 

\begin{Definition}\label{def:IndCohononeGStk}
The category of ind-coherent sheaves on a geometric stack $X$ is
$$ \IndCoh(X) := \wc{\QCoh(X)},$$
the left anticompletion of its category of quasicoherent sheaves. We write $\Psi_X: \IndCoh(X) \to \QCoh(X)$ for the left completion functor. 
\end{Definition}

In particular, $\Psi_X$ restricts to an equivalence $$\IndCoh(X)^+ \congto \QCoh(X)^+.$$
Unwinding the definitions, we see that $\IndCoh(X)$ is uniquely characterized by the following universal property: for all $\catC \in \PrStbrcpl$ we have
\begin{equation}\label{eq:ICunivprop} \LFun^b(\IndCoh(X), \catC) \cong \LFun^b(\QCoh(X), \wh{\catC}). 
\end{equation} 

By Proposition \ref{prop:acplcplequiv}, ind-coherent sheaves on geometric stacks inherit all bounded, colimit-preserving functorialities of quasicoherent sheaves. 
Recall from (\ref{eq:CorrQCohGStk}) that pullback and pushforward of quasicoherent sheaves were packaged as a functor 
\begin{equation*}
	\QCoh: \Corr(\GStkk)_{fcd,all} \to \Cathatinfty.
\end{equation*}
By construction its restriction to $\Corr(\GStkk)_{fcd;ftd}$, the 1-full subcategory which only includes correspondences $X \xleftarrow{h} Y \xrightarrow{f} Z$ in which $h$ is of finite Tor-dimension and $f$ is of finite cohomological dimension, lifts to a functor 
\begin{equation}\label{eq:boundedQCoh}
	\QCoh: \Corr(\GStkk)_{fcd;ftd} \to \PrStbcpl.
\end{equation}

\begin{Definition}\label{def:IndCohGStk} We define a functor 
	\begin{equation}\label{eq:boundedIndCoh}
		\IndCoh: \Corr(\GStkk)_{fcd;ftd} \to \PrStbacpl.
	\end{equation}
	by composing (\ref{eq:boundedQCoh}) with the equivalence  $\PrStbcpl \congto \PrStbacpl$ of Proposition \ref{prop:acplcplequiv}. 
\end{Definition}

If $f: X \to Y$ is a morphism of finite cohomological dimension in $\GStkk$, we write $f_*: \IndCoh(X) \to \IndCoh(Y)$ for the associated functor. To distinguish it from $f_*: \QCoh(X) \to \QCoh(Y)$ we sometimes denote them by $f_{IC*}$ and $f_{QC*}$, respectively, but usually we arrange for the meaning to be clear from context. The two are related by a canonical isomorphism $\Psi_Y f_{IC*} \cong f_{QC*} \Psi_X$, and the same remarks apply to the functor $h^*$ associated to a morphism $h$ of finite Tor-dimension. 

To extend our discussion to ind-geometric stacks, first consider the functor 
\begin{equation}\label{eq:truncIndCoh}
\IndCoh: \Corr(\GStkkplus)_{fcd;ftd} \to \PrL
\end{equation}
obtained by restricting (\ref{eq:boundedIndCoh}) to correspondences of truncated geometric stacks and composing with the forgetful functor $\PrStbacpl \to \PrL$. 
Now let $\Corr(\indGStkk)_{fcd;ftd}$ denote the 1-full subcategory of $\Corr(\indGStkk)$ which only includes correspondences $X \xleftarrow{h} Y \xrightarrow{f} Z$ such that $h$ is of finite Tor-dimension and $f$ is of ind-finite cohomological dimension. These are stable under composition of correspondences by Propositions \ref{prop:indPcompprops}, \ref{prop:indPbasechange}, and \ref{prop:indftdcompprop}. 

\begin{Definition}\label{def:IndCohindGStk}
We define a functor 
\begin{equation}\label{eq:IndCohindGStk}
\IndCoh: \Corr(\indGStkk)_{fcd;ftd} \to \PrL
\end{equation}
by left Kan extending (\ref{eq:truncIndCoh}) along  $\Corr(\GStkkplus)_{fcd;ftd} \subset \Corr(\indGStkk)_{fcd;ftd}$. 
\end{Definition}

This Kan extension exists by \cite[Cor. 4.3.2.16, Thm. 5.5.3.18]{LurHTT}. Taking either $h$ or $f$ to be the identity, the values of (\ref{eq:IndCohindGStk}) on the correspondence $X \xleftarrow{h} Y \xrightarrow{f} Z$ define functors $h^*: \IndCoh(X) \to \IndCoh(Y)$ and $f_*: \IndCoh(Y) \to \IndCoh(Z)$. We have the following variant of Proposition \ref{prop:cohonindgstks}, which is proved the same way. It implies in particular that 
\begin{equation}%\label{eq:cohviaindgeopres2}
	\IndCoh(X) \cong \colim \IndCoh(X_\al)
\end{equation}
in $\PrL$, where $X \cong \colim X_\al$ is any ind-geometric presentation. Here we write $\indGStkkfcd \subset \indGStkk$ for the subcategory which only includes morphisms of ind-finite cohomological dimension, identifying it with a subcategory $\Corr(\indGStkk)_{fcd;ftd}$ as before. 

\begin{Proposition}\label{prop:indcohonindgstks}
	The restriction of $\IndCoh$ to $\indGStkkfcd$ preserves filtered colimits along ind-closed immersions. 
\end{Proposition} 

\begin{Remark}\label{rem:IndCohnaive}
Let $X \cong \colim X_\al$ be an ind-geometric presentation with each $X_\al$ an algebraic stack of finite type. Then in particular $X$ is locally almost of finite type as a prestack, hence \cite[Sec. 10]{Gai13a} applies to give an a priori different notion of $\IndCoh(X)$ than Definition~\ref{def:IndCohindGStk}. However, Proposition \ref{prop:IndCohisIndofCoh} and \cite[Thm. 3.3.5]{DG13} imply they are consistent. 

On the other hand, while \cite[Sec. 10]{Gai13a} defines $\IndCoh(X)$ for any prestack $X$ which is locally almost of finite type, it does not appear this template can be extended to infinite type in a way that produces the category we wish to consider. The most obvious extension would be as follows (to suppress convergence issues we restrict to classical prestacks). First consider the $\PrL$-valued functor that assigns to a morphism $f: X \to Y$ of finite-type classical affine schemes the functor $f^!: \IndCoh(Y) \to \IndCoh(X)$. We may left Kan extend this to all classical affine schemes and then right Kan extend to associate a category $\IndCoh_{naive}(X)$ to an arbitrary classical prestack $X$. 

However, this category disagrees with $\IndCoh(X)$ already when $X = BG_\cO$ (here $G_\cO$ is the jet group of a complex algebraic group $G$), in which case $\IndCoh_{naive}(BG_\cO) \cong \QCoh(BG_\cO)$. But $\IndCoh(BG_\cO) \ncong \QCoh(BG_\cO)$ (or at least, the natural functor is not an equivalence) since the structure sheaf is compact in the former but not the latter (by \cite[Prop. 9.1.5.3]{LurSAG}, since $G_\cO$ is of infinite cohomological dimension). In this case we can recover $\IndCoh(BG_\cO)$ from $\IndCoh_{naive}(BG_\cO)$ by anticompletion, but when $X$ is arbitrary we cannot make sense of this anticompletion because $\IndCoh_{naive}(X)$ does not have a natural t-structure (we also cannot make sense of $\Coh(X)$ for essentially the same reason). 
\end{Remark}

\subsection{$!$-pullback and t-structures}
If a morphism $f: X \to Y$ ind-geometric stacks is ind-proper (hence of ind-finite cohomological dimension by Proposition \ref{prop:fcdprops}), we write $f^!: \IndCoh(Y) \to \IndCoh(X)$ for the right adjoint of $f_*$. As with other functors, we write $f_{QC}^!$ and $f_{IC}^!$ when the meaning of $f^!$ is not otherwise clear from context. For ind-geometric $X$ we define a standard t-structure on $\IndCoh(X)$ in terms of these adjoints, following \cite[Sec. I.3.1.2]{GR17II}. 

\begin{Proposition}\label{prop:IndCohindGStktstructure}
	If $X$ is an ind-geometric stack, $\IndCoh(X)$ has a t-structure defined by
	$$\IndCoh(X)^{\geq 0}:= \begin{Bmatrix} \text{$\cF \in \IndCoh(X)$  such that  $i^!(\cF) \in \IndCoh(X')^{\geq 0}$ for any} \\ \text{  truncated geometric substack  $i: X' \to X$ }\end{Bmatrix}. $$ 
		This t-structure is accessible, compatible with filtered colimits, right complete, and left anticomplete. If $X \cong \colim_\al X_\al$ is an ind-geometric presentation, the functors $i_{\al*}$ are t-exact and induce equivalences
	\begin{equation}\label{eq:ICindGtstruct}
		\IndCoh(X)^{\leq 0} \cong \colim_\al \IndCoh(X_\al)^{\leq 0}, \quad \IndCoh(X)^{\geq 0} \cong \colim_\al \IndCoh(X_\al)^{\geq 0}
	\end{equation}
	in $\PrL$. 
\end{Proposition}

We will use the following standard result, see e.g. \cite[Cor. 5.1.5]{BKV22}. 

\begin{Lemma}\label{lem:colimpres}
	Let $\catC \cong \colim \catC_\al$ be the colimit of a diagram $A \to \PrL$, with $F_\al: \catC_\al \to \catC$ the canonical functors and $G_\al: \catC \to \catC_\al$ their right adjoints. 
	Then for any $X \in \catC$, the objects $X \cong \colim F_\al G_\al (X)$ assemble into a diagram whose colimit is $X$. 
\end{Lemma}

\begin{Lemma}\label{lem:tfiltered}
	Let $\PrSttexrcpl \subset \PrStbrcpl$ denote the subcategory which only includes t-exact functors. Then $\PrSttexrcpl$ admits filtered colimits, and these are preserved by the functors to $\PrL$ given by $\catC \mapsto \catC$, $\catC \mapsto \catC^{\leq 0}$, and $\catC \mapsto \catC^{\geq 0}$. Moreover, the subcategory $\PrSttexacpl \subset \PrSttexrcpl$ which only includes left anticomplete t-structures is closed under filtered colimits. 
\end{Lemma}
\begin{proof}
Let $\Grothinftylexrcpl$ denote the category of Grothendieck prestable $\infty$-categories and left-exact functors. By \cite[Rem. C.3.1.5, Prop. C.3.2.1]{LurSAG} $\catC \mapsto \Sp(\catC)$ and $\catC \mapsto \catC^{\leq 0}$ induce inverse equivalences of $\Grothinftylexrcpl$ and $\PrSttexrcpl$. The claims about $\catC \mapsto \catC$ and $\catC \mapsto \catC^{\leq 0}$ now follow since $\Grothinftylexrcpl$ is closed under filtered colimits in $\PrL$ \cite[Prop. C.3.3.5]{LurSAG}, and since $\catC \mapsto \Sp(\catC)$ preserves small colimits in $\PrL$ \cite[Ex. 4.8.1.22]{LurHA}. 

Let $\catC \cong \colim \catC_\al$ be a filtered colimit in $\PrSttexrcpl$. Since the structure functors $F_{\al\be}: \catC_\al \to \catC_\be$, $F_\al: \catC_\al \to \catC$ are t-exact their right adjoints are left t-exact. Then since $\catC \cong \lim \catC_\al$ in $\PrR$ and since the inclusions $\catC_\al^{\geq 0} \subset \catC_\al$ are morphisms in $\PrR$, they identify the limit of the $\catC_\al^{\geq 0}$ in $\PrR$ as the full subcategory of $X \in \catC$ such that $F_\al^R(X) \in \catC_\al^{\geq 0}$ for all $\al$. This is equivalent to $\Map_\catC(F_\al(Y_\al), X) \cong 0$ for all $\al$ and all $Y_\al \in \catC_\al^{<0}$. But this is equivalent to $\Map_\catC(Y, X) \cong 0$ for all $Y \in \catC^{<0}$, since by Lemma \ref{lem:colimpres} and the previous paragraph $\catC^{<0}$ is generated under small colimits by objects of the form $F_\al(Y_\al)$ with $Y_\al \in \catC_\al^{<0}$. It follows that $\catC^{\geq 0} \cong \colim \catC_\al^{\geq 0}$ in $\PrL$. Finally, it follows from \cite[Cor. C.5.5.10]{LurSAG} and the discussion above that $\PrSttexacpl$ is closed under all colimits that exist in $\PrSttexrcpl$. 
\end{proof}

\begin{proof}[Proof of Proposition \ref{prop:IndCohindGStktstructure}]
Let $\indGStkkcl \subset \indGStkk$ denote the subcategory which only includes ind-closed immersions, similarly for $\GStkkpluscl \subset \GStkkplus$. By construction the restriction of (\ref{eq:boundedIndCoh}) to $\GStkkpluscl$ factors through $\PrSttexacpl$, and we write $\IndCoh^t: \indGStkkcl \to \PrSttexacpl$ for its left Kan extension. Adapting again the proof of Proposition \ref{prop:cohonindgstks}, we find that $\IndCoh^t(X) \cong \colim \IndCoh^t(X_\al)$ in $\PrSttexacpl$. But by Proposition \ref{prop:indcohonindgstks} $\IndCoh(X)$ is the underlying category of $\IndCoh^t(X)$. The claims now follow from Lemma \ref{lem:tfiltered} and Proposition~\ref{prop:igprestermsarereas} (the Lemma ensures each $i_\al^!$ is left t-exact, the Proposition ensures $i^!$ is left t-exact for all $i: X' \to X$). 
\end{proof}

\begin{Remark}
Since $\IndCoh(X)$ is left anticomplete, it can be recovered functorially from $\IndCoh(X)^+$. By contrast, let $\QCoh'(X)$ denote the colimit of the categories $\QCoh(X_\al)$ in $\PrSttexrcpl$. Left completeness is stable under limits rather than colimits in $\PrSttexrcpl$, so $\QCoh'(X)$ will in general be neither left complete nor left anticomplete. In particular, $\QCoh'(X)$ is a wilder category in that it cannot be recovered functorially from its bounded below objects. 
\end{Remark}

\begin{Remark}\label{rem:classical}
Let $X$ be a classical ind-geometric stack and $X \cong \colim_\al X_\al$ an ind-geometric presentation by classical geometric stacks. By construction we have $\IndCoh(X_\al)^\heartsuit \cong \QCoh(X_\al)^\heartsuit$ for all $\al$, and by Proposition \ref{prop:IndCohindGStktstructure} we have $\IndCoh(X)^\heartsuit \cong \colim_\al \IndCoh(X_\al)^\heartsuit$ in $\PrL$. It follows from \cite[Cor. 10.4.6.8, Prop. C.5.5.20, Thm. C.5.8.8]{LurSAG} that $\IndCoh(X)$ is the dg nerve of the category of injective complexes in $\IndCoh(X)^\heartsuit$, a construction first studied in \cite{Kra05}. The category $\IndCoh(X)^\heartsuit$ is, in the case of ind-schemes, the category of $\cO^!$-modules considered in \cite[Sec. 7.11.3]{BD}.
\end{Remark}

Using t-structures we can address the potential ambiguity in the definition of $\IndCoh(X)$ when $X$ is a non-truncated geometric stack.

\begin{Proposition}\label{prop:ICconsistenctdef}
Given a non-truncated geometric stack $X$, the categories $\IndCoh(X)$ defined by Definitions \ref{def:IndCohononeGStk} and \ref{def:IndCohindGStk} are canonically equivalent, and this equivalence identifies the t-structure of the former with that of Proposition \ref{prop:IndCohindGStktstructure}. 
\end{Proposition}
\begin{proof}
Temporarily denote the two categories by $\IndCoh_{geom}(X)$ and $\IndCoh_{ind}(X)$. For each~$n$ the morphism $i_n: \tau_{\leq n} X \to X$ is affine, hence yields a t-exact functor $i_{n*}: \IndCoh(\tau_{\leq n} X) \to \IndCoh_{geom}(X)$. By Propositions \ref{prop:indcohonindgstks} and \ref{prop:IndCohindGStktstructure} it suffices to show the induced t-exact functor 
\begin{equation}\label{eq:indconsistent}
\IndCoh_{ind}(X) \cong \colim \IndCoh(\tau_{\leq n} X) \to \IndCoh_{geom}(X)
\end{equation}
is an equivalence. Now for all $a \leq b$ (\ref{eq:indconsistent}) restricts to an equivalence $\IndCoh_{ind}(X)^{[a,b]} \cong  \IndCoh_{geom}(X)^{[a,b]},$ since for all $n \geq b - a$ we have
$$ \IndCoh(\tau_{\leq n} X)^{[a,b]} \cong \QCoh(\tau_{\leq n} X)^{[a,b]} \cong \QCoh( X)^{[a,b]} \cong \IndCoh(X)_{geom}^{[a,b]}. $$
By right completeness of the two t-structures it follows that (\ref{eq:indconsistent}) restricts to an equivalence $\IndCoh_{ind}(X)^{\geq a} \cong \IndCoh_{geom}(X)^{\geq a}$ for any $a$. But then (\ref{eq:indconsistent}) induces an equivalence of left completions, hence by Proposition~\ref{prop:acplcplequiv} is itself an equivalence since its source and target are left anticomplete. 
\end{proof}

\subsection{Relation to coherent sheaves}\label{sec:cohrelation} We now turn to the relationship between $\Coh(X)$ and $\IndCoh(X)$ when $X$ is reasonable. 

\begin{Proposition}\label{prop:CohsubcatofIndCoh}
	For any reasonable ind-geometric stack $X$ there is a fully faithful functor $\Coh(X) \into \IndCoh(X)$. It is induced from a canonical natural transformation between the functors $\Corr(\indGStkkreas)_{prop;ftd} \to \Cathatinfty$ obtained from (\ref{eq:cohcorrindgstkfunctor}) and (\ref{eq:IndCohindGStk}). In particular, we have a diagram 
	\begin{equation*}
		\begin{tikzpicture}
			[baseline=(current  bounding  box.center),thick,>=\arrtip]
			\node (a) at (0,0) {$\Coh(X')$};
			\node (b) at (3.2,0) {$\IndCoh(X')$};
			\node (c) at (0,-1.5) {$\Coh(X)$};
			\node (d) at (3.2,-1.5) {$\IndCoh(X),$};
			\draw[right hook->] (a) to node[above] {$ $} (b);
			\draw[->] (b) to node[right] {$i_* $} (d);
			\draw[->] (a) to node[left] {$i_* $}(c);
			\draw[right hook->] (c) to node[above] {$ $} (d);
		\end{tikzpicture}
	\end{equation*}
	for any reasonable geometric substack $i: X' \to X$. 
\end{Proposition}
\begin{proof}
	By construction the inclusion $\Coh(X) \subset \IndCoh(X)$ for geometric $X$ enhances to a natural transformation of functors $\Corr(\GStkkplus)_{prop;ftd} \to \Cathatinfty$. 
	The variant of (\ref{eq:cohcorrindgstkfunctor}) appearing in the statement is the left Kan extension of its restriction to  $\Corr(\GStkkplus)_{prop;ftd}$, since following the proof of Proposition \ref{prop:cohonindgstks} the restrictions of both to $\indGStkkprop$ are left Kan extended from $\GStkkplusprop$.
	The desired natural transformation and the pictured diagram then follow from the characteristic adjunction of left Kan extensions \cite[Prop. 4.3.2.17]{LurHTT}. 
	
	Given $\cG \in \Coh(X)$ we can write $\cG \cong i_{\al*}(\cG_\al)$ for some $\cG_\al \in \Coh(X_\al)$, increasing $\al$ if needed. Identifying $\cF$, $\cG$ with their images in $\IndCoh(X)$ we then have 
	\begin{align*}
		\Map_{IC(X)}(\cF, \cG) &\cong \Map_{IC(X_\al)}(\cF_\al, i_\al^! i_{\al*}(\cG_\al))\\
		&\cong \Map_{IC(X_\al)}(\cF_\al, \colim_{\be \geq \al} i_{\al\be}^! i_{\al\be*}(\cG_\al))\\
		&\cong \colim_{\be \geq \al} \Map_{IC(X_\al)}(\cF_\al,  i_{\al\be}^! i_{\al\be*}(\cG_\al))\\
		&\cong \colim_{\be \geq \al} \Map_{IC(X_\be)}(i_{\al\be*}(\cF_\al),   i_{\al\be*}(\cG_\al)).
	\end{align*} 
	Here the second isomorphism follows from Proposition \ref{prop:indcohonindgstks} and \cite[Lemma 1.1.10]{GR17II}, and the third follows since $\cF_\al$ and $\cG_\al$ are coherent and each $i_{\al\be}^! i_{\al\be*}$ is left t-exact. But since $\Coh(X_\be)$ is a full subcategory of $\IndCoh(X_\be)$ for all $\be \geq \al$, the last expression is equivalent to $\Map_{\Coh(X)}(\cF, \cG)$
	\cite[Lem. 0.2.1]{Roz}. 
\end{proof}

Recall from \cite[Prop. 9.1.5.1]{LurSAG} that when $X$ is geometric, $\Coh(X)$ can characterized as the full subcategory of bounded, almost compact objects in $\QCoh(X)$ (i.e. of $\cF \in \QCoh(X)^+$ such that $\tau^{\geq n}\cF$ is compact in $\QCoh(X)^{\geq n}$ for all $n$). The equivalence $\IndCoh(X)^+ \cong \QCoh(X)^+$ thus also identifies $\Coh(X)$ with the full subcategory of bounded, almost compact objects in $\IndCoh(X)$. We will see in Proposition \ref{prop:cohisacinindcoh} that, when $X$ is reasonable, the image of $\Coh(X)$ in $\IndCoh(X)$ still consists of bounded, almost compact objects. When $X$ is coherent, the two categories in fact determine one another. 

\begin{Proposition}\label{prop:IndCohisIndofCoh}
	If $X$ is a coherent ind-geometric stack then the canonical functor $\Ind(\Coh(X)) \to \IndCoh(X)$ is an equivalence, and induces equivalences $\Ind(\Coh(X)^{\leq 0}) \cong \IndCoh(X)^{\leq 0}$, $\Ind(\Coh(X)^{\geq 0}) \cong \IndCoh(X)^{\geq 0}$. 
\end{Proposition}
\begin{proof}
	For a coherent geometric stack this follows from \cite[Thm. C.6.7.1]{LurSAG}. The ind-geometric case then follows from Propositions \ref{prop:cohonindgstks}, \ref{prop:indcohonindgstks}, and \ref{prop:CohsubcatofIndCoh}, since ind-completion of idempotent-complete categories admitting finite colimits commutes with filtered colimits \cite[Lem. 7.3.5.11]{LurHA}. 
\end{proof}

\begin{Remark}
	In light of Proposition \ref{prop:IndCohisIndofCoh}, the notion of coherent ind-geometric stack is in a sense formally dual to the notion of perfect stack considered in \cite{BZFN10}.
\end{Remark}

\subsection{Adjunction of pushforward and pullback}

Suppose $f: X \to Y$ is both of finite Tor-dimension and ind-finite cohomological dimension. Then we have separately defined functors $f_{IC*}$ and $f^*_{IC}$, but the definition does not explicitly entail any direct relationship between them. Nonetheless, the two functors are adjoint in the expected way. 

\begin{Proposition}\label{prop:up*low*adj}
	Let $X$ and $Y$ be ind-geometric stacks and $f: X \to Y$ a morphism which is both of finite Tor-dimension and of ind-finite cohomological dimension. Then $f_{IC*}$ is right adjoint to $f^*_{IC}$. 
\end{Proposition}

\begin{Lemma}\label{lem:acplcpladjunction}
	Let $\wh{\catC}, \wh{\catD}$ be the left completions of $\wc{\catC}, \wc{\catD} \in \PrStbacpl$. Let $\wc{F}: \wc{\catC} \to \wc{\catD}$ and $\wc{G}: \wc{\catD} \to \wc{\catC}$ be bounded colimit-preserving functors, and let $\wh{F}: \wh{\catC} \to \wh{\catD}$ and $\wh{G}: \wh{\catD} \to \wh{\catC}$ be their images under the equivalences $\LFun^b(\wc{\catC}, \wc{\catD}) \cong \LFun^b(\wh{\catC}, \wh{\catD})$ and $\LFun^b(\wc{\catD}, \wc{\catC}) \cong \LFun^b(\wh{\catD}, \wh{\catC})$. Then $\wc{G}$ is right adjoint to $\wc{F}$ if and only if $\wh{G}$ is right adjoint to $\wh{F}$. If this is the case, the equivalences $\LFun^b(\wc{\catC}, \wc{\catC}) \cong \LFun^b(\wh{\catC}, \wh{\catC})$ and $\LFun^b(\wc{\catD}, \wc{\catD}) \cong \LFun^b(\wh{\catD}, \wh{\catD})$ together identify pairs of a compatible unit and counit for the two adjunctions. 
\end{Lemma}
\begin{proof}
	We consider the if direction, the other being symmetric. Recall that $\wh{G}$ being right adjoint to $\wh{F}$ is equivalent to the existence of unit and counit transformations $\wh{u}: \id_{\whD} \to \wh{G} \wh{F}$, $\wh{\ep}: \wh{F} \wh{G} \to \id_{\whC}$ and diagrams 
	\begin{equation*}
		\begin{tikzpicture}
			[baseline=(current  bounding  box.center),thick,>=\arrtip]
			\newcommand*{\ha}{1.5};
			\newcommand*{\va}{1.5};
			
			\node [matrix] at (0,0) {
				\node (ba) at (0,0) {$\wh{F} $};
				\node (ab) at (\ha,\va) {$\wh{F} \wh{G} \wh{F}$};
				\node (bc) at (\ha+\ha,0) {$\wh{F}$};
				\draw[->] (ba) to node[above left] {$\id_{\wh{F}} \cdot \wh{u} $} (ab);
				\draw[->] (ab) to node[above right] {$\wh{\ep} \cdot \id_{\wh{F}} $} (bc);
				\draw[->] (ba) to node[below] {$\id_{\wh{F}} $} (bc);\\
			};
			
			\node [matrix] at (6,0) {
				\node (ba) at (0,0) {$\wh{G} $};
				\node (ab) at (\ha,\va) {$\wh{G} \wh{F} \wh{G}$};
				\node (bc) at (\ha+\ha,0) {$\wh{G}$.};
				\draw[->] (ba) to node[above left] {$\wh{u} \cdot \id_{\wh{G}} $} (ab);
				\draw[->] (ab) to node[above right] {$  \id_{\wh{G}} \cdot \wh{\ep} $} (bc);
				\draw[->] (ba) to node[below] {$\id_{\wh{G}} $} (bc);\\
			};
		\end{tikzpicture}
	\end{equation*} 
	This is a reformulation of \cite[Def. 2.1.1]{RV22}, which is equivalent to \cite[Def. 5.2.2.1]{LurHTT} by \cite[Prop. F.5.6]{RV22}. 
	
	By hypotheses $\wh{u}$ and $\wh{\ep}$ are morphisms in $\LFun^b(\whD,\whD)$ and $\LFun^b(\whC,\whC)$, and we write $\wc{u}$ and $\wc{\ep}$ for the corresponding morphisms in $\LFun^b(\wcD,\wcD)$ and $\LFun^b(\wcC,\wcC)$. The above diagrams are respectively in $\LFun^b(\whC,\whD)$ and $\LFun^b(\whD,\whC)$, and we claim the corresponding diagrams in $\LFun^b(\wcC,\wcD)$ and $\LFun^b(\wcD,\wcC)$ witness $\wc{u}$ and $\wc{\ep}$ as the unit and counit of an adjunction between $\wc{F}$ and $\wc{G}$. In other words, we claim the equivalence $\LFun^b(\whC,\whD) \cong \LFun^b(\wcC,\wcD)$ takes $\wh{F} \wh{G} \wh{F}$, $\id_{\wh{F}} \cdot \wh{u}$, and $\wh{\ep} \cdot \id_{\wh{F}} $ respectively to $\wc{F} \wc{G} \wc{F}$, $\id_{\wc{F}} \cdot \wc{u}$, and $\wc{\ep} \cdot \id_{\wc{F}} $, similarly for the right diagram.  
	
	Write $\Psi_\catC: \wcC \to \whC$ and $\Psi_\catD: \wcD \to \whD$ for the canonical functors. Then $\wc{F}$ is characterized in $\LFun^b(\wcC, \wcD)$ by the condition $\Psi_\catD \wc{F} \cong \wh{F} \Psi_\catC$, similarly for $\wc{G}$. It follows that $\Psi_\catD \wc{F} \wc{G} \wc{F} \cong \wh{F} \wh{G} \wh{F} \Psi_{\catC}$, which likewise characterizes $\wc{F} \wc{G} \wc{F}$ as the functor corresponding to $\wh{F} \wh{G} \wh{F}$. Now consider the following diagram, in which all horizontal arrows equivalences.
	\begin{equation*}
		\begin{tikzpicture}
			[baseline=(current  bounding  box.center),thick,>=\arrtip]
			\newcommand*{\ha}{4.0}; \newcommand*{\hb}{4.0};
			\newcommand*{\va}{-1.5};
			\node (aa) at (0,0) {$\LFun^b(\whC,\whC)$};
			\node (ab) at (\ha,0) {$\LFun^b(\wcC,\whC)$};
			\node (ac) at (\ha+\hb,0) {$\LFun^b(\wcC,\wcC)$};
			\node (ba) at (0,\va) {$\LFun^b(\whC,\whD)$};
			\node (bb) at (\ha,\va) {$\LFun^b(\wcC,\whD)$};
			\node (bc) at (\ha+\hb,\va) {$\LFun^b(\wcC,\wcD)$};
			\draw[->] (aa) to node[above] {$-  \Psi_\catC $} (ab);
			\draw[<-] (ab) to node[above] {$\Psi_\catC  - $} (ac);
			\draw[->] (ba) to node[above] {$-  \Psi_\catC $} (bb);
			\draw[<-] (bb) to node[above] {$\Psi_\catD - $} (bc);
			\draw[->] (aa) to node[left] {$\wh{F} - $} (ba);
			\draw[->] (ab) to node[right] {$\wh{F} - $} (bb);
			\draw[->] (ac) to node[right] {$\wc{F} - $} (bc);
		\end{tikzpicture}
	\end{equation*}
	Here the compositions around the left square are evidently isomorphic, while those around the right square are because of the isomorphism $\Psi_\catD \wc{F} \cong \wh{F} \Psi_\catC$. The morphism $\id_{\wh{F}} \cdot \wh{u}$ is the image of $\wh{u}$ under the left vertical map, while $\id_{\wc{F}} \cdot \wc{u}$ is the image of $\wc{u}$ under the right. But by definition $\wc{u}$ is the image of $\wh{u}$ under the overall top equivalence, hence $\id_{\wc{F}} \cdot \wc{u}$ is the image of $\id_{\wh{F}} \cdot \wh{u}$  under the overall bottom equivalence. The remaining conditions are checked the same way. 
\end{proof}

\begin{proof}[Proof of Proposition \ref{prop:up*low*adj}]
	We let subscripts be implicit in the proof, $f_*$ always meaning $f_{IC*}$, etc. If $X$ and $Y$ are truncated and geometric, the claim follows immediately from Lemma~\ref{lem:acplcpladjunction}. In general, let $Y \cong \colim Y_\al$ be an ind-geometric presentation and $f_\al: X_\al \to Y_\al$ the base change of $f$. We then have an ind-geometric presentation $X \cong \colim X_\al$ since $f$ is of finite Tor-dimension and filtered colimits in $\Stkkconv$ are left exact. Let $A$ denote the index category and $\Delta^1_{\al \be} \subset A$ the morphism associated to $\al \leq \be$. 
		
	By construction we have a functor $\Delta^1 \times A \to \PrL$ taking $\Delta^1 \times \Delta^1_{\al \be}$ to the diagram witnessing the isomorphism $i_{\al \be*}f_{\al*} \cong f_{\be *} i'_{\al \be *}$, and a similar functor $(\Delta^1)^{\op} \times A \to \PrL$ packaging the isomorphisms $f_{\be}^* i_{\al \be*} \cong  i'_{\al \be *} f_{\al}^*$. Passing to right adjoints, these are equivalent to the data of a functor $A^{\op} \to \Fun((\Delta^1)^\op, \Cathatinfty)$ taking $\al$ to $f_{\al*}^R$ and a functor $A^{\op} \to \Fun(\Delta^1, \Cathatinfty)$ taking $\al$ to $f_{\al}^{*R}$. 
	
	Note that the unit/counit compatibility of Lemma \ref{lem:acplcpladjunction} implies more precisely that the isomorphism $f_{\be}^* i_{\al \be*} \cong  i'_{\al \be *} f_{\al}^*$ is the Beck-Chevalley transformation associated to the isomorphism $i_{\al \be*}f_{\al*} \cong f_{\be *} i'_{\al \be *}$. In the notation of \cite[Def. 4.7.5.16]{LurHA}, the above functors thus take values in $\Fun^{\RAd}((\Delta^1)^\op, \Cathatinfty)$ and  $\Fun^{\LAd}(\Delta^1, \Cathatinfty)$, respectively, and correspond under the equivalence  $\Fun^{\RAd}((\Delta^1)^\op, \Cathatinfty) \cong \Fun^{\LAd}(\Delta^1, \Cathatinfty)$ of \cite[Cor. 4.7.5.18]{LurHA}. By the same result these subcategories are closed under limits in $\Fun((\Delta^1)^\op, \Cathatinfty)$ and $\Fun(\Delta^1, \Cathatinfty)$. But by Proposition \ref{prop:indcohonindgstks} and \cite[Cor. 5.1.2.3]{LurHTT} we have $f_{*}^R \cong \lim f_{\al*}^R$ and $f^{*R} \cong \lim f_{\al}^{*R}$, hence $f_{*}^R$ is right adjoint to $f^{*R}$, hence $f_{*}$ is right adjoint to $f^{*}$. 
\end{proof}

The definition of ind-coherent pushforward and $*$-pullback also defines base change isomorphisms for suitable Cartesian squares. These isomorphisms are compatible with the adjunction of Proposition \ref{prop:up*low*adj} in the following sense. 

\begin{Proposition}
	Let the following be a Cartesian diagram of ind-geometric stacks in which $f$ is both of finite Tor-dimension and of ind-finite cohomological dimension. 
	\begin{equation*}
		\begin{tikzpicture}
			[baseline=(current  bounding  box.center),thick,>=\arrtip]
			\node (a) at (0,0) {$X'$};
			\node (b) at (3,0) {$Y'$};
			\node (c) at (0,-1.5) {$X$};
			\node (d) at (3,-1.5) {$Y$};
			\draw[->] (a) to node[above] {$f' $} (b);
			\draw[->] (b) to node[right] {$h $} (d);
			\draw[->] (a) to node[left] {$h' $}(c);
			\draw[->] (c) to node[above] {$f $} (d);
		\end{tikzpicture}
	\end{equation*}
	If $h$ is of finite Tor-dimension (resp. of  ind-finite cohomological dimension), the isomorphism $f^* h_* \cong h'_* f'^*$ of functors $\IndCoh(Y') \to \IndCoh(X)$ is the Beck-Chevalley transformation associated to the isomorphism $ h'^* f^* \cong  f'^* h^*$ of functors $\IndCoh(Y) \to \IndCoh(X')$ (resp. the isomorphism $f_* h'_* \cong h_* f'_*$ of functors $\IndCoh(X') \to \IndCoh(Y)$). 
\end{Proposition}

\begin{proof}
	Consider the case with $h$ of ind-finite cohomological dimension, the finite Tor-dimension case following by a variation of the same argument. If $X$, $Y$, and $Y'$ are truncated and geometric the claim follows immediately from Lemma \ref{lem:acplcpladjunction}. If $h$ is the immersion of an ind-geometric substack the claim was established during the proof of Proposition \ref{prop:up*low*adj}. 
	
	We pass to right adjoints and identify the isomorphisms $f^{*R} h_*^R \cong h'^R_* f'^{*R}$ and $f_*^R h'^R_* \cong h^R_* f'^R_*$ respectively with a morphism $f^{*R} \to f'^{*R}$ in $\Fun(\Delta^1, \Cathatinfty)$ and a morphism $f_*^{R} \to f'^{R}_*$ in $\Fun((\Delta^1)^{\op}, \Cathatinfty)$, performing such identifications without comment in the rest of the proof. In the notation of \cite[Def. 4.7.5.16]{LurHA}, we want to show these belong to $\Fun^{\LAd}(\Delta^1, \Cathatinfty)$ and $\Fun^{\RAd}((\Delta^1)^{\op}, \Cathatinfty)$, respectively, and correspond to each other under the equivalence $\Fun^{\LAd}(\Delta^1, \Cathatinfty) \cong \Fun^{\RAd}((\Delta^1)^{\op}, \Cathatinfty)$ of \cite[Cor. 4.7.5.18]{LurHA}. 
	
	First suppose $X$ and $Y$ are truncated and geometric, let $Y' \cong \colim Y'_\al$ be an ind-geometric presentation, and write $f'_\al: X'_\al \to Y'_\al$ for the base change of $f'$. As in the proof of Proposition~\ref{prop:up*low*adj}, we have $f'^{*R} \cong \lim f'^{*R}_\al$ in both $\Fun^{\LAd}(\Delta^1, \Cathatinfty)$ and $\Fun(\Delta^1, \Cathatinfty)$, hence $f^{*R} \to f'^{*R}$ is in $\Fun^{\LAd}(\Delta^1, \Cathatinfty)$ since each $f^{*R} \to f'^{*R}_\al$ is. Similarly, $f_*^{R} \to f'^{R}_{*} \cong \lim f'^{R}_{\al*}$ is in $\Fun^{\RAd}((\Delta^1)^{\op}, \Cathatinfty)$ since each $f_*^{R} \to f'^{R}_{\al*}$ is, and it corresponds to $f^{*R} \to f'^{*R}$ under \cite[Cor. 4.7.5.18]{LurHA} since each $f_*^{R} \to f'^{R}_{\al*}$ corresponds to $f^{*R} \to f'^{*R}_\al$. 
	
Now let $Y \cong \colim Y_\al$ be an ind-geometric presentation. For any $\al$ we have a diagram
\begin{equation*}
	\begin{tikzpicture}[baseline=(current  bounding  box.center),thick,>=\arrtip]
			\newcommand*{\ha}{1.5}; \newcommand*{\hb}{1.5}; \newcommand*{\hc}{1.5};
\newcommand*{\va}{-.9}; \newcommand*{\vb}{-.9}; \newcommand*{\vc}{-.9}; 
		\node (ab) at (\ha,0) {$X'_\al$};
		\node (ad) at (\ha+\hb+\hc,0) {$Y'_\al$};
		\node (ba) at (0,\va) {$X'$};
		\node (bc) at (\ha+\hb,\va) {$Y'$};
		\node (cb) at (\ha,\va+\vb) {$X_\al$};
		\node (cd) at (\ha+\hb+\hc,\va+\vb) {$Y_\al$};
		\node (da) at (0,\va+\vb+\vc) {$X$};
		\node (dc) at (\ha+\hb,\va+\vb+\vc) {$Y$};
		%\draw[->] (ab) to node[above] {$\phi' $} (ad);
		\draw[->] (ab) to node[above] {$f'_\al $} (ad);
		%\draw[->] (ab) to node[above left] {$\theta' $} (ba);
		\draw[->] (ab) to node[above left, pos=.25] {$j'_{\al} $} (ba);
		%\draw[->] (ab) to node[left,pos=.8] {$\psi' $} (cb);
		\draw[->] (ab) to node[right,pos=.2] {$h'_\al $} (cb);
		\draw[->] (ad) to node[below right] {$j_\al $} (bc);
		\draw[->] (ad) to node[right] {$h_\al $} (cd);
		%\draw[->] (ba) to node[left] {$h' $} (da);
		\draw[->] (ba) to node[left] {$ $} (da);
		\draw[->] (cb) to node[above,pos=.25] {$f_\al $} (cd);
		%\draw[->] (cb) to node[above left] {$g' $} (da);
		\draw[->] (cb) to node[above left, pos=.25] {$i'_\al $} (da);
		\draw[->] (cd) to node[below right] {$i_\al $} (dc);
		\draw[->] (da) to node[above,pos=.75] {$ $} (dc);
		
		\draw[-,line width=6pt,draw=white] (ba) to  (bc);
		%\draw[->] (ba) to node[above,pos=.75] {$f' $} (bc);
		\draw[->] (ba) to node[above,pos=.75] {$ $} (bc);
		\draw[-,line width=6pt,draw=white] (bc) to  (dc);
		\draw[->] (bc) to node[right,pos=.2] {$ $} (dc);
	\end{tikzpicture}
\end{equation*}
with Cartesian faces. We have already shown that the compositions $f^{*R} \to f^{*R}_\al \to f'^{*R}_\al$ and $f^{R}_* \to f^{R}_{\al*} \to f'^{R}_{\al*}$ belong to $\Fun^{\LAd}(\Delta^1, \Cathatinfty)$ and $\Fun^{\RAd}((\Delta^1)^{\op}, \Cathatinfty)$, and correspond under \cite[Cor. 4.7.5.18]{LurHA}. By closure of these subcategories under limits, and by Proposition \ref{prop:indcohonindgstks}, it suffices to show $f'^{*R}_\be \to f'^{*R}_\al$ and $f'^{R}_{\be*} \to f'^{R}_{\al*}$ belong to $\Fun^{\LAd}(\Delta^1, \Cathatinfty)$ and $\Fun^{\RAd}((\Delta^1)^{\op}, \Cathatinfty)$ and correspond under \cite[Cor. 4.7.5.18]{LurHA}. 

Consider the following diagram in $\Fun(\IndCoh(X_\be), \IndCoh(Y'_\al))$.
\begin{equation*}
	\begin{tikzpicture}
		[baseline=(current  bounding  box.center),thick,>=\arrtip]
		\newcommand*{\ha}{3.5}; \newcommand*{\hb}{3.5};
		\newcommand*{\va}{-1.5};
		\node (aa) at (0,0) {$ j^R_{\al \be *} h^R_{\be*} f^{*R}_\be $};
		\node (ab) at (\ha,0) {$  j^R_{\al \be *}  f'^{*R}_\be h'^R_{\be*}$};
		\node (ac) at (\ha+\hb,0) {$  f'^{*R}_\al   j'^R_{\al \be *} h'^R_{\be*} $};
		\node (ba) at (0,\va) {$ h_{\al*}^R   i^R_{\al \be *} f^{*R}_\be$};
		\node (bb) at (\ha,\va) {$  h_{\al*}^R  f^{*R}_\al i'^R_{\al \be *}$};
		\node (bc) at (\ha+\hb,\va) {$  f'^{*R}_\al    h'^R_{\al*} i'^R_{\al \be *}$};
		\draw[->] (aa) to node[above] {$\sim $} (ab);
		\draw[->] (ab) to node[above] {$\sim $} (ac);
		\draw[->] (ba) to node[above] {$\sim $} (bb);
		\draw[->] (bb) to node[above] {$\sim $} (bc);
		\draw[->] (aa) to node[below,rotate=90] {$\sim $} (ba);
		%\draw[->] (ab) to node[right] {$h' $} (bb);
		\draw[->] (ac) to node[below,rotate=90] {$\sim $} (bc);
	\end{tikzpicture} 
\end{equation*}
We have already shown $f'^{*R}_\al$ and $f'^{R}_{\al*}$, etc., are adjoint, and the claim at hand is equivalent to the top right arrow being the Beck-Chevalley transformation associated to the isomorphism $j'^R_{\al\be *} f'^R_{\be*} \cong f'^R_{\al*} j^R_{\al \be *}$. But this follows since we have shown the corresponding claim for the top left and bottom arrows, and since all arrows in the diagram are isomorphisms. 
\end{proof}

\section{Ind-proper $!$-pullback}\label{sec:upper!}

In this section we further study the functor $f^!: \IndCoh(Y) \to \IndCoh(X)$ associated to an ind-proper morphism $f: X \to Y$ of ind-geometric stacks. We first establish the almost continuity of $f^!$ in the almost ind-finitely presented case (Proposition \ref{prop:upper!almostcontind}). This is then used to extend the compatibility between pushforward and $!$-pullback to the ind-geometric setting (Proposition \ref{prop:up!low*geom}). 

\subsection{Almost continuity} 

We begin with the geometric case, where we have the following generalizations of \cite[Lem. 6.4.1.5, Prop. 6.4.1.4]{LurSAG}. Recall that $f^!$ being almost continuous means its restriction to $\QCoh(Y)^{\geq n}$ is continuous for all $n$. 

\begin{Proposition}\label{prop:upper!almostcont}
	If  $f: X \to Y$ is a proper, almost finitely presented morphism of geometric stacks, then $f^!: \QCoh(Y) \to \QCoh(X)$ is almost continuous. 
\end{Proposition} 

\begin{Proposition}\label{prop:upper!ftdbasechange}
	Let the following be a Cartesian diagram of geometric stacks.
	\begin{equation*}
		\begin{tikzpicture}
			[baseline=(current  bounding  box.center),thick,>=\arrtip]
			\node (a) at (0,0) {$X'$};
			\node (b) at (3,0) {$Y'$};
			\node (c) at (0,-1.5) {$X$};
			\node (d) at (3,-1.5) {$Y$};
			\draw[->] (a) to node[above] {$f' $} (b);
			\draw[->] (b) to node[right] {$h $} (d);
			\draw[->] (a) to node[left] {$h' $}(c);
			\draw[->] (c) to node[above] {$f $} (d);
		\end{tikzpicture}
	\end{equation*}
	If $h$ is of finite Tor-dimension and $f$ is proper and almost finitely presented, then the Beck-Chevalley map $h'^* f^!(\cF) \to f'^! h^*(\cF)$ in $\QCoh(X')$ is an isomorphism for all $\cF \in \QCoh(Y)^+$. 
\end{Proposition}

\begin{Lemma}\label{lem:upper!flatcoverbasechange}
	Proposition \ref{prop:upper!ftdbasechange} is true when $Y'$ is affine and $h$ is faithfully flat. 
\end{Lemma}
\begin{proof}
	Let $Y_\bul$ denote the Cech nerve of $h$ (so $Y_0 = Y'$) and $f_k: X_k \to Y_k$ the base change of $f$. 
	Given a morphism $p: i \to j$ in $\Delta_s$, let $h_p: Y_j \to Y_i$ denote the associated map and $h'_p: X_j \to X_i$ its base change. By Proposition \ref{prop:propprops} we can choose $n$ so that $f_*$ takes $\QCoh(X)^{\leq 0}$ to $\QCoh(Y)^{\leq n}$, hence so does each $f_{k}$ since $Y_k$ is affine over $Y$. Then $\tau^{\geq 0} \circ f_*$ and $f^!$ restrict to an adjunction between $\QCoh(X^{\geq -n})$ and $\QCoh(Y)^{\geq 0}$, likewise for each $f_k$. 	The categories $\QCoh(X_k)^{\geq -n}$ and $\QCoh(Y_k)^{\geq 0}$, together with the functors $h^*_p$, $h'^*_p$, and $\tau^{\geq 0} \circ f_{k*}$, form a diagram $\Delta^1 \times \Delta_s \to \Cathatinfty$. By \cite[Prop. 6.4.1.4]{LurSAG} the Beck-Chevalley transformation $h'^*_p f^!_i \to f'^!_j h^*_p$ restricts to an isomorphism of functors $\QCoh(Y_i)^{\geq 0} \to \QCoh(X_j)^{\geq -n}$ for any $p$. Since $h$ is faithfully flat we have $\QCoh(X)^{\geq -n} \cong \lim_{\Delta_s} \QCoh(X_i)^{\geq -n}$ and $\QCoh(Y)^{\geq 0} \cong \lim_{\Delta_s} \QCoh(Y_i)^{\geq 0}$, and the claim follows from \cite[Cor. 4.7.5.18]{LurHA}. 
\end{proof}

\begin{proof}[Proof of Proposition \ref{prop:upper!almostcont}]
	Let $\cF \cong \colim \cF_\al$ be a filtered colimit in $\QCoh(Y)^{\geq 0}$, let $h: X' \cong \Spec A \to Y$ be a flat cover, and define $f': X' \to Y'$, $h': X' \to X$ by base change. Since $h'^*$ is continuous and conservative it suffices to show $\colim h'^* f^!(\cF_\al) \to h'^* f^!(\cF)$ is an isomorphism. By Lemma~\ref{lem:upper!flatcoverbasechange} and left boundedness of $f^!$ this is equivalent to $\colim  f'^! h^*(\cF_\al) \to f'^! h^*(\cF)$ being an isomorphism. Since $h^*$ is  t-exact this follows from Lemma \ref{lem:lower*almostcontaffine}. 
\end{proof}

\begin{proof}[Proof of Proposition \ref{prop:upper!ftdbasechange}]
	Let $\phi: U \cong \Spec A \to Y$ and $\theta: U' \cong \Spec A' \to U \times_Y Y'$ be flat covers. We obtain a diagram 
\begin{equation*}%\label{eq:up!compcube}
	\begin{tikzpicture}[baseline=(current  bounding  box.center),thick,>=\arrtip]
			\newcommand*{\ha}{1.5}; \newcommand*{\hb}{1.5}; \newcommand*{\hc}{1.5};
\newcommand*{\va}{-.9}; \newcommand*{\vb}{-.9}; \newcommand*{\vc}{-.9}; 
		\node (ab) at (\ha,0) {$Z'$};
		\node (ad) at (\ha+\hb+\hc,0) {$U'$};
		\node (ba) at (0,\va) {$X'$};
		\node (bc) at (\ha+\hb,\va) {$Y'$};
		\node (cb) at (\ha,\va+\vb) {$Z$};
		\node (cd) at (\ha+\hb+\hc,\va+\vb) {$U$};
		\node (da) at (0,\va+\vb+\vc) {$X$};
		\node (dc) at (\ha+\hb,\va+\vb+\vc) {$Y$};
		%\draw[->] (ab) to node[above] {$\phi' $} (ad);
		\draw[->] (ab) to node[above] {$g' $} (ad);
		%\draw[->] (ab) to node[above left] {$\theta' $} (ba);
		\draw[->] (ab) to node[above left, pos=.25] {$\psi' $} (ba);
		%\draw[->] (ab) to node[left,pos=.8] {$\psi' $} (cb);
		\draw[->] (ab) to node[right,pos=.2] {$\xi' $} (cb);
		\draw[->] (ad) to node[below right] {$\psi $} (bc);
		\draw[->] (ad) to node[right] {$\xi $} (cd);
		%\draw[->] (ba) to node[left] {$h' $} (da);
		\draw[->] (ba) to node[left] {$ $} (da);
		\draw[->] (cb) to node[above,pos=.25] {$g $} (cd);
		%\draw[->] (cb) to node[above left] {$g' $} (da);
		\draw[->] (cb) to node[above left, pos=.25] {$\phi' $} (da);
		\draw[->] (cd) to node[below right] {$\phi $} (dc);
		\draw[->] (da) to node[above,pos=.75] {$ $} (dc);
		
		\draw[-,line width=6pt,draw=white] (ba) to  (bc);
		%\draw[->] (ba) to node[above,pos=.75] {$f' $} (bc);
		\draw[->] (ba) to node[above,pos=.75] {$ $} (bc);
		\draw[-,line width=6pt,draw=white] (bc) to  (dc);
		\draw[->] (bc) to node[right,pos=.2] {$ $} (dc);
	\end{tikzpicture}
\end{equation*}
in which all but the left and right faces are Cartesian. Note that $\psi$ is faithfully flat and $\xi$ is of finite Tor-dimension, since they are the compositions of $\theta$ with the base changes of $\phi$ and $h$, respectively. Since $\psi'^*$ is conservative, it suffices to show the top left arrow in
\begin{equation*}
	\begin{tikzpicture}
		[baseline=(current  bounding  box.center),thick,>=\arrtip]
		\newcommand*{\ha}{3.5}; \newcommand*{\hb}{3.5};
		\newcommand*{\va}{-1.5};
		\node (aa) at (0,0) {$\psi'^* h'^* f^!(\cF)$};
		\node (ab) at (\ha,0) {$\psi'^* f'^! h^*(\cF)$};
		\node (ac) at (\ha+\hb,0) {$g'^!\psi^* h^*(\cF)$};
		\node (ba) at (0,\va) {$\xi'^* \phi'^*f^!(\cF)$};
		\node (bb) at (\ha,\va) {$\xi'^* g^! \phi^*(\cF)$};
		\node (bc) at (\ha+\hb,\va) {$g'^! \xi^* \phi^*(\cF)$};
		\draw[->] (aa) to node[above] {$ $} (ab);
		\draw[->] (ab) to node[above] {$  $} (ac);
		\draw[->] (ba) to node[above] {$ $} (bb);
		\draw[->] (bb) to node[above] {$ $} (bc);
		\draw[->] (aa) to node[below,rotate=90] {$\sim $} (ba);
		%\draw[->] (ab) to node[right] {$h' $} (bb);
		\draw[->] (ac) to node[below,rotate=90] {$\sim $} (bc);
	\end{tikzpicture} 
\end{equation*}
is an isomorphism. This follows since the bottom left and top right arrows are isomorphisms by Lemma \ref{lem:upper!flatcoverbasechange}, and the bottom right is by \cite[Prop. 6.4.1.4]{LurSAG}. 
\end{proof}

If $f: X \to Y$ is a proper morphism of geometric stacks, $f_{QC}^!$ and $f_{IC}^!$ are left bounded since $f_{QC*}$ and $f_{IC*}$ are bounded. Given the equivalences $\IndCoh(-)^+ \cong \QCoh(-)^+$ and the isomorphism $\Psi_Y f_{IC*} \cong f_{QC*} \Psi_X$, Propositions \ref{prop:upper!almostcont} and \ref{prop:upper!ftdbasechange}  imply the following. 

\begin{Corollary}\label{cor:upper!almostcontIndCoh}
If  $f: X \to Y$ is a proper, almost finitely presented morphism of geometric stacks, then $f^!: \IndCoh(Y) \to \IndCoh(X)$ is almost continuous. 
\end{Corollary}

\begin{Corollary}\label{cor:upper!ftdbasechangeIndCoh}
	Under the hypotheses of Proposition \ref{prop:upper!ftdbasechange}, the Beck-Chevalley map $h'^* f^!(\cG) \to f'^! h^*(\cG)$ in $\IndCoh(X')$ is an isomorphism for all $\cG \in \IndCoh(Y)^+$. 
\end{Corollary}

We then have the following ind-geometric extension of Proposition \ref{prop:upper!almostcont}. 

\begin{Proposition}\label{prop:upper!almostcontind} 
Let $X$ and $Y$ be reasonable ind-geometric stacks and $f: X \to Y$ an ind-proper, almost ind-finitely presented morphism. If $f$ is of finite cohomological dimension, then $f^!: \IndCoh(Y) \to \IndCoh(X)$ is almost continuous. If $f$ is of ind-finite cohomological dimension and $X$ and $Y$ are coherent, then $f^!$ is continuous. 
\end{Proposition} 

\begin{proof} 
	The second claim follows since $f_*$ preserves coherence (Proposition \ref{prop:propprops}) and $\IndCoh(X)$ and $\IndCoh(Y)$ are compactly generated by coherent sheaves in this case (Proposition~\ref{prop:IndCohisIndofCoh}). For the first claim, let $X \cong \colim X_\al$ be a reasonable presentation. By the proof of \cite[Prop. 5.5.3.13]{LurHTT}, categories admitting filtered colimits and continuous functors among them form a subcategory of $\Cathatinfty$ which is closed under limits. 
	Since the restriction $i_{\al\be}^!: \IndCoh(X_\be)^{\geq 0} \to \IndCoh(X_\al)^{\geq 0}$ is continuous for all $\be \geq \al$ (Corollary \ref{cor:upper!almostcontIndCoh}), it follows that $i_\al^!$ is almost continuous for any $\al$. 
	
	Now let $\cF \cong \colim \cF_\be$ be a filtered colimit in $\IndCoh(Y)^{\geq 0}$. For any $\al$ we can refactor $f \circ i_\al$ as $i'_\al \circ f_\al$ for some reasonable geometric substack $i'_\al: Y_\al \to Y$ and some proper, almost finitely presented map $f_\al: X_\al \to Y_\al$.  Note that $i'^!_\al$ is almost continuous since we may extend $i'_\al$ to a reasonable presentation of $Y$. Since $\IndCoh(X) \cong \lim \IndCoh(X_\al)$ in $\Cathatinfty$, it suffices to show the second factor in 
	$$ \colim_\be i^!_\al f^! (\cF_\be) \to  i^!_\al  \colim_\be f^! (\cF_\be) \to i^!_\al f^! (\colim_\be \cF_\be) $$
	is an isomorphism for all $\al$. The first factor is an isomorphism since $f$ is of finite cohomological dimension, hence $f^!$ is left bounded, and since $i_\al^!$ is almost continuous. But $i_\al^! f^! \cong f^!_\al i'^!_\al$, so the composition is an isomorphism by the left t-exactness of $i'^!_\al$ and the almost continuity of $f^!_\al$ and $i'^!_\al$. 
	\end{proof}

We can now establish the following claim from Section \ref{sec:cohrelation}. 

\begin{Proposition}\label{prop:cohisacinindcoh}
If $X$ is a reasonable ind-geometric stack, the full subcategory $\Coh(X) \subset \IndCoh(X)$ consists of bounded, almost compact objects. 
\end{Proposition}
\begin{proof}
Fix a reasonable presentation $X \cong \colim_\al X_\al$. Given $\cF \in \Coh(X)$ we can write $\cF \cong i_{\al*}(\cF_\al)$ for some $\al$ and some $\cF_\al \in \Coh(X_\al)$. The t-exactness of $i_{\al*}$ implies $\cF$ is bounded in $\IndCoh(X)$, while the left boundedness and almost continuity of $i_\al^!$ (Proposition~\ref{prop:upper!almostcontind}) imply $\cF$ is almost compact. 
\end{proof}

\subsection{Pushforward and $!$-pullback}\label{sec:!pullpush}
We turn next to the commutation of pushforward and ind-proper $!$-pullback, generalizing the result of \cite[Prop. 2.9.2]{GR14} for ind-schemes of ind-finite type. To simplify the proof we assume the stacks involved are coherent, then discuss how the result may be generalized. We do caution that in the following statement the coherence of $X'$ is an additional hypothesis beyond the coherence of $X$, $Y$, and~$Y'$. 

\begin{Proposition}\label{prop:up!low*geom}
	Let the following be a Cartesian diagram of ind-geometric stacks.
	\begin{equation*}
		\begin{tikzpicture}
			[baseline=(current  bounding  box.center),thick,>=\arrtip]
			\node (a) at (0,0) {$X'$};
			\node (b) at (3,0) {$Y'$};
			\node (c) at (0,-1.5) {$X$};
			\node (d) at (3,-1.5) {$Y$};
			\draw[->] (a) to node[above] {$f' $} (b);
			\draw[->] (b) to node[right] {$h $} (d);
			\draw[->] (a) to node[left] {$h' $}(c);
			\draw[->] (c) to node[above] {$f $} (d);
		\end{tikzpicture}
	\end{equation*}
	Suppose that all stacks in the diagram are coherent, that $h$ is of ind-finite cohomological dimension, and that $f$ is ind-proper and almost ind-finitely presented. Then for any $\cF \in \IndCoh(Y')$ the Beck-Chevalley map $h'_* f'^!(\cF) \to f^! h_*(\cF)$ is an isomorphism. 
\end{Proposition}
\begin{proof}
By Proposition \ref{prop:upper!almostcontind}, all functors in the statement are continuous, hence by coherence of~$Y$ it suffices to show the claim for $\cF \in \Coh(Y)$. When $X$, $Y$, and $Y'$ are geometric, the functors in the statement are left bounded and compatible with the equivalences $\IndCoh(-)^+ \cong \QCoh(-)^+$, hence it suffices to show the map $h'_{QC*} f'^!_{QC}(\cF) \to f^!_{QC} h_{QC*}(\cF)$ is an isomorphism. But this follows since it is obtained from the isomorphism $f'_{QC*} h'^*_{QC} \congto  h_{QC}^* f_{QC*}$ in $\Fun(\QCoh(X),\QCoh(Y'))$ by taking right adjoints. 

Next suppose that $f$ is geometric and that $h$ is the inclusion of a reasonable geometric substack, which we may assume is a term in a reasonable presentation $Y \cong \colim Y_\al$. Writing $f_\al: X_\al \to Y_\al$ and $i'_{\al\be}: X_\al \to X_\be$ for the base changes of $f$ and $i_{\al\be}$, the functors $f^!_\al$, $i^!_{\al\be}$, and $i'^!_{\al\be}$ form a diagram $(A \times \Delta^1)^{\op} \to \Cathatinfty$, where $A$ is our index category. Each $X_\al$ is coherent by Proposition~\ref{prop:cohbasechangeunderafpclimm} and is geometric since $f$ is geometric, hence by the first paragraph $i'_{\al\be*} f^!_\al \to f^!_\be i_{\al\be*}$ is an isomorphism for all $\al \leq \be$. Since  $\IndCoh(X) \cong \colim \IndCoh(X_\al)$ and  $\IndCoh(Y) \cong \colim \IndCoh(Y_\al)$ in $\PrL$ (Proposition~\ref{prop:indcohonindgstks}), the claim follows by \cite[Prop. 4.7.5.19]{LurHA}. 

Suppose now that $X$ and $Y$ are geometric (hence so are $f$ and $f'$), and let $Y' \cong \colim Y'_\al$ be a reasonable presentation. Since $\cF \in \Coh(Y')$ we have $\cF \cong i_{\al*}(\cF_\al)$ for some~$\al$ and some $\cF_\al \in \Coh(Y'_\al)$. We want to show the second factor of 
\begin{equation*}
	 h'_* i'_{\al*} f^!_\al(\cF_\al) \to h'_* f'^! i_{\al*}(\cF_\al) \to f^! h_*  i_{\al*}(\cF_\al)
\end{equation*}
is an isomorphism, where $f_\al: X'_\al \to Y'_\al$ is the base change of $f$. Given that $X'_\al$ is coherent by Proposition~\ref{prop:cohbasechangeunderafpclimm}, this follows since the first factor is an isomorphism by the previous paragraph and the composition is by the first paragraph. 

Next suppose that that $f$ is the inclusion of a reasonable geometric substack, which again we may assume is a term in a reasonable presentation $Y \cong \colim Y_\al$. Writing $h_\al: Y'_\al \to Y_\al$ and $i'_{\al\be}: Y'_\al \to Y'_\be$ for the base changes of $h$ and $i_{\al\be}$, the functors $h_{\al*}$, $i_{\al\be*}$, and $i'_{\al\be*}$ form a diagram $A \times \Delta^1 \to \PrL$, where $A$ is our index category. Each $Y'_\al$ is coherent by Proposition~\ref{prop:cohbasechangeunderafpclimm}, hence by the previous paragraph $h_{\al*} i'^!_{\al\be} \to h_{\be*}  i^!_{\al\be}$ is an isomorphism for all $\al \leq \be$. Since $\IndCoh(Y') \cong \colim \IndCoh(Y'_\al)$ and  $\IndCoh(Y) \cong \colim \IndCoh(Y_\al)$ in $\PrL$ (Proposition~\ref{prop:indcohonindgstks}), the claim follows by \cite[Prop. 4.7.5.19]{LurHA}. 

In the general case, let $X \cong \colim X_\al$ be a reasonable presentation. For any $\al$ we can factor $f \circ i_\al$ through some reasonable geometric substack $j_\al: Y_\al \to Y$, and have a diagram
\begin{equation*}%\label{eq:up!compcube}
	\begin{tikzpicture}[baseline=(current  bounding  box.center),thick,>=\arrtip]
			\newcommand*{\ha}{1.5}; \newcommand*{\hb}{1.5}; \newcommand*{\hc}{1.5};
\newcommand*{\va}{-.9}; \newcommand*{\vb}{-.9}; \newcommand*{\vc}{-.9}; 
		\node (ab) at (\ha,0) {$X'_\al$};
		\node (ad) at (\ha+\hb+\hc,0) {$Y'_\al$};
		\node (ba) at (0,\va) {$X'$};
		\node (bc) at (\ha+\hb,\va) {$Y'$};
		\node (cb) at (\ha,\va+\vb) {$X_\al$};
		\node (cd) at (\ha+\hb+\hc,\va+\vb) {$Y_\al$};
		\node (da) at (0,\va+\vb+\vc) {$X$};
		\node (dc) at (\ha+\hb,\va+\vb+\vc) {$Y$};
		\draw[->] (ab) to node[above] {$f'_\al $} (ad);
		\draw[->] (ab) to node[above left, pos=.25] {$i'_\al $} (ba);
		\draw[->] (ab) to node[right,pos=.2] {$h'_\al $} (cb);
		\draw[->] (ad) to node[below right] {$j'_\al $} (bc);
		\draw[->] (ad) to node[right] {$h_\al $} (cd);
		\draw[->] (ba) to node[left] {$h' $} (da);
		\draw[->] (cb) to node[above,pos=.25] {$f_\al $} (cd);
		\draw[->] (cb) to node[above left, pos=.25] {$i_\al $} (da);
		\draw[->] (cd) to node[below right] {$j_\al $} (dc);
		\draw[->] (da) to node[above,pos=.75] {$f $} (dc);
		
		\draw[-,line width=6pt,draw=white] (ba) to  (bc);
		\draw[->] (ba) to node[above,pos=.75] {$ $} (bc);
		\draw[-,line width=6pt,draw=white] (bc) to  (dc);
		\draw[->] (bc) to node[right,pos=.2] {$h $} (dc);
	\end{tikzpicture}
\end{equation*} 
with all faces but the top and bottom Cartesian. We then have a diagram
\begin{equation*}
	\begin{tikzpicture}
		[baseline=(current  bounding  box.center),thick,>=\arrtip]
		\newcommand*{\ha}{3.5}; \newcommand*{\hb}{3.5};
		\newcommand*{\va}{-1.5};
		\node (aa) at (0,0) {$h'_{\al*} i'^!_\al f'^!(\cF)$};
		\node (ab) at (\ha,0) {$i^!_\al h'_* f'^!(\cF)$};
		\node (ac) at (\ha+\hb,0) {$i^!_\al f^! h_*(\cF)$};
		\node (ba) at (0,\va) {$h'_{\al*} f'^!_\al j'^!_{\al}(\cF)$};
		\node (bb) at (\ha,\va) {$ f^!_\al h_{\al*} j'^!_{\al}(\cF)$};
		\node (bc) at (\ha+\hb,\va) {$f^!_\al j^!_{\al}  h_* (\cF)$};
		\draw[->] (aa) to node[above] {$ $} (ab);
		\draw[->] (ab) to node[above] {$  $} (ac);
		\draw[->] (ba) to node[above] {$ $} (bb);
		\draw[->] (bb) to node[above] {$ $} (bc);
		\draw[->] (aa) to node[below,rotate=90] {$\sim $} (ba);
		%\draw[->] (ab) to node[right] {$h' $} (bb);
		\draw[->] (ac) to node[below,rotate=90] {$\sim $} (bc);
	\end{tikzpicture} 
\end{equation*}
in $\IndCoh(Y_\al)$. Since the functors $i^!_\al$ determine an isomorphism $\IndCoh(Y) \cong \lim \IndCoh(Y_\al)$ in  $\Cathatinfty$, it suffices to show the top right arrow 
is an isomorphism for all $\al$.  But, given that $X'_\al$ and $Y'_\al$ are coherent by Proposition \ref{prop:cohbasechangeunderafpclimm}, the top left and bottom right arrows are isomorphisms by the previous paragraph, and the bottom left is by the third paragraph. 
\end{proof}

As mentioned above, the coherence of $X$, $Y$, and $Y'$ does not imply that of $X'$, though in motivating applications $X$, $Y$, and $Y'$ satisfy stronger hypotheses that do imply this. A basic case is when $Y'$ is locally Noetherian, in which case so is $X'$ by the hypotheses on $f$. More generally, we will see in \cite{CWtm} that if $Y'$ is ind-tamely presented and $h$ is affine, then coherence of $X'$ follows from that of $X$. 

One can remove the coherence hypotheses in Proposition \ref{prop:up!low*geom} at the cost of adding other hypotheses, and in particular adding boundedness conditions on $\cF$. The most obvious complication that arises is that $f^!$ need not take $\IndCoh(Y)^+$ to $\IndCoh(X)^+$ unless $f$ is of finite (rather than merely ind-finite) cohomological dimension. 

Beyond imposing this condition on $f$, we can also note that $f^!$ does in general take $\IndCoh(Y)^+_{lim}$ to $\IndCoh(X)^+_{lim}$, where $\IndCoh(Y)^{+}_{lim} \subset  \IndCoh(Y)$ is the full subcategory of $\cF$ such that $i^!(\cF) \in \IndCoh(Y'')^+$ for every truncated geometric substack $i: Y'' \to Y$. Though $h_*$ does not in general take  $\IndCoh(Y')^+_{lim}$ to $\IndCoh(Y)^+_{lim}$, it does if we further assume $h$ is formally geometric. Here we say an ind-geometric stack $Z$ is formally geometric if its underlying classical stack $\tau_{\leq 0}Z$ is geometric, and we define formally geometric morphisms by base change to affine schemes. For example, the inclusion of any truncated geometric substack is a formally geometric morphism. With significant additional care, the proof of Proposition \ref{prop:up!low*geom} can then be extended to show the following. 

\begin{Proposition}\label{prop:up!low*indgeomgeneral}
	Let the following be a Cartesian diagram of ind-geometric stacks.
	\begin{equation*}
		\begin{tikzpicture}
			[baseline=(current  bounding  box.center),thick,>=\arrtip]
			\node (a) at (0,0) {$X'$};
			\node (b) at (3,0) {$Y'$};
			\node (c) at (0,-1.5) {$X$};
			\node (d) at (3,-1.5) {$Y$};
			\draw[->] (a) to node[above] {$f' $} (b);
			\draw[->] (b) to node[right] {$h $} (d);
			\draw[->] (a) to node[left] {$h' $}(c);
			\draw[->] (c) to node[above] {$f $} (d);
		\end{tikzpicture}
	\end{equation*}
	Suppose that  $X$ and $Y$ are reasonable, that $h$ is of ind-finite cohomological dimension, and that $f$ is ind-proper and almost ind-finitely presented. 
	Suppose also that $f$ is of finite cohomological dimension  (resp. that $h$ is formally geometric). 
	Then for any $\cF \in \IndCoh(Y')^+$ (resp. $\cF \in \IndCoh(Y')^+_{lim}$) the Beck-Chevalley map $h'_* f'^!(\cF) \to f^! h_*(\cF)$ is an isomorphism. 
\end{Proposition}

\section{External products and sheaf Hom}\label{sec:sheafHom}

Given a geometric stack $Y$, sheaf Hom from $\cF \in \QCoh(Y)$ is defined by the adjunction
\begin{equation*}
	- \otimes \cF : \QCoh(Y) \leftrightarrows \QCoh(Y) : \cHom(\cF, -).
\end{equation*}
If $Y$ is a reasonable ind-geometric stack and $\cF \in \Coh(Y)$, there is still a natural functor $\cHom(\cF, -): \IndCoh(Y) \to \IndCoh(Y)$, despite the absence of a tensor product of ind-coherent sheaves in general. This is because we do still have an external tensor product, and for any ind-geometric $X$ we can consider the adjunction 
\begin{equation*}
	- \boxtimes \cF : \IndCoh(X) \leftrightarrows \IndCoh(X \times Y) : (- \boxtimes \cF)^R.
\end{equation*}
To make explicit their dependence on $X$, we will denote these functors by $e_{\cF,X}$ and $e_{\cF,X}^R$. When $X$ and $Y$ are geometric, we have an isomorphism $\cHom(\cF, -) \cong e_{\cF,Y}^R \Delta_{Y*}$, and this formula provides a useful definition of sheaf Hom in the ind-geometric setting. 

In this section we show that ind-coherent sheaf Hom, and more fundamentally the functor~$e_{\cF,X}^R$, retains many good properties from the quasi-coherent setting. In particular, it is almost continuous (Proposition \ref{prop:eFXRalmostcont}) and compatible with pushforward (Propositions \ref{prop:eFXRlower*} and \ref{prop:indhomrelationcoherent}). We also show the ind-coherent external product is itself compatible with ind-proper $!$-pullback (Proposition \ref{prop:eFXupper!compatindgeom}). A basic technical theme is the close analogy between $e_{\cF,X}^R$ and ind-proper $!$-pullback, with various key proofs in this section being parallel to corresponding proofs in Section~\ref{sec:upper!}. 

\subsection{Quasi-coherent sheaf Hom} 
We begin by collecting some basic facts about quasi-coherent sheaf Hom on geometric stacks. An immediate compatibility is that for any morphism $f: X \to Y$ in $\GStkk$ and any $\cF \in \QCoh(Y)$, the isomorphism $f^*(- \otimes \cF) \cong f^*(-) \otimes f^*(\cF)$ gives rise to an isomorphism of right adjoints
\begin{equation}\label{eq:QCHomupper*}
	\cHom(\cF, f_*(-)) \cong f_* \cHom(f^*(\cF), -).
\end{equation}
We also have the following results, which generalize various special cases in the literature.  For example, when $h$ is flat a different proof of Proposition \ref{prop:cHomftdbasechange} is given in \cite[Prop. 9.5.3.3]{LurSAG}, and when the objects involved are schemes the result appears as \cite[Lem. 0AA7]{Sta}. 

%We begin by establishing the basic facts we will need about quasi-coherent sheaf Hom on geometric stacks. As in Section \ref{sec:reasgeomstacks}, many of these are simple extensions to the geometric setting of existing results about Artin stacks. 

\begin{Proposition}\label{prop:cHomalmostcont}
	If $Y$ is a geometric stack and $\cF \in \QCoh(Y)$ is almost perfect, then $\cHom(\cF,-): \QCoh(Y) \to \QCoh(Y)$ is left bounded and almost continuous. 
\end{Proposition} 

\begin{Proposition}\label{prop:cHomftdbasechange}
	Let $X$ and $Y$ be geometric stacks, $h: X \to Y$ a morphism of finite Tor-dimension, and $\cF \in \QCoh(Y)$ an almost perfect sheaf. Then the Beck-Chevalley map $h^* \cHom(\cF, \cG) \to \cHom(h^*(\cF), h^*(\cG))$ is an isomorphism for all $\cG \in \QCoh(Y)^+$. 
\end{Proposition}

\begin{Lemma}\label{lem:cHomalmostcontaffine}
	Proposition \ref{prop:cHomalmostcont} is true when $Y$ is affine. 
\end{Lemma}
\begin{proof}
	In this case $\QCoh(Y)$ is compactly generated by perfect sheaves. If $\cG \in \QCoh(Y)$ is perfect then $\cF \otimes \cG$ is almost perfect, hence almost continuity follows by applying \cite[Prop. 5.5.7.2]{LurHTT} to $\tau^{\geq n} (\cF \ot -): \QCoh(Y) \to \QCoh(Y)^{\geq n}$. Left boundedness follows since $\cF$ and hence $- \ot \cF$ are right bounded. 
\end{proof}

\begin{Lemma}\label{lem:cHomaffineftdbasechange}
	Proposition \ref{prop:cHomftdbasechange} is true when $X$ and $Y$ are affine. 
\end{Lemma}
\begin{proof}
	Let $Y \cong \Spec A$ and $X \cong \Spec B$. Since $h$ is affine $h_*$ conservative, hence it suffices to show $h_* h^* \cHom(\cF, \cG) \to h_* \cHom(h^*(\cF), h^*(\cG))$ is an isomorphism. Rewriting the second term using the projection formula, one sees this is the specialization of the Beck-Chevalley map $\theta_M: \cHom(\cF, \cG) \ot M \to \cHom(\cF, \cG \ot M)$ in the case $M = B$.
	
	Write $\catC$ for the full subcategory of  $M \in \Mod_A$ such that $\theta_{M}$ is an isomorphism. The assignment $M \mapsto \theta_{M}$ extends to a functor $\Mod_A \to \Mod_A^{\Delta^1}$, which is exact since the source and target of $\theta_M$ are exact in $M$. It follows that $\catC$ is a stable subcategory closed under retracts, as isomorphisms form such a subcategory of $\Mod_A^{\Delta^1}$. Clearly $A \in \catC$, hence $\catC$ contains all perfect $A$-modules. If $M$ is of Tor-dimension $\leq n$, then we can write it as a filtered colimit $M \cong \colim_\al M_\al$ of perfect $A$-modules of Tor-dimension $\leq n$ \cite[Prop. 9.6.7.1]{LurSAG}. The claim now follows since tensoring is continuous, since the $\cG \ot M_\al$ are uniformly bounded below, and since $\cHom(\cF,-)$ is almost continuous by Lemma~\ref{lem:cHomalmostcontaffine}. 
\end{proof}

\begin{Lemma}\label{lem:cHomflatcoverbasechange}
	Proposition \ref{prop:cHomftdbasechange} is true when $X$ is affine and $h$ is faithfully flat. 
\end{Lemma}
\begin{proof}
	Let $X_\bul$ denote the Cech nerve of $h$ (so $X_0 = X$), and let $h_k: X_k \to Y$ denote the natural map. 
	Given a morphism $p: i \to j$ in $\Delta_s$, let $h_p: X_j \to X_i$ denote the associated map. Choose $n$ so that $\cF \in \QCoh(Y)^{\leq n}$. Then $\tau^{\geq n}(\cF \ot -)$ and $\cHom(\cF,-)$ restrict to an adjunction between $\QCoh(Y)^{\geq 0}$ and $\QCoh(Y)^{\geq n}$, similarly for $\cF_k = h_k^*(\cF) \in \QCoh(X_k)^{\leq n}$. The categories $\QCoh(X_k)^{\geq 0}$ and $\QCoh(X_k)^{\geq n}$, together with the functors $h^*_p$ and $\tau^{\geq n}(\cF_k \ot -)$, form a diagram $\Delta^1 \times \Delta_s \to \Cathatinfty$. By Lemma \ref{lem:cHomaffineftdbasechange} the Beck-Chevalley transformation $h_p^* \cHom(\cF_i,-) \to \cHom(\cF_j, h^*_p(-))$ restricts to an isomorphism of functors $\QCoh(X_i)^{\geq n} \to \QCoh(X_j)^{\geq 0}$ for any $p$. Since $h$ is faithfully flat we have $\QCoh(X)^{\geq m} \cong \lim_{\Delta_s} \QCoh(X_i)^{\geq m}$ for any $m$, and the claim follows from \cite[Cor. 4.7.5.18]{LurHA}. 
\end{proof}

\begin{proof}[Proof of Proposition \ref{prop:cHomalmostcont}]
	Again left boundedness follows since $\cF$ is right bounded. 	Let $\cG \cong \colim \cG_\al$ be a filtered colimit in $\QCoh(Y)^{\geq 0}$ and let $h: X \cong \Spec A \to Y$ be a flat cover. Since $h^*$ is continuous and conservative it suffices to show $\colim h^* \cHom(\cF,\cG_\al) \to h^* \cHom(\cF,\cG)$ is an isomorphism. By Lemma~\ref{lem:cHomflatcoverbasechange} this is equivalent to $\colim  \cHom(h^*(\cF),h^*(\cG_\al)) \to  \cHom(h^*(\cF),h^*(\cG))$ being an isomorphism. Since $h^*$ is t-exact this follows from Lemma~\ref{lem:cHomalmostcontaffine}. 
\end{proof}

\begin{proof}[Proof of Proposition \ref{prop:cHomftdbasechange}]
	Let $\phi: V \cong \Spec A \to Y$ be a flat cover, and consider the diagram
	\begin{equation*}
		\begin{tikzpicture}
			[baseline=(current  bounding  box.center),thick,>=\arrtip]
			\node (a) at (0,0) {$U$};
			\node (b) at (3,0) {$X$};
			\node (c) at (0,-1.5) {$V$};
			\node (d) at (3,-1.5) {$Y$};
			\draw[->] (a) to node[above] {$\psi $} (b);
			\draw[->] (b) to node[right] {$h $} (d);
			\draw[->] (a) to node[left] {$g $}(c);
			\draw[->] (c) to node[above] {$\phi $} (d);
		\end{tikzpicture}
	\end{equation*}
	induced by a flat cover $\xi: U \cong \Spec B \to V \times_Y X$. Note that $\psi$ is faithfully flat and $g$ is of finite Tor-dimension, since they are the compositions of $\xi$ with the base changes of $\phi$ and $f$, respectively. Since $\psi^*$ is conservative it suffices to show the top left arrow in 
	\begin{equation*}
		\begin{tikzpicture}
			[baseline=(current  bounding  box.center),thick,>=\arrtip]
			\newcommand*{\ha}{4.5}; \newcommand*{\hb}{5.0};
			\newcommand*{\va}{-1.5};
			\node (aa) at (0,0) {$\psi^* h^* \cHom(\cF,\cG)$};
			\node (ab) at (\ha,0) {$\psi^*  \cHom(h^*(\cF),h^*(\cG))$};
			\node (ac) at (\ha+\hb,0) {$  \cHom(\psi^*h^*(\cF),\psi^*h^*(\cG))$};
			\node (ba) at (0,\va) {$g^* \phi^* \cHom(\cF,\cG)$};
			\node (bb) at (\ha,\va) {$g^*  \cHom(\phi^*(\cF),\phi^*(\cG))$};
			\node (bc) at (\ha+\hb,\va) {$\cHom(g^*\phi^*(\cF),g^*\phi^*(\cG))$};
			\draw[->] (aa) to node[above] {$ $} (ab);
			\draw[->] (ab) to node[above] {$  $} (ac);
			\draw[->] (ba) to node[above] {$ $} (bb);
			\draw[->] (bb) to node[above] {$ $} (bc);
			\draw[->] (aa) to node[below,rotate=90] {$\sim $} (ba);
			%\draw[->] (ab) to node[right] {$h' $} (bb);
			\draw[->] (ac) to node[below,rotate=90] {$\sim $} (bc);
		\end{tikzpicture} 
	\end{equation*}
	is an isomorphism. This follows since the bottom left and top right arrows are isomorphisms by Lemma \ref{lem:cHomflatcoverbasechange}, and the bottom right is by Lemma \ref{lem:cHomaffineftdbasechange}. 
\end{proof}

%We note that a different proof of Proposition \ref{prop:cHomftdbasechange} is given in \cite[Prop. 9.5.3.3]{LurSAG} when $h$ is flat. We also note the more evident compatibility that for any morphism $f: X \to Y$ in $\GStkk$ and any $\cF \in \QCoh(Y)$, the isomorphism $f^*(- \otimes \cF) \cong f^*(-) \otimes f^*(\cF)$ gives rise to an isomorphism of right adjoints
%\begin{equation}\label{eq:QCHomupper*}
%	\cHom(\cF, f_*(-)) \cong f_* \cHom(f^*(\cF), -).
%\end{equation}

\subsection{Quasi-coherent external products}
Let $X$ and $Z$ be geometric stacks, $p_X: X \times Z \to X$ and $p_Z: X \times Z \to Z$ the projections, and $\cF \in \QCoh(Z)$. By definition the external product $e_{\cF,X}: \QCoh(X) \to \QCoh(X \times Z)$ and its right adjoint are given by
$$e_{\cF,X} := p_X^*(-) \ot p_Z^*(\cF), \quad \quad e_{\cF,X}^R := p_{X*} \cHom(p_Z^*(\cF), - ).$$ 
Propositions \ref{prop:lower*almostcont}, \ref{prop:lower*ftdbasechange}, \ref{prop:cHomalmostcont}, and \ref{prop:cHomftdbasechange} immediately imply the following results. 

\begin{Proposition}\label{prop:eFXRQCohalmostcont}
If $Y$ and $Z$ are geometric stacks and $\cF \in \QCoh(Z)$ is almost perfect, then $e_{\cF,X}^R: \QCoh(X \times Z) \to \QCoh(X)$ is left bounded and almost continuous.
\end{Proposition} 

\begin{Proposition}
Let $X$, $Y$, and $Z$ be geometric stacks, $h: X \to Y$ a morphism of finite Tor-dimension, and $\cF \in \QCoh(Z)$ an almost perfect sheaf. Then the Beck-Chevalley map $h^* e_{\cF,Y}^R(\cG) \to e_{\cF,X}^R (h \times \id_Z)^*(\cG)$ is an isomorphism for all $\cG \in \QCoh(Y \times Z)^+$. 
\end{Proposition} 

We will be most interested in external products when $\kk$ is an ordinary (Noetherian) ring of finite global dimension, in which case we have the following boundedness condition. 

\begin{Proposition}\label{prop:cohextprod}
	Suppose $\kk$ is an ordinary ring of finite global dimension. Then $e_{\cF,X}$ is left bounded for any $X, Z \in \GStkk$ and $\cF \in \QCoh(X)^+$. In particular, if $X$ and $Z$ are truncated and $\cG \in \Coh(X)$, $\cF \in \Coh(Z)$, then $X\times Z$ is truncated and $\cG \boxtimes \cF \in \Coh(X \times Z)$. 
\end{Proposition}
\begin{proof}
If $\cF \in \QCoh(Z)^{\geq n}$ and $\kk$ is of global dimension $d$, then we claim $e_{\cF,X}$ takes $\QCoh(X)^{\geq 0}$ to $\QCoh(X \times Z)^{\geq n-d}$. When $X$ and $Z$ are affine, this follows from the fact that $(\cG \boxtimes \cF)_\kk \cong (\cG)_\kk \ot_\kk (\cF)_\kk$, where $(-)_\kk$ denotes restriction of scalars to $\Mod_\kk$. If $g: U \cong \Spec A \to X$ and $h: V \cong \Spec B \to Z$ are flat covers, the general case follows since $(g \times h)^*e_{\cF,X}(\cG) \cong e_{h^*(\cF), U} g^*(\cG)$. The last claim follows since almost perfect sheaves are closed under pullbacks and tensor products. 
\end{proof}

\begin{Corollary}
Suppose $\kk$ is an ordinary ring of finite global dimension. If $X$ and $Y$ are reasonable ind-geometric stacks, then so is $X \times Y$. 
\end{Corollary}

Proposition~\ref{prop:cohextprod} also ensures that external products commute with the various almost continuous functorialities we have considered, as described by the following three results.

\begin{Proposition}\label{prop:eFXlower*compat}
	Suppose $\kk$ is an ordinary ring of finite global dimension. 
	Let $X$, $Y$, and $Z$ be geometric stacks, $f: X \to Y$ a morphism, and $\cF \in \QCoh(Z)^+$. Then the Beck-Chevalley map $e_{\cF,Y} f_*(\cG) \to (f \times \id_Z)_* e_{\cF,X}(\cG)$ is an isomorphism for all $\cG \in \QCoh(X)^+$. 
\end{Proposition}
\begin{proof}
	First suppose $Y$ and $Z$ are affine. Then $p_Y: Y \times Z \to Y$ is affine, hence $p_{Y*}$ is conservative and it suffices to show  $$p_{Y*} e_{\cF,Y} f_*(\cG) \to p_{Y*} (f \times \id_Z)_* e_{\cF,X}(\cG) \congto f_* p_{X*} e_{\cF,X}(\cG)$$ is an isomorphism. Using the projection formula for the affine morphisms $p_X$ and $p_Y$ \cite[Cor. 6.3.4.3]{LurSAG}, we can identify this with the Beck-Chevalley map $$f_*(\cG) \ot p_{Y*} p^*_Z(\cF) \to f_*(\cG \ot f^* p_{Y*} p^*_Z(\cF)).$$ 
	But $p_{Y*} p^*_Z(\cF) \cong q^*_{Y}q_{Z*}(\cF)$, where $q_Y$, $q_Z$ are the structure maps to $\Spec \kk$, hence $p_{Y*} p^*_Z(\cF)$ is of finite Tor-dimension by our hypotheses on $\kk$ and left boundedness of $\cF$. The claim now follows from the proof of Lemma~\ref{lem:lower*affineftdbasechange}. 	
	
	In the general case, let $g: U \cong \Spec A \to Y$ and $h: V \cong \Spec B \to Z$ be flat covers. We obtain a diagram
	\begin{equation*}
		\begin{tikzpicture}[baseline=(current  bounding  box.center),thick,>=\arrtip]
			\newcommand*{\ha}{1.5}; \newcommand*{\hb}{1.5}; \newcommand*{\hc}{1.5};
\newcommand*{\va}{-.9}; \newcommand*{\vb}{-.9}; \newcommand*{\vc}{-.9}; 
			\node (ab) at (\ha,0) {$W \times V$};
			\node (ad) at (\ha+\hb+\hc,0) {$U \times V$};
			\node (ba) at (0,\va) {$X \times Z$};
			\node (bc) at (\ha+\hb,\va) {$Y \times Z$};
			\node (cb) at (\ha,\va+\vb) {$W$};
			\node (cd) at (\ha+\hb+\hc,\va+\vb) {$U$};
			\node (da) at (0,\va+\vb+\vc) {$X$};
			\node (dc) at (\ha+\hb,\va+\vb+\vc) {$Y$};
			\draw[->] (ab) to node[above] {$ $} (ad);
			\draw[->] (ab) to node[above left, pos=.25] {$g' \times h $} (ba);
			\draw[->] (ab) to node[right,pos=.2] {$ $} (cb);
			\draw[->] (ad) to node[above left, pos=.75] {$ $} (bc);
			\draw[->] (ad) to node[right] {$ $} (cd);
			\draw[->] (ba) to node[left] {$ $} (da);
			\draw[->] (cb) to node[above,pos=.25] {$ $} (cd);
			\draw[->] (cb) to node[above left, pos=.25] {$g' $} (da);
			\draw[->] (cd) to node[below right] {$g $} (dc);
			\draw[->] (da) to node[above,pos=.75] {$ $} (dc);
			
			\draw[-,line width=6pt,draw=white] (ba) to  (bc);
			\draw[->] (ba) to node[above,pos=.75] {$ $} (bc);
			\draw[-,line width=6pt,draw=white] (bc) to  (dc);
			\draw[->] (bc) to node[right,pos=.2] {$ $} (dc);
		\end{tikzpicture}
	\end{equation*}
	in which all but the left and right faces are Cartesian. Since $(g \times h)^*$ is conservative, it suffices to show the top left arrow in 
	\begin{equation*}
		\begin{tikzpicture}
			[baseline=(current  bounding  box.center),thick,>=\arrtip]
			\newcommand*{\ha}{4.8}; \newcommand*{\hb}{5.5};
			\newcommand*{\va}{-1.5};
			\node (aa) at (0,0) {$(g \times h)^* e_{\cF,Y} f_*(\cG)$};
			\node (ab) at (\ha,0) {$(g \times h)^* (f \times \id_Z)_* e_{\cF,X} (\cG)$};
			\node (ac) at (\ha+\hb,0) {$ (f' \times \id_V)_* (g' \times h)^* e_{\cF,X} (\cG)$};
			\node (ba) at (0,\va) {$ e_{h^*(\cF),U} g^* f_*(\cG)$};
			\node (bb) at (\ha,\va) {$e_{h^*(\cF),U} f'_* g'^*(\cG)$};
			\node (bc) at (\ha+\hb,\va) {$(f' \times \id_V)_* e_{h^*(\cF),W} g'^*(\cG)$};
			\draw[->] (aa) to node[above] {$ $} (ab);
			\draw[->] (ab) to node[above] {$  $} (ac);
			\draw[->] (ba) to node[above] {$ $} (bb);
			\draw[->] (bb) to node[above] {$ $} (bc);
			\draw[->] (aa) to node[below,rotate=90] {$\sim $} (ba);
			%\draw[->] (ab) to node[right] {$h' $} (bb);
			\draw[->] (ac) to node[below,rotate=90] {$\sim $} (bc);
		\end{tikzpicture} 
	\end{equation*}
	is an isomorphism. This follows since the bottom left arrow is an isomorphism by Proposition~\ref{prop:lower*ftdbasechange}, the top right is by this and Proposition~\ref{prop:cohextprod}, and the bottom right is by the previous paragraph. 
\end{proof}

\begin{Proposition}\label{prop:eFXupper!compat}
	Suppose $\kk$ is an ordinary ring of finite global dimension. 
	Let $X$, $Y$, and $Z$ be geometric stacks, $f: X \to Y$ a proper, almost finitely presented morphism, and $\cF \in \QCoh(Z)^+$. Then the Beck-Chevalley map $e_{\cF,X}f^!(\cG) \to  (f \times \id_Z)^!e_{\cF, Y} (\cG)$ is an isomorphism for all $\cG \in \QCoh(Y)^+$. 
\end{Proposition}
\begin{proof}
	First suppose $Y$ and $Z$ are affine. Then $p_X: X \times Z \to X$ is affine, hence $p_{X*}$ is conservative and it suffices to show the first factor of $$ p_{X*} e_{\cF,X} f^!(\cG) \to p_{X*} (f \times \id_Z)^! e_{\cF, Y}(\cG) \to  f^! p_{Y*} e_{\cF, Y}(\cG) $$
	is an isomorphism. The second factor is an isomorphism by the first paragraph of the proof of Proposition \ref{prop:up!low*geom}, so it suffices to show the composition is.  Using the projection formula for the affine morphisms $p_X$ and $p_Y$ \cite[Cor. 6.3.4.3]{LurSAG}, we can identify this with the Beck-Chevalley map $$f^!(\cG) \ot f^*p_{Y*} p^*_Z(\cF) \to f^!(\cG \ot p_{Y*} p^*_Z(\cF)).$$ 
	But $p_{Y*} p^*_Z(\cF) \cong q^*_{Y}q_{Z*}(\cF)$, where $q_Y$, $q_Z$ are the structure maps to $\Spec \kk$, hence $p_{Y*} p^*_Z(\cF)$ is of finite Tor-dimension by our hypotheses on $\kk$ and left boundedness of $\cF$. The claim now follows from \cite[Lem. 6.4.1.8]{LurSAG}. 
	
	In the general case, let $g: U \cong \Spec A \to Y$ and $h: V \cong \Spec B \to Z$ be flat covers. We obtain a diagram
	\begin{equation}%\label{eq:up!compcube}
		\begin{tikzpicture}[baseline=(current  bounding  box.center),thick,>=\arrtip]
			\newcommand*{\ha}{1.5}; \newcommand*{\hb}{1.5}; \newcommand*{\hc}{1.5};
\newcommand*{\va}{-.9}; \newcommand*{\vb}{-.9}; \newcommand*{\vc}{-.9}; 
			\node (ab) at (\ha,0) {$W$};
			%\node (ad) at (\ha+\hb+\hc,0) {$X' \times \Spec B$};
			\node (ad) at (\ha+\hb+\hc,0) {$W \times V$};
			\node (ba) at (0,\va) {$X$};
			\node (bc) at (\ha+\hb,\va) {$X \times Z$};
			\node (cb) at (\ha,\va+\vb) {$U$};
			%\node (cb) at (\ha,\va+\vb) {$\Spec A$};
			\node (cd) at (\ha+\hb+\hc,\va+\vb) {$U \times V$};
			%\node (cd) at (\ha+\hb+\hc,\va+\vb) {$\Spec A \ot B$};
			\node (da) at (0,\va+\vb+\vc) {$Y$};
			\node (dc) at (\ha+\hb,\va+\vb+\vc) {$Y \times Z$};
			\draw[<-] (ab) to node[above] {$ $} (ad);
			\draw[->] (ab) to node[above left, pos=.25] {$g' $} (ba);
			\draw[->] (ab) to node[right,pos=.2] {$f' $} (cb);
			\draw[->] (ad) to node[below right] {$ $} (bc);
			\draw[->] (ad) to node[right] {$ $} (cd);
			\draw[->] (ba) to node[left] {$f $} (da);
			\draw[<-] (cb) to node[above,pos=.25] {$ $} (cd);
			\draw[->] (cb) to node[above left, pos=.25] {$g $} (da);
			\draw[->] (cd) to node[below right] {$g \times h $} (dc);
			\draw[<-] (da) to node[above,pos=.75] {$ $} (dc);
			
			\draw[-,line width=6pt,draw=white] (ba) to  (bc);
			\draw[<-] (ba) to node[above,pos=.75] {$ $} (bc);
			\draw[-,line width=6pt,draw=white] (bc) to  (dc);
			\draw[->] (bc) to node[right,pos=.2] {$ $} (dc);
		\end{tikzpicture}
	\end{equation}
	in which all but the top and bottom faces are Cartesian. Since $(g' \times h)^*$ is conservative, it suffices to show the top left arrow in
	\begin{equation*}
		\begin{tikzpicture}
			[baseline=(current  bounding  box.center),thick,>=\arrtip]
			\newcommand*{\ha}{4.8}; \newcommand*{\hb}{5.5};
			\newcommand*{\va}{-1.5};
			\node (aa) at (0,0) {$(g' \times h)^* e_{\cF,X} f^!(\cG)$};
			\node (ab) at (\ha,0) {$(g' \times h)^* (f \times \id_Z)^! e_{\cF,Y} (\cG)$};
			\node (ac) at (\ha+\hb,0) {$ (f' \times \id_V)^! (g \times h)^* e_{\cF,Y} (\cG)$};
			\node (ba) at (0,\va) {$ e_{h^*(\cF),W} g'^* f^!(\cG)$};
			\node (bb) at (\ha,\va) {$e_{h^*(\cF),W} f'^! g^*(\cG)$};
			\node (bc) at (\ha+\hb,\va) {$(f' \times \id_V)^! e_{h^*(\cF),U} g^*(\cG)$};
			\draw[->] (aa) to node[above] {$ $} (ab);
			\draw[->] (ab) to node[above] {$  $} (ac);
			\draw[->] (ba) to node[above] {$ $} (bb);
			\draw[->] (bb) to node[above] {$ $} (bc);
			\draw[->] (aa) to node[below,rotate=90] {$\sim $} (ba);
			%\draw[->] (ab) to node[right] {$h' $} (bb);
			\draw[->] (ac) to node[below,rotate=90] {$\sim $} (bc);
		\end{tikzpicture} 
	\end{equation*}
	is an isomorphism. This follows since the bottom left arrow is an isomorphism by Proposition~\ref{prop:upper!ftdbasechange}, the top right is by this and Proposition \ref{prop:cohextprod}, and the bottom right is by the first paragraph. 
\end{proof}

To express the compatibility of external products and sheaf Hom, we write $\eop_{\cF,X}: \QCoh(X) \to \QCoh(Z \times X)$ for $\cF \boxtimes -$, the external product with $\cF$ on the left. Given $\cF' \in \Coh(Z')$, the associativity isomorphism $ \cF' \boxtimes (- \boxtimes \cF) \cong (\cF' \boxtimes -) \boxtimes \cF,$ is written as $\eop_{\cF',X \times Z}e_{\cF,X} \cong e_{\cF, Z' \times X} \eop_{\cF',X}$ in this notation. We obtain a Beck-Chevalley transformation, which in standard notation would be written as
$$ \cF' \boxtimes p_{X*} \cHom(p_Z^*(\cF),-) \to p_{Z' \times X *} \cHom(p_Z^*(\cF), \cF' \boxtimes -). $$ 

\begin{Proposition}\label{prop:eFXcHomcompat}
	Suppose $\kk$ is an ordinary ring of finite global dimension. 
	Let $X$ $Z$ be geometric stacks, let $\cF \in \QCoh(X)$ be almost perfect, and let $\cF' \in \QCoh(Z)^+$ be arbitrary. Then the Beck-Chevalley map $\eop_{\cF',X} \cHom(\cF,\cG) \to \cHom(p_X^*(\cF), \eop_{\cF',X}(\cG))$ is an isomorphism for all $\cG \in \QCoh(X)^+$. 
\end{Proposition}
\begin{proof}
	First suppose $X$ and $Z$ are affine. Then $p_X: X \times Z \to X$ is affine, hence $p_{X*}$ is conservative and it suffices to show  show 
	$$ p_{X*} \eop_{\cF', X} \cHom(\cF,\cG) \to p_{X*} \cHom(p_X^*(\cF), \eop_{\cF', X}(\cG)) \congto \cHom(\cF, p_{X*} \eop_{\cF',X}(\cG))  $$
	is an isomorphism. 
	Using the projection formula for the affine morphisms $p_X$ and $p_Y$ \cite[Cor. 6.3.4.3]{LurSAG}, we can identify this with the Beck-Chevalley map $$ p_{X*} p^*_Z(\cF') \otimes \cHom(\cF,\cG) \to \cHom(\cF, p_{X*} p^*_Z(\cF') \otimes \cG).  $$ But $p_{X*} p^*_Z(\cF) \cong q^*_{X}q_{Z*}(\cF)$, where $q_X$, $q_Z$ are the structure maps to $\Spec \kk$, hence $p_{X*} p^*_Z(\cF)$ is of finite Tor-dimension by our hypotheses on $\kk$ and left boundedness of $\cF$. The claim now follows from \cite[Lem. 6.5.3.7]{LurSAG}. 

In the general case, let $g: U \cong \Spec A \to X$ and $h: V \cong \Spec B \to Z$ be flat covers. Since $(h \times g)^*$ is conservative, it suffices to show the top left arrow in 
\begin{equation*}
	\begin{tikzpicture}
		[baseline=(current  bounding  box.center),thick,>=\arrtip]
		\newcommand*{\ha}{5.1}; \newcommand*{\hb}{5.8};
		\newcommand*{\va}{-1.5};
		\node (aa) at (0,0) {$(h \times g)^* \eop_{\cF',X} \cHom(\cF,\cG)$};
		\node (ab) at (\ha,1.0) {$(h \times g)^* \cHom(p^*_X(\cF), \eop_{\cF',X}(\cG))$};
		\node (ac) at (\ha+\hb,0) {$\cHom( (h \times g)^* p^*_X(\cF), (h \times g)^* \eop_{\cF',X}(\cG))$};
		%\node (ac) at (\ha+\hb,0) {$\cHom(p^*_U g^*(\cF), (h \times g)^* \eop_{\cF',X}(\cG))$};
		\node (ba) at (0,\va) {$ \eop_{h^*(\cF'),U} g^* \cHom(\cF,\cG)$};
		\node (bb) at (\ha,\va) {$\eop_{h^*(\cF'),U} \cHom(g^*(\cF),g^*(\cG)) $};
		\node (bc) at (\ha+\hb,\va) {$\cHom(p^*_U g^*(\cF), \eop_{h^*(\cF'),U}g^*(\cG)) $};
		\draw[->] (aa) to node[above] {$ $} (ab);
		\draw[->] (ab) to node[above] {$  $} (ac);
		\draw[->] (ba) to node[above] {$ $} (bb);
		\draw[->] (bb) to node[above] {$ $} (bc);
		\draw[->] (aa) to node[below,rotate=90] {$\sim $} (ba);
		%\draw[->] (ab) to node[right] {$h' $} (bb);
		\draw[->] (ac) to node[below,rotate=90] {$\sim $} (bc);
	\end{tikzpicture} 
\end{equation*}
is an isomorphism. This follows since the bottom left arrow is an isomorphism by Proposition~\ref{prop:cHomftdbasechange}, the top right is by this and Proposition~\ref{prop:cohextprod}, and the bottom right is by the previous paragraph. 
\end{proof}

\begin{Proposition}\label{prop:eFXeFXRcompat}
Suppose $\kk$ is an ordinary ring of finite global dimension. 
Let $X$, $Z$, and $Z'$ be geometric stacks, let $\cF \in \QCoh(Z)$ be almost perfect, and let $\cF' \in \QCoh(Z')^+$ be arbitrary. Then the Beck-Chevalley map $\eop_{\cF',X} e_{\cF,X}^R(\cG) \to e_{\cF, Z' \times X}^R \eop_{\cF',X \times Z}(\cG)$ is an isomorphism for all $\cG \in \QCoh(X \times Z)^+$. 
\end{Proposition}
\begin{proof}
By definition the given map factors as 
\begin{align*}
\eop_{\cF',X} p_{X*}\cHom(p_Z^*(\cF),\cG) &\to p_{Z' \times X*} \eop_{\cF',X \times Z} \cHom(p_Z^*(\cF),\cG)\\ &\to p_{Z' \times X *}\cHom(p_Z^*(\cF), \eop_{\cF',X \times Z}(\cG)),
\end{align*}
hence is an isomorphism by Propositions \ref{prop:eFXlower*compat} and \ref{prop:eFXcHomcompat}. 
\end{proof}

\subsection{Ind-geometric external products: the coherent case} \label{sec:indextcohcase}

We now consider external products of coherent sheaves on reasonable ind-geometric stacks. The constructions in the remaining sections will make crucial use of Proposition~\ref{prop:cohextprod}, hence \emph{we assume that $\kk$ is an ordinary ring of finite global dimension for the rest of the paper}. 

Let $X \cong \colim X_\al$ and $Z \cong \colim Z_\be$ be reasonable presentations. We will define 
\begin{equation}\label{eq:cohext}
- \boxtimes -: \Coh(X) \times \Coh(Z) \to \Coh(X \times Z)	
\end{equation}
 so that it fits into a diagram
\begin{equation}\label{eq:cohextcompat}
	\begin{tikzpicture}
		[baseline=(current  bounding  box.center),thick,>=\arrtip]
		\node (a) at (0,0) {$\Coh(X_\al) \times \Coh(Z_\be)$};
		\node (b) at (5.0,0) {$\Coh(X_\al \times Z_\be)$};
		\node (c) at (0,-1.5) {$\Coh(X) \times \Coh(Z) $};
		\node (d) at (5.0,-1.5) {$\Coh(X \times Z)$};
		\draw[->] (a) to node[above] {$- \boxtimes - $} (b);
		\draw[->] (b) to node[right] {$(i_{\al} \times i_\be)_* $} (d);
		\draw[->] (a) to node[left] {$i_{\al*} \times i_{\be*} $}(c);
		\draw[->] (c) to node[above] {$- \boxtimes - $} (d);
	\end{tikzpicture}
\end{equation}
for all $\al$, $\be$. The behavior of the ind-geometric external product on objects is completely determined by these diagrams, given the behavior of the geometric external product. 

Following \cite[Sec. 9]{GR17}, external products are most fully encoded as lax symmetric monoidal structures on sheaf theories. Recall that the restriction $\QCoh: \Affk^\op \to \Cathatinfty$ enhances to a symmetric monoidal functor $\QCoh: \Affk^\op  \to \PrStk$, where $\Affk^\op \cong \CAlgk$ and where $\PrStk$ denotes the category of presentable, stable $\Mod_\kk$-module categories \cite[Rem. 4.8.5.19]{LurHA}. It follows that $\QCoh: \Affk^\op \to \Cathatinfty$ is itself lax symmetric monoidal \cite[Cor. 4.8.1.4]{LurHA}. By Proposition~\ref{prop:cohextprod} we obtain a lax symmetric monoidal structure on the induced functor
\begin{equation}\label{eq:cohaffineextsec}
	\Coh: \Affkftd^\op \to \Cathatinfty,
\end{equation}
where $\Affkftd \subset \Affk$ is the 1-full subcategory which only includes morphisms of finite Tor-dimension. 
Now recall from Definition \ref{def:cohonindgstks} that the basic functorialities of coherent sheaves on reasonable ind-geometric stacks were packaged as a functor
\begin{equation}\label{eq:cohcorrindgstkfunctorextsec}
	\Coh: \Corr(\indGStkkreas)_{prop,ftd} \to \Catinfty
\end{equation}
extending (\ref{eq:cohaffineextsec}). 

\begin{Proposition}\label{prop:reascohlaxsymmstruct}
The functor (\ref{eq:cohcorrindgstkfunctorextsec}) has a canonical lax symmetric monoidal structure extending that of (\ref{eq:cohaffineextsec}). 
\end{Proposition}

The proof will use the following standard result, see \cite[Prop. 14.5.1]{Ras14} or, for a close variant, \cite[Lem. 2.16]{AFT17}. Here if $\catC$ is a symmetric monoidal category, then $\catC^\otimes \to \Fin_*$ denotes the associated coCartesian fibration.

\begin{Proposition}\label{prop:laxext}
	Let $\catC$, $\catD$ be symmetric monoidal categories, $\catC' \subset \catC$ a full symmetric monoidal subcategory, and $\Phi': \catC' \to \catD$ a lax symmetric monoidal functor. Suppose that $\catD$ admits small limits and that $\catC'_{/X} \times \catC'_{/Y} \to \catC'_{/X \otimes Y}$ is right cofinal for all $X, Y \in \catC$. Then the right Kan extension $\Phi: \catC \to \catD$ of $\Phi'$ admits a canonical lax symmetric monoidal structure extending that of $\Phi'$. Explicitly, it is the given by the right Kan extension $\Phi^\otimes: \catC^\otimes \to \catD^\otimes$ of $\Phi'^{\otimes}: \catC'^\otimes \to \catD^\otimes$ relative to $\Fin_*$. If instead $\catD$ admits small colimits and $\catC'_{X/} \times \catC'_{Y/} \to \catC'_{X \otimes Y/}$ is left cofinal for all $X, Y \in \catC$, then the corresponding claim holds for left Kan extensions.
\end{Proposition}

\begin{proof}[Proof of Proposition \ref{prop:reascohlaxsymmstruct}]
Let $\AlgSpk \subset \GStkk$ denote the full subcategory of (geometric, spectral) algebraic spaces. By Proposition \ref{prop:laxext} the symmetric monoidal structure on $\QCoh: \Affk^\op  \to \PrStk$ extends to a lax symmetric monoidal structure on its right Kan extension $\QCoh: \AlgSpk^\op  \to \PrStk$. By \cite[Cor. 9.4.2.3, Prop. 9.6.1.1]{LurSAG} this is in fact symmetric monoidal. By \cite[Cor. 1.2.2]{Ste20}, \cite[Cor. 3.4.2.2]{LurSAG} this extends to a symmetric monoidal structure functor on the functor $\QCoh: \Corr(\AlgSpk) \to \PrStk$, hence to a lax symmetric monoidal structure on the induced functor $\QCoh: \Corr(\AlgSpk) \to \Cathatinfty$. 

Propositions \ref{prop:laxext} and \ref{prop:corrKanext} extend this to a lax symmetric monoidal structure on the right Kan extension $\QCoh: \Corr(\GStkk)_{alg,all} \to \Cathatinfty$, where $alg$ is the class of relative algebraic spaces. Using Proposition~\ref{prop:cohextprod} this induces a lax symmetric monoidal structure on the restriction $\Coh: \Corr(\GStkkplus)_{prop,ftd} \to \Cathatinfty$. The functor (\ref{eq:cohcorrindgstkfunctorextsec}) is the left Kan extension of this, and again inherits a lax symmetric monoidal structure by Propositions \ref{prop:laxext} and \ref{prop:corrKanext}. 
\end{proof}

In particular, the data of the lax symmetric monoidal structure of Proposition \ref{prop:reascohlaxsymmstruct} includes the data of a functor (\ref{eq:cohext}) for any $X$ and $Z$, together with the data of the diagrams (\ref{eq:cohextcompat}). 

As a corollary, we obtain an external product of ind-coherent sheaves
on coherent ind-geometric stacks. Given Proposition \ref{prop:cohidcomp}, and inspecting the proof of Proposition \ref{prop:reascohlaxsymmstruct}, we first note that (\ref{eq:cohcorrindgstkfunctorextsec}) lifts to a lax symmetric monoidal functor to the category of idempotent-complete categories with finite colimits. If $X$ and $Z$ are coherent, it then follows from \cite[Lem. 5.3.2.11]{LurHA} that there is a unique extension of (\ref{eq:cohext}) to a continuous functor
\begin{equation}\label{eq:cohextcohcase}
	- \boxtimes -: \IndCoh(X) \otimes \IndCoh(Z) \to \IndCoh(X \times Z),
\end{equation}
where the left-hand term refers to the tensor product in $\PrL$. 

To obtain a global statement, note that the full subcategory $\Corr(\indGStkkNoeth)_{prop,ftd} \subset \Corr(\indGStkkreas)_{prop,ftd}$ of locally Noetherian ind-geometric stacks is a symmetric monoidal subcategory. We restrict to it the functor 
\begin{equation}\label{eq:IndCohindGStkextsec}
	\IndCoh: \Corr(\indGStkk)_{fcd;ftd} \to \PrL
\end{equation}
of Definition \ref{def:IndCohindGStk}. Together with \cite[Lem. 5.3.2.11]{LurHA}, Propositions \ref{prop:cohidcomp}, \ref{prop:IndCohisIndofCoh}, and \ref{prop:reascohlaxsymmstruct} then immediately imply the following. 

\begin{Proposition}\label{prop:cohIClaxsymmstruct}
The restriction of (\ref{eq:IndCohindGStkextsec}) to $\Corr(\indGStkkNoeth)_{prop,ftd}$ has a canonical lax symmetric monoidal structure which extends the restriction of the lax symmetric monoidal structure on (\ref{eq:cohcorrindgstkfunctorextsec}) defined by Proposition \ref{prop:reascohlaxsymmstruct}. 
\end{Proposition}
Note that the obstruction to extending this result to general coherent ind-geometric stacks is that these are not closed under products. However, in \cite{CWtm} we will extend it to a certain class of well-behaved coherent ind-geometric stacks, and this will encompass all motivating examples we have in mind (i.e. those appearing in \cite{CW2}). 

\subsection{Ind-geometric external products: the general case} 
If $X$ and $Z$ are coherent ind-geometric stacks and $\cG \in \IndCoh(X)$, $\cF \in \IndCoh(Z)$, we defined  the external product $\cG \boxtimes \cF \in \IndCoh(X \times Z)$ in the previous section. We now extend this definition to include the case where $X$ and $Z$ are not necessarily coherent, which will in turn require us to assume either $\cF$ or $\cG$ is bounded.

As a result, we will not attempt to generalize the full data of a lax symmetric monoidal structure on $\IndCoh(-)$ as in Proposition \ref{prop:cohIClaxsymmstruct}, as boundedness hypotheses make even formulating a precise generalization cumbersome. Moreover, the material in this section is not needed for our intended applications, which only concern the coherent setting.  But as in previous sections, the poor formal properties of coherent stacks mean that even if one only wants to prove a given result in the coherent setting, it is convenient if the structures involved in the proof are defined in the general ind-geometric setting, where one can make constructions more freely. 

To start, let $X$ and $Z$ be geometric stacks. The assignment $\cF \mapsto -\boxtimes \cF$ extends to a functor $\QCoh(Z) \to \LFun(\QCoh(X), \QCoh(X \times Z))$. By Proposition~\ref{prop:cohextprod} this restricts to a functor $\QCoh(Z)^b \to \LFun^b(\QCoh(X), \QCoh(X \times Z))$, where $\QCoh(Z)^b  \subset \QCoh(Z)$ is the full subcategory of bounded sheaves. Since $\QCoh(Z)^b \cong \IndCoh(Z)^b$, it follows from the universal property of $\IndCoh(-)$ that we have a canonical diagram of the following form. 
\begin{equation}\label{eq:eFXdef}
	\begin{tikzpicture}
		[baseline=(current  bounding  box.center),thick,>=\arrtip]
		\node (a) at (0,0) {$\IndCoh(X) \times \IndCoh(Z)^b$};
		\node (b) at (5.8,0) {$\IndCoh(X \times Z)$};
		\node (c) at (0,-1.5) {$\QCoh(X) \times \QCoh(Z)^b$};
		\node (d) at (5.8,-1.5) {$\QCoh(X \times Z)$};
		\draw[->] (a) to node[above] {$- \boxtimes - $} (b);
		\draw[->] (b) to node[right] {$\Psi_{X \times Z} $} (d);
		\draw[->] (a) to node[left] {$\Psi_X \times \Psi_Z $}(c);
		\draw[->] (c) to node[above] {$- \boxtimes - $} (d);
	\end{tikzpicture}
\end{equation}

We want to generalize the top arrow of this diagram to the case where $X$ and $Z$ are ind-geometric. As in the case of coherent sheaves, the behavior of this extension on objects will be determined by its compatibility with pushforward from truncated geometric substacks. To formalize this we will use the fact that the above diagram is functorial in the following sense. Here $\Corr(\GStkk)_{alg;ftd}$ denotes the 1-full subcategory of $\Corr(\GStkk)$ which only includes correspondences $X \xleftarrow{h} Y \xrightarrow{f} Z$ such that $h$ is of finite Tor-dimension and $f$ is a relative algebraic space. 

\begin{Proposition}\label{prop:eFXfunctoriality}
	There exists a diagram 
		\begin{equation}\label{eq:eFXfunctorialclaim}
		\begin{tikzpicture}
			[baseline=(current  bounding  box.center),thick,>=\arrtip]
			\node (a) at (0,0) {$\IndCoh(-) \times \IndCoh(-)^{b}$};
			\node (b) at (5.8,0) {$\IndCoh(- \times -)$};
			\node (c) at (0,-1.5) {$\QCoh(-) \times \QCoh(-)^b$};
			\node (d) at (5.8,-1.5) {$\QCoh(- \times -)$};
			\draw[->] (a) to node[above] {$- \boxtimes - $} (b);
			\draw[->] (b) to node[right] {$\Psi_{(- \times -)} $} (d);
			\draw[->] (a) to node[left] {$\Psi_{(-)} \times \Psi_{(-)} $}(c);
			\draw[->] (c) to node[above] {$- \boxtimes - $} (d);
		\end{tikzpicture}
	\end{equation}
of functors $\Corr(\GStkk)_{alg;ftd}^{\times 2} \to \Cathatinfty$ which specializes to the diagram (\ref{eq:eFXdef}) when evaluated on any $X, Z \in \GStkk$.
\end{Proposition} 

We postpone the proof of Proposition \ref{prop:eFXfunctoriality} while we use it define external products in the desired generality. To simplify the needed constructions we restrict our attention to the case where $Z$ is reasonable and $\cF$ is coherent. 

Note first that the top arrow of (\ref{eq:eFXfunctorialclaim}) can be encoded as a functor $\Corr(\GStkk)_{alg;ftd}^{\times 2} \to \Cathatinftysup{\Delta^1}$, where $\Cathatinftysup{\Delta^1} := \Fun(\Delta^1, \Cathatinfty)$. Restricting its domain and values we obtain a functor $\Corr(\GStkkplus)_{alg;ftd}  \times \Corr(\GStkkplus)_{prop;ftd} \to \Cathatinftysup{\Delta^1}$ of the form
$$ \IndCoh(-) \times \Coh(-) \to \IndCoh(- \times -). $$
For any $X$, $Z \in \GStkk$ the specialization of this expression preserves small colimits in $\IndCoh(X)$, hence there exists a unique extension to a functor $\Corr(\GStkkplus)_{alg;ftd}  \times \Corr(\GStkkplus)_{prop;ftd} \to (\PrL)^{\Delta^1}$ of the form
\begin{equation}\label{eq:PrLextprodgeom}
 \IndCoh(-) \ot \Ind(\Coh(-)) \to \IndCoh(- \times -). 
\end{equation} 

Next define a functor $\Corr(\indGStkk)_{alg;ftd} \times \Corr(\indGStkkreas)_{prop;ftd} \to (\PrL)^{\Delta^1}$ by left Kan extending (\ref{eq:PrLextprodgeom}). Here $\Corr(\indGStkk)_{alg;ftd}$ refers to correspondences whose forward morphism is a relative ind-algebraic space in the obvious sense. This extension exists, and moreover is of the same form as (\ref{eq:PrLextprodgeom}), since $(\PrL)^{\Delta^1}$ admits small colimits and  $(\PrL)^{\Delta^1} \to (\PrL)^{\times 2}$ preserves them \cite[Cor. 5.1.2.3]{LurHTT}, since the tensor product in $\PrL$ preserves small colimits in each variable \cite[Rem. 4.8.1.23]{LurHA}, and since ind-completion of idempotent-complete categories admitting finite colimits commutes with filtered colimits \cite[Lem. 7.3.5.11]{LurHA}. 

\begin{Definition}\label{def:eFXonindgstacks}
We define a functor $$\Corr(\indGStkk)_{alg;ftd}  \times \Corr(\indGStkkreas)_{prop;ftd} \to \Cathatinftysup{\Delta^1}$$ of the form
$$- \boxtimes -: \IndCoh(-) \times \Coh(-) \to \IndCoh(- \times -)$$
by taking the functor $\Corr(\indGStkk)_{alg;ftd} \times \Corr(\indGStkkreas)_{prop;ftd} \to (\PrL)^{\Delta^1}$ defined above, passing to its underlying $\Cathatinftysup{\Delta^1}$-valued functor, and then composing with the canonical natural transformation $\IndCoh(-) \times \Coh(-) \to \IndCoh(-) \ot \Ind(\Coh(-))$. 
\end{Definition}

When $X$ and $Z$ are reasonable, it follows by construction that the functor
$$ - \boxtimes -: \IndCoh(X) \times \Coh(Z) \to \IndCoh(X \times Z) $$
is compatible with the coherent external product (\ref{eq:cohext}) in the obvious sense. We now return to the proof of Proposition \ref{prop:eFXfunctoriality}. 
 
\begin{proof}[Proof of Proposition \ref{prop:eFXfunctoriality}]
First recall from the proof of Proposition \ref{prop:reascohlaxsymmstruct} that the functor $\QCoh: \Corr(\GStkk)_{alg;ftd} \to \Cathatinfty$ has a lax symmetric monoidal structure. Part of the data of this is the bottom arrow of (\ref{eq:eFXfunctorialclaim}), which we regard as a functor $\Corr(\GStkk)_{alg;ftd}^{\times 2} \to \Cathatinftysup{\Delta^1}$ taking $(X, Z)$ to the functor $\QCoh(X) \times \QCoh(Z)^b \to \QCoh(X \times Z)$. %Here we write $\Cathatinftysup{\Delta^1} := \Fun(\Delta^1, \Cathatinfty)$, similarly in other cases below. 

By construction this functor $\Corr(\GStkk)_{alg;ftd}^{\times 2} \to \Cathatinftysup{\Delta^1}$ factors through the category $\Cathat{}_{cc}^{b,\Delta^1}$ defined as follows. Using the notation from Section \ref{sec:anticompletion}, we set $\Pr_{cc} := \PrStbcpl \times \Cathatinfty \times \PrStbcpl$, regarding this as a category over $\Cathatinftysup{\times 2}$ via $(\wh{\catC}, \catD, \wh{\catE}) \mapsto (\wh{\catC} \times \catD, \wh{\catE})$. We then write $\Cathat{}_{cc}^{\Delta^1} := \Cathatinftysup{\Delta^1} \times_{\Cathatinftysup{\times 2}} \Pr_{cc}$ for the category of tuples  $(\wh{\catC}, \catD, \wh{\catE}, F)$, where $\wh{\catC}, \wh{\catE} \in \PrStbcpl$, $\catD \in \Cathatinfty$, and $F: \wh{\catC} \times \catD \to \wh{\catE}$. Finally, we write $\Cathat{}_{cc}^{b,\Delta^1}$ for the full subcategory of such tuples whose associated functor $\catD \to \Fun(\wh{\catC}, \wh{\catE})$ takes values in $\LFun^b(\wh{\catC}, \wh{\catE})$. 

We will prove the claim by constructing a functor $\Cathat{}_{cc}^{b,\Delta^1} \to \Cathatinftysup{\Delta^1 \times \Delta^1}$ which takes the bottom arrow in (\ref{eq:eFXdef}) to the entire diagram. Let us set $\Pr_{aa} := \PrStbacpl \times \Cathatinfty \times \PrStbacpl$ and $\Pr_{ac} := \PrStbacpl \times \Cathatinfty \times \PrStbcpl$, defining $\Cathat{}_{aa}^{\Delta^1}$, etc., as above. The main step will be to first construct a diagram
\begin{equation}\label{eq:eFXfunctorial3}
	\begin{tikzpicture}
		[baseline=(current  bounding  box.center),thick,>=\arrtip]
		\newcommand*{\ha}{3}; \newcommand*{\hb}{3};
		\newcommand*{\va}{-1.5};
		\node (aa) at (0,0) {$\Cathat{}_{aa}^{\Delta^1} $};
		\node (ab) at (\ha,0) {$\Cathat{}_{ac}^{\Delta^1}$};
		\node (ac) at (\ha+\hb,0) {$\Cathat{}_{cc}^{\Delta^1}$};
		\node (ba) at (0,\va) {$\Cathat{}_{aa}^{b,\Delta^1} $};
		\node (bb) at (\ha,\va) {$\Cathat{}_{ac}^{b,\Delta^1}$};
		\node (bc) at (\ha+\hb,\va) {$\Cathat{}_{cc}^{b,\Delta^1}$};
		\draw[->] (aa) to node[above] {$ $} (ab);
		\draw[<-] (ab) to node[above] {$ $} (ac);
		\draw[->] (ba) to node[above] {$ \sim $} (bb);
		\draw[<-] (bb) to node[above] {$ \sim $} (bc);
		\draw[right hook->] (ba) to node[left] {$  $} (aa);
		\draw[right hook->] (bb) to node[right] {$  $} (ab);
		\draw[right hook->] (bc) to node[right] {$  $} (ac);
	\end{tikzpicture}
\end{equation} 
in which the bottom functors are equivalences, and such that under these equivalences the bottom arrow in (\ref{eq:eFXdef}) corresponds respectively to the top arrow and the overall composition in~(\ref{eq:eFXdef}). 

Let us explicitly construct the top left functor in (\ref{eq:eFXfunctorial3}) and show that it restricts to an equivalence $\Cathat{}^{b,\Delta^1_{01}}_{aa} \congto \Cathat{}^{b,\Delta^1_{01}}_{ac}$; the construction of the right square is parallel. To do this we introduce the following pair of diagrams.
\begin{equation}\label{eq:eFXfunctorial4}
	\begin{tikzpicture}
		[baseline=(current  bounding  box.center),thick,>=\arrtip]
		\newcommand*{\ha}{2.7}; \newcommand*{\hb}{2.7};
		\newcommand*{\va}{-1.5};
		\newcommand*{\vb}{-1.5};
		
		\node [matrix] at (0,0) {
		\node (aa) at (0,0) {$\Cathatinftysup{\Delta^2} $};
		\node (ab) at (\ha,0) {$\Cathatinftysup{\Lambda^2_2}$};
		\node (ac) at (\ha+\hb,0) {$\Cathatinftysup{\Delta^1_{02}}$};
		\node (ba) at (0,\va) {$\Cathatinftysup{\Lambda^2_1} $};
		\node (bb) at (\ha,\va) {$\Cathatinftysup{\Delta^0_0 \cup \Delta^1_{12}}$};
		\node (bc) at (\ha+\hb,\va) {$\Cathatinftysup{\Delta^0_0 \cup \Delta^0_{2}}$};
		\node (ca) at (0,\va+\vb) {$\Cathatinftysup{\Delta^1_{01}}$};
		\node (cb) at (\ha,\va+\vb) {$\Cathatinftysup{\Delta^0_{0} \cup \Delta^0_1}$};
		\node (cc) at (\ha+\hb,\va+\vb) {$\Pr_{aa}$};
		\draw[->] (aa) to node[above] {$ $} (ab);
		\draw[->] (ab) to node[above] {$ $} (ac);
		\draw[->] (ba) to node[above] {$ $} (bb);
		\draw[->] (bb) to node[above] {$ $} (bc);
		\draw[->] (ca) to node[above] {$ $} (cb);
		\draw[<-] (cb) to node[above] {$ $} (cc);
		\draw[->] (aa) to node[left] {$  $} (ba);
		\draw[->] (ba) to node[left] {$  $} (ca);
		\draw[->] (ab) to node[right] {$  $} (bb);
		\draw[->] (bb) to node[right] {$  $} (cb);
		\draw[->] (ac) to node[right] {$  $} (bc);
		\draw[<-] (bc) to node[right] {$  $} (cc);
		\draw[->] (cc) to node[right] {$  $} (bb);\\
		};

\node [matrix] at (7.5,0) {
\node (aa) at (0,0) {$\Cathat{}^{\Delta^2}_{aac} $};
\node (ab) at (\ha,0) {$\Cathat{}^{\Lambda^2_2}_{aac}$};
\node (ac) at (\ha+\hb,0) {$\Cathat{}^{\Delta^1_{02}}_{ac}$};
\node (ba) at (0,\va) {$\Cathat{}^{\Lambda^2_1}_{aac} $};
\node (ca) at (0,\va+\vb) {$\Cathat{}^{\Delta^1_{01}}_{aa}$};
\node (cc) at (\ha+\hb,\va+\vb) {$\Pr_{aa}$};
\draw[->] (aa) to node[above] {$ $} (ab);
\draw[->] (ab) to node[above] {$ $} (ac);
\draw[->] (ca) to node[above] {$ $} (cc);
\draw[->] (aa) to node[left] {$  $} (ba);
\draw[->] (ba) to node[left] {$  $} (ca);
\draw[->] (ac) to node[right] {$  $} (cc);\\
};
	\end{tikzpicture}
\end{equation} 
Here the subscripts in e.g. $\Delta^1_{02}$ indicate a particular 1-simplex of $\Delta^2$, and the arrows in the left diagram not involving $\Pr_{aa}$ are induced by restriction. The unit of the localization $\catC \mapsto \wh{\catC}$ on $\PrStbrcpl$ induces a functor $\PrStbacpl \to \Cathatinftysup{\Delta^1}$ taking $\wc{\catC}$ to $\wc{\catC} \to \wh{\catC}$, and the diagonal arrow out of $\Pr_{aa}$ is the induced functor $(\wc{\catC}, \catD, \wc{\catE}) \mapsto (\wc{\catC} \times \catD, \wc{\catE} \to \wh{\catE})$. The horizontal and vertical arrows out of $\Pr_{aa}$ thus take $(\wc{\catC}, \catD, \wc{\catE})$ to  $(\wc{\catC} \times \catD, \wc{\catE})$ and $(\wc{\catC} \times \catD, \wh{\catE})$, respectively. 
	
	In the right diagram, $\Cathat{}^{\Delta^1_{01}}_{aa}$ and $\Cathat{}^{\Delta^1_{02}}_{ac}$ are respectively the fiber products of the bottom row and right column of the left diagram (which is consistent with our existing notation after forgetting subscripts). The remaining three categories are the fiber products of their counterparts on the left with $\Pr_{aa}$ over $\Cathatinftysup{\Delta^0_0 \cup \Delta^1_{12}}$. Note that their natural maps to $\Pr_{aa}$ indeed factor through those of $\Cathat{}^{\Delta^1_{01}}_{aa}$ and $\Cathat{}^{\Delta^1_{02}}_{ac}$ as indicated. 
	
	We claim the leftmost vertical functors in the right diagram are equivalences. For the top left, this follows since it is base changed from its counterpart on the left, which is an equivalence by \cite[Cor. 2.3.2.2]{LurHTT}. For the bottom left, this follows from the bottom left square of the left diagram being Cartesian. Composing the inverse equivalences with the top arrows we obtain a functor $\Cathat{}^{\Delta^1_{01}}_{aa} \to \Cathat{}^{\Delta^1_{01}}_{ac}$ as desired. 
	
	The fiber of this functor over a particular $(\wc{\catC},\catD, \wc{\catE}) \in \Pr_{aa}$ is the map $\Fun(\wc{\catC} \times \catD, \wc{\catE})^{\cong} \to \Fun(\wc{\catC} \times \catD, \wh{\catE})^{\cong}$ given by composition with $\wc{\catE} \to \wh{\catE}$ (the superscripts indicate that non-invertible natural transformations are excluded). Since the corresponding map $\LFun^b(\wc{\catC},\wc{\catE}) \to \LFun^b(\wc{\catC},\wh{\catE})$ is an equivalence by definition, it follows that $\Cathat{}^{\Delta^1_{01}}_{aa} \to \Cathat{}^{\Delta^1_{01}}_{ac}$ restricts to a functor $\Cathat{}^{b,\Delta^1_{01}}_{aa} \to \Cathat{}^{b,\Delta^1_{01}}_{ac}$ which in turns restricts to an isomorphism of fibers over $\Pr_{aa}$. 
	
	Recall that $\Cathatinftysup{\Delta^1}$ is a bifibration over $\Cathatinftysup{\times 2}$ \cite[Cor. 2.4.7.11]{LurHTT}. It follows from the definitions that bifibrations are stable under pullback along products of maps and under restriction to full subcategories. In particular, $\Cathat{}^{b,\Delta^1_{01}}_{aa}$ and $\Cathat{}^{b,\Delta^1_{01}}_{ac}$ are bifibrations over $\Pr_{aa}$, factored as the product of $\PrStbacpl \times \Cathatinfty$ and $\PrStbacpl$. It now follows from \cite[Prop. 2.4.7.6]{LurHTT} and the previous paragraph that $\Cathat{}^{b,\Delta^1_{01}}_{aa} \to \Cathat{}^{b,\Delta^1_{01}}_{ac}$ is an equivalence. 
	
	To complete the proof, note that by construction the bottom row of (\ref{eq:eFXfunctorial3}) factors as 
	\begin{equation*}\label{eq:eFXfunctorial2}
		\Cathat{}_{aa}^{b,\Delta^1_{01}} \xleftarrow{\sim} \Cathat{}_{aac}^{b,\Delta^2} \congto \Cathat{}_{ac}^{b,\Delta^1_{02}} \xleftarrow{\sim} \Cathat{}_{acc}^{b,\Delta^2} \congto \Cathat{}_{cc}^{b,\Delta^1_{12}}.
	\end{equation*}  
Here we again use subscripts to indicate edges in $\Delta^2$, $\Cathat{}_{acc}^{\Delta^2}$ is the evident counterpart of $\Cathat{}_{aac}^{\Delta^2}$, and $\Cathat{}_{acc}^{b,\Delta^2}$, $\Cathat{}_{aac}^{b,\Delta^2}$ are the full subcategories corresponding to $\Cathat{}_{aa}^{b,\Delta^1}$. The middle three terms in the above factorization map respectively to $\Cathatinftysup{\Delta^2}$, $\Cathatinftysup{\Delta^1}$, and $\Cathatinftysup{\Delta^2}$ compatibly with the relevant maps, hence we obtain the desired functor
$$ \Cathat{}_{cc}^{b,\Delta^1} \to \Cathatinftysup{\Delta^2} \times_{\Cathatinftysup{\Delta^1_{02}}} \Cathatinftysup{\Delta^2} \cong \Cathatinftysup{\Delta^1 \times \Delta^1}.\qedhere$$
\end{proof}

\subsection{Ind-geometric external products: properties}

Let $X$ and $Z$ be ind-geometric stacks with $Z$ reasonable, and suppose $\cF \in \Coh(Z)$. We again write $e_{\cF, X}$ for the induced functor $- \boxtimes \cF: \IndCoh(X) \to \IndCoh(X \times Z)$. We now extend a few basic results about $e_{\cF,X}$ from the geometric setting, in particular its compatibility with proper $!$-pullback. 

First we record more explicitly the naturality of $e_{\cF,X}$ implied by Definition \ref{def:eFXonindgstacks}. If $f: X \to Y$ is a relative ind-algebraic space (for example, an ind-affine morphism such as the inclusion of a truncated geometric substack), $g: Z' \to Z$ is ind-proper and almost ind-finitely presented, and $\cF' \in \Coh(Z')$, then we have an isomorphism 
\begin{equation}\label{eq:eFXlow*compat}
e_{g_*(\cF'),Y} f_* \cong (f \times g)_* e_{\cF',X}.
\end{equation}
Likewise if $h: X \to Y$ and $g: Z' \to Z$ are of finite Tor-dimension and $\cF \in \Coh(Z)$, then we have an isomorphism
\begin{equation}\label{eq:eFXup*compat}
	e_{g^*(\cF),X} h^* \cong (h \times g)^* e_{\cF,Y}.
\end{equation}

\begin{Proposition}\label{prop:eFXbounded}
	Let $X$ and $Z$ be ind-geometric stacks such that $Z$ is reasonable, and suppose $\cF \in \Coh(Z)$. Then $e_{\cF,X}$ is bounded. 
\end{Proposition}
\begin{proof}
	Fix an ind-geometric presentation $X \cong \colim X_\al$ and write $\cF \cong i_*(\cF')$, where $i: Z' \to Z$ is a reasonable geometric substack and $\cF' \in \Coh(Z') \cap \IndCoh(Z')^{[m,n]}$. If $\cG_\al \in \IndCoh(X_\al)^{\geq 0}$ for some $\al$, then by  t-exactness of $(i_\al \times i)_*$ and the proof of Proposition~\ref{prop:cohextprod} we have $e_{\cF,X}i_{\al*}(\cG_\al) \cong (i_\al \times i)_* e_{\cF',X_\al}(\cG_\al) \in \IndCoh(X)^{\geq m'}$, where $m'$ is $m$ minus the global dimension of $\kk$. Similarly, if $\cG_\al \in \IndCoh(X_\al)^{\leq 0}$ then $e_{\cF,X}i_{\al*}(\cG_\al) \in \IndCoh(X)^{\leq n}$. 
	
	Given that $\cG \cong \colim i_{\al*} i_{\al}^!(\cG)$ for any $\cG\in \IndCoh(X)$ (Proposition \ref{prop:indcohonindgstks} and Lemma \ref{lem:colimpres}), it follows that $e_{\cF,X}$ takes $\IndCoh(X)^{\geq 0}$ to $\IndCoh(X \times Z)^{\geq m'}$ since $i_{\al*} i_{\al}^!$ is left t-exact and $\IndCoh(X \times Z)^{\geq m'}$ is closed under filtered colimits. If $\cG\in \IndCoh(X)^{\leq 0}$, then we additionally have $\cG \cong \colim i_{\al*} \tau^{\leq 0} i_{\al}^!(\cG)$. This again follows from Proposition \ref{prop:IndCohindGStktstructure} and Lemma~\ref{lem:colimpres}, given that $\tau^{\leq 0} i_\al^!$ is right adjoint to the restriction $i_{\al*}: \IndCoh(X_\al)^{\leq 0} \to \IndCoh(X)^{\leq 0}$. It now follows that $e_{\cF,X}$ takes $\IndCoh(X)^{\leq 0}$ to $\IndCoh(X \times Z)^{\leq n}$. 
\end{proof}

We now extend Proposition \ref{prop:eFXupper!compat} to the ind-geometric setting. For simplicity we give a proof assuming coherence hypotheses, then indicate how the statement may be generalized. 

\begin{Proposition}\label{prop:eFXupper!compatindgeom}
	Let $X$, $Y$, and $Z$ be ind-geometric stacks with $X$, $Y$, $X \times Z$, and $Y \times Z$ coherent, and let $f: X \to Y$ be an ind-proper, almost ind-finitely presented morphism. Then for all $\cF \in \Coh(Z)$ and $\cG \in \IndCoh(Y)$ the Beck-Chevalley map $e_{\cF,X}f^!(\cG) \to  (f \times \id_Z)^!e_{\cF, Y} (\cG)$ is an isomorphism. 
\end{Proposition}
\begin{proof}
	First suppose $X$, $Y$, and $Z$ are truncated and geometric. Then $e_{\cF,X}f^!(\cG)$ and $ (f \times \id_Z)^!e_{\cF, Y}$ are continuous and $\IndCoh(Y)$ is compactly generated, so it suffices to consider $\cG \in \Coh(Y)$. But the restrictions of all functors involved to left bounded subcategories commute with the equivalences $\IndCoh(-)^+ \cong \QCoh(-)^+$, so the claim follows from Proposition \ref{prop:eFXupper!compat}. 

	Now suppose that $Z$ is truncated and geometric, and that $f$ is the inclusion of a reasonable geometric substack, which we may assume is a term in a reasonable presentation $Y \cong \colim Y_\al$. By construction the functors $e_{\cF,Y_\al}$, $i_{\al*}$, and $(i_\al \times \id_Z)_*$ form a diagram $A \times \Delta^1 \to \Cathatinfty$, where $A$ is our index category.	Each $Y_\al \times Z$ is coherent by Proposition \ref{prop:cohbasechangeunderafpclimm}, hence by the previous paragraph $e_{\cF,Y_\al}i_{\al\be}^! \to  (i_{\al\be} \times \id_Z)^!e_{\cF, Y_\be}$ is an isomorphism for all $\al, \be$. 
	Since $\IndCoh(Y) \cong \colim \IndCoh(Y_\al)$ and  $\IndCoh(Y \times Z) \cong \colim \IndCoh(Y_\al \times Z)$ in $\PrL$ (Proposition \ref{prop:indcohonindgstks}), the claim follows by \cite[Prop. 4.7.5.19]{LurHA}. 
	
	Still assuming $Z$ is truncated and geometric, let $X \cong \colim X_\al$ be a reasonable presentation. For any $\al$ we can find a reasonable geometric substack $j_\al: Y_\al \to Y$ fitting into a diagram
\begin{equation*}
	\begin{tikzpicture}[baseline=(current  bounding  box.center),thick,>=\arrtip]
			\newcommand*{\ha}{1.5}; \newcommand*{\hb}{1.5}; \newcommand*{\hc}{1.5};
\newcommand*{\va}{-.9}; \newcommand*{\vb}{-.9}; \newcommand*{\vc}{-.9}; 
		\node (ab) at (\ha,0) {$X_\al$};
		\node (ad) at (\ha+\hb+\hc,0) {$X_\al \times Z$};
		\node (ba) at (0,\va) {$X$};
		\node (bc) at (\ha+\hb,\va) {$X \times Z$};
		\node (cb) at (\ha,\va+\vb) {$Y_\al$};
		\node (cd) at (\ha+\hb+\hc,\va+\vb) {$Y_\al \times Z$};
		\node (da) at (0,\va+\vb+\vc) {$Y$};
		\node (dc) at (\ha+\hb,\va+\vb+\vc) {$Y \times Z$};
		%\draw[->] (ab) to node[above] {$\phi' $} (ad);
		\draw[<-] (ab) to node[above] {$  $} (ad);
		%\draw[->] (ab) to node[above left] {$\theta' $} (ba);
		\draw[->] (ab) to node[above left, pos=.25] {$i_{\al} $} (ba);
		%\draw[->] (ab) to node[left,pos=.8] {$\psi' $} (cb);
		\draw[->] (ab) to node[left,pos=.8] {$f_\al $} (cb);
		\draw[->] (ad) to node[below right] {$  $} (bc);
		\draw[->] (ad) to node[right] {$  $} (cd);
		%\draw[->] (ba) to node[left] {$h' $} (da);
		\draw[->] (ba) to node[left] {$f $} (da);
		\draw[<-] (cb) to node[above,pos=.25] {$  $} (cd);
		%\draw[->] (cb) to node[above left] {$g' $} (da);
		\draw[->] (cb) to node[above left, pos=.25] {$j_\al $} (da);
		\draw[->] (cd) to node[below right] {$ $} (dc);
		\draw[<-] (da) to node[above,pos=.75] {$ $} (dc);
		
		\draw[-,line width=6pt,draw=white] (ba) to  (bc);
		%\draw[->] (ba) to node[above,pos=.75] {$f' $} (bc);
		\draw[<-] (ba) to node[above,pos=.75] {$ $} (bc);
		\draw[-,line width=6pt,draw=white] (bc) to  (dc);
		\draw[->] (bc) to node[right,pos=.2] {$ $} (dc);
	\end{tikzpicture}
\end{equation*}
	in which all but the left and right faces are Cartesian. We then have a diagram
	\begin{equation*}
		\begin{tikzpicture}
			[baseline=(current  bounding  box.center),thick,>=\arrtip]
			\newcommand*{\ha}{4.0}; \newcommand*{\hb}{5.4};
			\newcommand*{\va}{-1.5};
			\node (aa) at (0,0) {$e_{\cF,X_\al} i^!_\al f^!(\cG)$};
			\node (ab) at (\ha,0) {$(i_\al \times \id_Z)^! e_{\cF,X} f^!(\cG)$};
			\node (ac) at (\ha+\hb,0) {$(i_\al \times \id_Z)^! (f \times \id_Z)^! e_{\cF,Y}(\cG)$};
			\node (ba) at (0,\va) {$ e_{\cF,X_\al} f^!_\al j^!_\al(\cG)$};
			\node (bb) at (\ha,\va) {$ (f_\al \times \id_Z)^! e_{\cF,Y_\al} j^!_\al(\cG)$};
			\node (bc) at (\ha+\hb,\va) {$(f_\al \times \id_Z)^! (j_\al \times \id_Z)^! e_{\cF,Y} (\cG)$};
			\draw[->] (aa) to node[above] {$ $} (ab);
			\draw[->] (ab) to node[above] {$  $} (ac);
			\draw[->] (ba) to node[above] {$ $} (bb);
			\draw[->] (bb) to node[above] {$ $} (bc);
			\draw[->] (aa) to node[below,rotate=90] {$\sim $} (ba);
			%\draw[->] (ab) to node[right] {$h' $} (bb);
			\draw[->] (ac) to node[below,rotate=90] {$\sim $} (bc);
		\end{tikzpicture} 
	\end{equation*}
in $\IndCoh(X_\al \times Z)$. Since the functors $(i_\al \times \id_Z)^!$ determine an isomorphism $\IndCoh(X \times Z) \cong \lim \IndCoh(X_\al \times Z)$ in $\Cathatinfty$, it suffices to show the top right arrow is an isomorphism for all $\al$. But, given that $X_\al \times Z$ and $Y_\al \times Z$ are coherent by Proposition \ref{prop:cohbasechangeunderafpclimm}, the bottom right and top left arrows are isomorphisms by the previous paragraph, and the bottom left is by the first paragraph. 

Now let $Z$ be ind-geometric, and write $\cF \cong i_*(\cF')$ for some reasonable geometric substack $i: Z' \to Z$ and some $\cF' \in \Coh(Z')$. By (\ref{eq:eFXlow*compat}) we have isomorphisms $e_{\cF,Y} \cong (\id_Y \times i)_* e_{\cF', Y}$ and $e_{\cF,X} \cong (\id_X \times i)_* e_{\cF', X}$. Thus we are trying to show the composition
$$ (\id_X \times i)_* e_{\cF', X}f^! \to  (\id_X \times i)_* (f \times \id_{Z'})^! e_{\cF', Y} \to (f \times \id_Z)^! (\id_Y \times i)_* e_{\cF', Y} $$ 
is an isomorphism. But, given that $X \times Z'$ and $Y \times Z'$ are coherent by Proposition \ref{prop:cohbasechangeunderafpclimm}, the first factor is an isomorphism by the previous paragraph, and the second factor is by Proposition \ref{prop:up!low*geom}. 
\end{proof}

With more care, one can show the following weaker result in the general case (recall the definition of $\IndCoh(-)^+_{lim}$ from Section \ref{sec:!pullpush}). 

\begin{Proposition}%\label{prop:eFXupper!compatindgeom}
Let $X$, $Y$, and $Z$ be ind-geometric stacks with $Z$ reasonable, let $f: X \to Y$ be an ind-proper, almost ind-finitely presented morphism. Then $e_{\cF, Y}$ takes $\IndCoh(Y)^+_{lim}$ to $\IndCoh(Y \times Z)^+_{lim}$ for all $\cF \in \Coh(Z)$, and for all $\cG \in \IndCoh(Y)^+_{lim}$ the Beck-Chevalley map $e_{\cF,X}f^!(\cG) \to  (f \times \id_Z)^!e_{\cF, Y} (\cG)$ is an isomorphism. 
\end{Proposition}

We have the following companion to Proposition~\ref{prop:ICconsistenctdef}, which says that Definition~\ref{def:eFXonindgstacks} behaves as expected on non-truncated geometric stacks. 

\begin{Proposition}\label{prop:eFXconsistenctdef} 
	Let $X$ and $Z$ be geometric stacks such that $Z$ is reasonable, and suppose $\cF \in \Coh(Z)$. Then we have an isomorphism $\Psi_{X \times Z}e_{\cF,X} \cong e_{\Psi_Z(\cF),X} \Psi_X$ of functors $\IndCoh(X) \to \QCoh(X \times Z)$. 
\end{Proposition}
\begin{proof}
	Let $X \cong \colim X_\al$ , $Z \cong \colim Z_\be$ be respectively  an ind-geometric and a reasonable presentation, and write $\cF \cong i_{\be*}(\cF_\be)$ for some $\be$ and $\cF_\be \in \Coh(Z_\be)$. By Proposition~\ref{prop:eFXfunctoriality} the functors $e_{i_{\be\ga*}(\cF_\be),X_\al}$ 
	form a filtered system in $(\Pr^L)^{\Delta^1}$ which lifts to a filtered system in $(\Pr^L)^{\Delta^1}_{/e_{\Psi_Z(\cF),X}}$ (i.e. given termwise by taking $e_{i_{\be\ga*}(\cF_\be),X_\al}$ to the diagram realizing the isomorphism $(i_{\al} \times i_{\ga})_{*,QC} \Psi_{X_\al \times Z_\ga} e_{i_{\be\ga*}(\cF_\be),X_\al} \cong e_{\Psi_Z(\cF),X} i_{\al*,QC} \Psi_{X_\al}$). By Proposition~\ref{prop:indcohonindgstks} and \cite[Prop. 1.2.13.8]{LurHTT} its colimit in $(\Pr^L)^{\Delta^1}_{/e_{\Psi_Z(\cF),X}}$ is a diagram whose top and bottom arrows are $e_{\cF,X}$ and $e_{\Psi_Z(\cF),X}$. But the vertical arrows in this diagram are t-exact and induce equivalences of left completions by Proposition~\ref{prop:IndCohindGStktstructure}, and by t-exactness of the $\Psi_{(-)}$ functors and the pushforward functors in the filtered system, hence they are isomorphic to $\Psi_X$ and $\Psi_{X \times Z}$. 
\end{proof}

\subsection{Ind-coherent sheaf Hom}
Let $X$ and $Z$ be ind-geometric stacks such that $Z$ is reasonable, and suppose $\cF \in \Coh(Z)$. By construction $e_{\cF, X}$ has a right adjoint $ e_{\cF,X}^R: \IndCoh(X \times Z) \to \IndCoh(X)$.   When $X = Z$, we define ind-coherent sheaf Hom via the formula
\begin{equation}\label{eq:cHom}
 \cHom(\cF, - ) := e_{\cF,X}^R \Delta_{X*}: \IndCoh(X) \to \IndCoh(X).
\end{equation}
In this section we generalize various basic properties about quasi-coherent sheaf Hom on geometric stacks to this setting, in particular its compatibility with pushforward (Propositions~\ref{prop:eFXRlower*} and \ref{prop:indhomrelationcoherent}) and external products (Proposition \ref{prop:eFXReFXindgeomcoh}). We begin with the following justification of definition (\ref{eq:cHom}). 

\begin{Proposition}\label{prop:eFXRprops1}
Let $X$ and $Z$ be ind-geometric stacks such that $Z$ is reasonable, and let $\cF \in \Coh(Z)$. Then $e_{\cF,X}^R$ is left bounded. If $X$ and $Z$ are geometric, the Beck-Chevalley map $\Psi_X e_{\cF,X}^R(\cG) \to e_{\Psi_Z(\cF),X}^R \Psi_{X \times Z} (\cG)$ is an isomorphism for all $\cG \in \IndCoh(X \times Z)^+$, and the induced map $\Psi_X \cHom(\cF,\cG) \to \cHom(\Psi_X(\cF),\Psi_X(\cG))$ is an isomorphism for all $\cG \in \IndCoh(X)^+$. 
\end{Proposition}
\begin{proof}
 Since $e_{\cF, X}$ is bounded (Proposition \ref{prop:eFXbounded}),  $e_{\cF,X}^R$ is left bounded and the two functors restrict to an adjunction between $\IndCoh(X)^+$ and $\IndCoh(X \times Z)^+$. The analogous statement holds for $e_{\Psi_Z(\cF), X}$ and $e_{\Psi_Z(\cF),X}^R$, and the second claim follows and since $\Psi_{(-)}$ restricts to an equivalence $\IndCoh(-)^+ \congto \QCoh(-)^+$ and since $\Psi_{X \times Z} e_{\cF,X} \cong e_{\Psi_Z(\cF),X} \Psi_X$ (Proposition~\ref{prop:eFXconsistenctdef}). The third follows since $\Delta_{X*}$ is also compatible with the $\Psi_{(-)}$ functors, and since we have an isomorphism $- \ot \Psi_Z(\cF) \cong \Delta_X^* e_{\Psi_Z(\cF),X}$ of functors $\QCoh(X) \to \QCoh(X)$. 
\end{proof}

Propositions \ref{prop:cHomalmostcont}, \ref{prop:cHomftdbasechange}, and \ref{prop:eFXRprops1} immediately imply the following. 

\begin{Corollary}\label{cor:cHomalmostcontind}
	If $Y$ is a geometric stack and $\cF \in \QCoh(Y)$ is coherent, then $\cHom(\cF,-): \IndCoh(Y) \to \IndCoh(Y)$ is left bounded and almost continuous. 
\end{Corollary} 

\begin{Corollary}\label{cor:cHomftdbasechangeind}
	Let $X$ and $Y$ be geometric stacks, $h: X \to Y$ a morphism of finite Tor-dimension, and $\cF \in \Coh(Y)$. Then the Beck-Chevalley map $h^* \cHom(\cF, \cG) \to \cHom(h^*(\cF), h^*(\cG))$ is an isomorphism for all $\cG \in \IndCoh(Y)^+$. 
\end{Corollary}

Next we observe the following naturality properties of $e_{\cF,X}^R$. Suppose that $f: X \to Y$ and $g: Z' \to Z$ are ind-proper morphisms of ind-geometric stacks, that $Z'$ and $Z$ are reasonable, and that $g$ is almost ind-finitely presented. Then if $\cF \in \Coh(Z')$ and $\cF \cong g_*(\cF')$, the isomorphism (\ref{eq:eFXlow*compat}) yields an isomorphism of right adjoints
\begin{equation}\label{eq:eFXRupper!com}
	e_{\cF',X}^R (f \times g)^! \cong f^! e_{\cF,Y}^R.
\end{equation}

Similarly, suppose that $h:X \to Y$ and $g: Z' \to Z$ are of finite Tor-dimension as well as of ind-finite cohomological dimension. Then if $\cF \in \Coh(Z)$ and $\cF' \cong g^*(\cF)$, the isomorphism (\ref{eq:eFXup*compat}) yields (implicitly using Proposition \ref{prop:up*low*adj}) an isomorphism of right adjoints
\begin{equation}\label{eq:eFXRlower*com}
	e_{\cF,Y}^R  (h \times g)_* \cong h_* e_{\cF',X}^R.
\end{equation}

We now consider the generalizations of Propositions \ref{prop:eFXRQCohalmostcont} and \ref{prop:cHomalmostcont}. 

\begin{Proposition}\label{prop:eFXRalmostcont}
	Let $X$ and $Z$ be reasonable ind-geometric stacks, and suppose $\cF \in \Coh(Z)$. Then $ e_{\cF,X}^R$ is almost continuous. If $X$ and $X \times Z$ are coherent, then $ e_{\cF,X}^R$ is continuous.
\end{Proposition} 

\begin{proof}
The second claim follows since by construction $e_{\cF,X}$ preserves coherence. When $X$ and $Z$ are truncated geometric stacks the first claim holds by Propositions \ref{prop:eFXRQCohalmostcont} and \ref{prop:eFXRprops1}. Still assuming $Z$ is truncated and geometric, let $X \cong \colim X_\al$ be a reasonable presentation and $\cG \cong \colim \cG_\be$ a filtered colimit in $\IndCoh(X \times Z)^{\geq 0}$. Since $\IndCoh(X) \cong \lim \IndCoh(X_\al)$ in $\Cathatinfty$, it suffices to show the second factor in 
	$$ \colim_\be i_\al^! e_{\cF,X}^R(\cG_\be) \to  i_\al^! \colim_\be e_{\cF,X}^R(\cG_\be) \to i_\al^! e_{\cF,X}^R( \colim_\be \cG_\be)  $$
	is an isomorphism for all $\al$. The first factor is an isomorphism since $e_{\cF,X}^R$ is left bounded (Proposition \ref{prop:eFXRprops1}) and $i_\al^!$ is almost continuous (Proposition \ref{prop:upper!almostcontind}). But $i_\al^! e_{\cF,X}^R \cong e_{\cF,X_\al}^R (i_\al \times \id_Z)^!$ by (\ref{eq:eFXRupper!com}), so the composition is an isomorphism by the left t-exactness of $(i_\al \times \id_Z)^!$ and the almost continuity of $e_{\cF,X_\al}^R$ and $(i_\al \times \id_Z)^!$.  
	
	Finally, suppose $Z \cong \colim Z_\al$ is a reasonable presentation, and write $\cF \cong i_{\al*}(\cF_\al)$ for some~$\al$ and $\cF_\al \in \Coh(Z_\al)$. By (\ref{eq:eFXRupper!com}) we have $e_{\cF,X}^R \cong e_{\cF_\al,X}^R i_\al^!$, and the claim follows since $i_\al^!$ is left t-exact and since $e_{\cF_\al,X}^R$ and $i_\al^!$ are almost continuous. 
\end{proof}
	
	\begin{Corollary}
Let $X$ be a reasonable (resp. coherent) ind-geometric stack and $\cF \in \Coh(X)$. Then $\cHom(\cF,-): \IndCoh(X) \to \IndCoh(X)$ is almost continuous (resp. continuous).
	\end{Corollary}
\begin{proof}
Follows from Proposition \ref{prop:eFXRalmostcont} and continuity of $\Delta_{X*}$. 
\end{proof}

If $X$, $Y$, and $Z$ are geometric stacks and $\cF \in \QCoh(Z)$, then for any $f: X \to Y$ the isomorphism $(f\times \id_Z)^* e_{\cF, Y} \cong e_{\cF,X} f^*$ of functors $\QCoh(Y) \to \QCoh(X \times Z)$ yields an isomorphism 
\begin{equation}\label{eq:QCeFXlow*}
	f_* e_{X,\cF}^R \cong e_{Y,\cF}^R (f\times \id_Z)_*
\end{equation}
of right adjoints $\QCoh(X \times Z) \to \QCoh(Y)$. This is an external counterpart of the isomorphism (\ref{eq:QCHomupper*}). 

Now suppose that $X$, $Y$, and $Z$ are ind-geometric and $f: X \to Y$ is of ind-finite cohomological dimension. In this setting $f_*: \IndCoh(X) \to \IndCoh(Y)$ typically does not have a left adjoint. We can still define an analogue of (\ref{eq:QCeFXlow*}), however, by considering the Beck-Chevalley transformation
\begin{equation}\label{eq:ICeFXlow*}
	f_* e_{\cF,X}^R \to e_{\cF,Y}^R (f\times \id_Z)_*
\end{equation} 
associated to the isomorphism
$(f\times \id_Z)_* e_{\cF,X} \cong e_{\cF,Y}f_*$ in $\Fun(\IndCoh(X) , \IndCoh(Y \times Z))$. 
In the geometric case one can check that if we restrict to bounded below subcategories, (\ref{eq:ICeFXlow*}) is identified with the isomorphisms (\ref{eq:QCeFXlow*}) under the equivalences $\IndCoh(-)^+ \cong \QCoh(-)^+$.

\begin{Proposition}\label{prop:eFXRlower*}
	Let $X$, $Y$, and $Z$ be coherent ind-geometric stacks such that $X \times Z$ and $Y \times Z$ are coherent, and let $f: X \to Y$ be a morphism of ind-finite cohomological dimension. 
	Then for any $\cF \in \Coh(Z)$ and $\cG \in \IndCoh(X \times Z)$, the Beck-Chevalley map  $f_* e_{\cF,X}^R(\cG) \to e_{\cF,Y}^R (f \times \id_Z)_*(\cG)$ is an isomorphism. 
\end{Proposition} 
\begin{proof}
	By Proposition \ref{prop:eFXRalmostcont}, all functors in the statement are continuous, hence by coherence of $X \times Z$ it suffices to show the claim for $\cG \in \Coh(X \times Z)$. When $X$, $Y$, and $Z$ are geometric and $Z$ is truncated the claim then follows from (\ref{eq:QCeFXlow*}), since these functors are also left bounded (Proposition \ref{prop:eFXRprops1}) and compatible with the equivalences $\IndCoh(-)^+ \cong \QCoh(-)^+$.
	
	Next suppose $f$ is the inclusion of a term in a reasonable presentation $Y \cong \colim Y_\al$, still assuming $Z$ is truncated and geometric. By construction the functors $e_{\cF,Y_\al}^R$, $i^!_{\al}$, and $(i_\al \times \id_Z)^!$ form a diagram $(A \times \Delta^1)^{\op} \to \Cathatinfty$, where $A$ is our index category. Each $Y_\al \times Z$ is coherent by Proposition \ref{prop:cohbasechangeunderafpclimm}, hence by the first paragraph $i_{\al\be*} e_{\cF,Y_\al}^R \to e_{\cF,Y_\be}^R (i_{\al\be} \times \id_Z)_*$ is an isomorphism for all $\al \leq \be$. Since  $\IndCoh(Y) \cong \colim \IndCoh(Y_\al)$ and  $\IndCoh(Y \times Z) \cong \colim \IndCoh(Y_\al \times Z)$ in $\PrL$ (Proposition \ref{prop:indcohonindgstks}), the claim follows by \cite[Prop. 4.7.5.19]{LurHA}. 

	Now let $X \cong \colim X_\al$ be a reasonable presentation, supposing again that $Y$ and $Z$ are geometric and $Z$ is truncated. Since $\cG \in \Coh(X \times Z)$ we have $\cG \cong (i_\al \times \id_Z)_*(\cG_\al)$ for some~$\al$ and some $\cG_\al \in \Coh(X_\al \times Z)$. We want to show the second factor of 
		\begin{equation*}
		f_* i_{\al*} e_{\cF,X_\al}^R(\cG_\al) \to f_* e_{\cF,X}^R (i_{\al} \times \id_Z)_*(\cG_\al) \to e_{\cF,Y}^R (f \times \id_Z)_* (i_{\al} \times \id_Z)_*(\cG_\al)
	\end{equation*}
is an isomorphism. Given that $X_\al \times Z$ is coherent by Proposition~\ref{prop:cohbasechangeunderafpclimm}, this follows since the first factor is by the previous paragraph and the composition is by the first paragraph. 

In the general case, let $Y \cong \colim Y_\al$ be a reasonable presentation, and write $\cF \cong j_*(\cF')$ for some reasonable geometric substack $j: Z' \to Z$ and some $\cF' \in \Coh(Z')$. For any $\al$ we have a diagram
\begin{equation*}%\label{eq:up!compcube}
	\begin{tikzpicture}[baseline=(current  bounding  box.center),thick,>=\arrtip]
			\newcommand*{\ha}{1.5}; \newcommand*{\hb}{1.5}; \newcommand*{\hc}{1.5};
\newcommand*{\va}{-.9}; \newcommand*{\vb}{-.9}; \newcommand*{\vc}{-.9}; 
		\node (ab) at (\ha,0) {$X_\al$};
		%\node (ad) at (\ha+\hb+\hc,0) {$X' \times \Spec B$};
		\node (ad) at (\ha+\hb+\hc,0) {$X_\al \times Z'$};
		\node (ba) at (0,\va) {$X$};
		\node (bc) at (\ha+\hb,\va) {$X \times Z$};
		\node (cb) at (\ha,\va+\vb) {$Y_\al$};
		%\node (cb) at (\ha,\va+\vb) {$\Spec A$};
		\node (cd) at (\ha+\hb+\hc,\va+\vb) {$Y_\al \times Z'$};
		%\node (cd) at (\ha+\hb+\hc,\va+\vb) {$\Spec A \ot B$};
		\node (da) at (0,\va+\vb+\vc) {$Y$};
		\node (dc) at (\ha+\hb,\va+\vb+\vc) {$Y \times Z$};
		\draw[<-] (ab) to node[above] {$ $} (ad);
		\draw[->] (ab) to node[above left, pos=.25] {$i'_\al $} (ba);
		\draw[->] (ab) to node[right,pos=.2] {$f_\al $} (cb);
		\draw[->] (ad) to node[below right] {$ $} (bc);
		\draw[->] (ad) to node[right] {$ $} (cd);
		\draw[->] (ba) to node[left] {$f $} (da);
		\draw[<-] (cb) to node[above,pos=.25] {$ $} (cd);
		\draw[->] (cb) to node[above left, pos=.25] {$i_\al $} (da);
		\draw[->] (cd) to node[below right] {$ $} (dc);
		\draw[<-] (da) to node[above,pos=.75] {$ $} (dc);
		
		\draw[-,line width=6pt,draw=white] (ba) to  (bc);
		\draw[<-] (ba) to node[above,pos=.75] {$ $} (bc);
		\draw[-,line width=6pt,draw=white] (bc) to  (dc);
		\draw[->] (bc) to node[right,pos=.2] {$ $} (dc);
	\end{tikzpicture}
\end{equation*} 
with all faces but the top and bottom Cartesian. We have a diagram
\begin{equation*}
	\begin{tikzpicture}
		[baseline=(current  bounding  box.center),thick,>=\arrtip]
		\newcommand*{\ha}{5.2}; \newcommand*{\hb}{5.8};
		\newcommand*{\va}{-1.5};
		\node (aa) at (0,0) {$f_{\al*} i'^!_\al e_{\cF,X}^R(\cG)$};
		\node (ab) at (\ha,0) {$i^!_\al f_* e_{\cF,X}^R(\cG)$};
		\node (ac) at (\ha+\hb,0) {$i^!_\al e_{\cF,Y}^R (f \times \id_Z)_*(\cG)$};
		\node (ba) at (0,\va) {$f_{\al*} e_{\cF',X_\al}^R (i'_\al \times j)^! (\cG)$};
		\node (bb) at (\ha,\va) {$e_{\cF',Y_\al}^R (f_{\al} \times \id_{Z'})_* (i'_\al \times j)^!(\cG)$};
		\node (bc) at (\ha+\hb,\va) {$e_{\cF',Y_\al}^R  (i_\al \times j)^! (f \times \id_Z)_{*}(\cG)$};
		\draw[->] (aa) to node[above] {$ $} (ab);
		\draw[->] (ab) to node[above] {$  $} (ac);
		\draw[->] (ba) to node[above] {$ $} (bb);
		\draw[->] (bb) to node[above] {$ $} (bc);
		\draw[->] (aa) to node[below,rotate=90] {$\sim $} (ba);
		%\draw[->] (ab) to node[right] {$h' $} (bb);
		\draw[->] (ac) to node[below,rotate=90] {$\sim $} (bc);
	\end{tikzpicture} 
\end{equation*}
in $\IndCoh(Y_\al)$, where the vertical isomorphisms are given by (\ref{eq:eFXRupper!com}). Since the functors $i^!_\al$ determine an isomorphism $\IndCoh(Y) \cong \lim \IndCoh(Y_\al)$ in  $\Cathatinfty$, it suffices to show the top right arrow
is an isomorphism for all $\al$. Proposition~\ref{prop:cohbasechangeunderafpclimm} implies that $X_\al$, $X_\al \times Z'$, and $Y_\al \times Z'$ (note that e.g. $i_\al \times j$ factors as $(i_\al \times \id_Z) \circ (\id_{Y_\al} \times j)$). The claim then follows since the top left and bottom right arrows are isomorphisms by Proposition \ref{prop:up!low*geom} and the bottom left is by the previous paragraph. 
\end{proof}

As with the analogous Proposition \ref{prop:up!low*geom}, the coherence hypotheses in Proposition \ref{prop:eFXRlower*} simplify the proof considerably but are not entirely essential. With more work one can show the following extension. 

\begin{Proposition}\label{prop:eFXRlower*general}
	Let $X$, $Y$, and $Z$ be ind-geometric stacks such that $Y$ and $Z$ are reasonable, and let $f: X \to Y$ be a morphism of ind-finite cohomological dimension. 
	Then for any $\cF \in \Coh(Z)$ and $\cG \in \IndCoh(X \times Z)^+$ the Beck-Chevalley map  $f_* e_{\cF,X}^R(\cG) \to e_{\cF,Y}^R (f \times \id_Z)_*(\cG)$ is an isomorphism. 
\end{Proposition} 

Next recall that if $f: X \to Y$ is a proper morphism of geometric stacks and $\cF \in \QCoh(X)$, the projection isomorphism $f_*(\cF \ot f^*(-)) \cong f_*(\cF) \ot -$ yields an isomorphism
\begin{equation}\label{eq:QCHomupper!}
	f_* \cHom(\cF, f^!(-)) \cong \cHom(f_*(\cF), -) 
\end{equation}
of right adjoints. Suppose instead that $X$ and $Y$ are reasonable ind-geometric stacks and that $f$ is ind-proper and almost ind-finitely presented. Again $f_*: \IndCoh(X) \to \IndCoh(Y)$  will typically not have a left adjoint, but we can define a transformation
\begin{equation}\label{eq:cHomup!map}
	f_* \cHom(\cF, f^!(-)) \to \cHom(f_*(\cF), -)
\end{equation} 
of functors $\IndCoh(Y) \to \IndCoh(Y)$ as the composition
$$ f_* e_{\cF,X}^R \Delta_{X *} f^! \to e_{\cF,Y}^R (f \times \id_X)_*\Delta_{X *} f^! \to e_{f_*(\cF),Y}^R \Delta_{Y*} $$
of Beck-Chevalley maps. Note that we implicitly use the isomorphism $e_{f_*(\cF),Y}^R \cong   e_{\cF,Y}^R (\id_Y \times f)^!$ of (\ref{eq:eFXRupper!com}). 
In the geometric case one can check that if we restrict to left bounded subcategories, (\ref{eq:cHomup!map}) is identified with (\ref{eq:QCHomupper!}) under the equivalences $\IndCoh(-)^+ \cong \QCoh(-)^+$. 

\begin{Proposition}\label{prop:indhomrelationcoherent}
	Let $X$ and $Y$ be coherent ind-geometric stacks such that $X \times X$, $X \times Y$, and $Y \times Y$ are coherent, and let $f: X \to Y$ be an ind-proper, almost ind-finitely presented morphism. Then for any $\cF \in \Coh(X)$ and $\cG \in \IndCoh(Y)$ the natural map $f_*\cHom(\cF, f^!(\cG)) \to \cHom(f_*(\cF), \cG)$ is an isomorphism. 
\end{Proposition}
\begin{proof}
	Follows from Propositions \ref{prop:up!low*geom} and \ref{prop:eFXRlower*}. 
\end{proof}

Similarly, conditioned on proofs of Propositions \ref{prop:up!low*indgeomgeneral} and \ref{prop:eFXRlower*general} one obtains the following. 

\begin{Proposition}\label{prop:indhomrelation}
	Let $X$ and $Y$ be reasonable ind-geometric stacks, and let $f: X \to Y$ be an ind-proper, almost ind-finitely presented morphism of finite cohomological dimension. Then for any $\cF \in \Coh(X)$ and $\cG \in \IndCoh(Y)^+$ the natural map $f_*\cHom(\cF, f^!(\cG)) \to \cHom(f_*(\cF), \cG)$ is an isomorphism. 
\end{Proposition}

At the level of objects, the extension of sheaf Hom from the geometric to the ind-geometric setting is uniquely determined by these results, since we can always write $\cF \in \Coh(X)$ as~$i_*(\cF')$ for some reasonable geometric substack $i: X' \to X$ and $\cF' \in \Coh(X')$ (note that $i$ is of cohomological dimension zero). 

With Proposition \ref{prop:eFXRlower*} in hand, we can also establish the ind-geometric extension of Proposition \ref{prop:eFXcHomcompat}. 

\begin{Proposition}\label{prop:eFXReFXindgeomcoh}
	Let $X$, $Z$, and $Z'$ be reasonable ind-geometric stacks such that $X$, $X \times Z$, $Z' \times X$, and $Z' \times X \times Z$ are coherent. Then for all $\cF \in \Coh(Z)$, $\cF' \in \Coh(Z')$, and $\cG \in \IndCoh(X \times Z)$, the Beck-Chevalley map $\eop_{\cF',X} e_{\cF,X}^R(\cG) \to e_{\cF, Z' \times X}^R \eop_{\cF',X \times Z}(\cG)$ is an isomorphism.
\end{Proposition}
\begin{proof}
	First suppose $X$, $Z$, and $Z'$ are truncated and geometric. Then $\eop_{\cF',X} e_{\cF,X}^R(\cG)$ and $e_{\cF, Z' \times X}^R \eop_{\cF',X \times Z}$ are continuous (Proposition \ref{prop:eFXRalmostcont}) and $\IndCoh(X \times Z)$ is compactly generated, so it suffices to consider $\cG \in \Coh(X \times Z)$. But the restrictions of all functors involved to left bounded subcategories commute with the equivalences $\IndCoh(-)^+ \cong \QCoh(-)^+$ (Proposition \ref{prop:eFXRprops1}), so the claim follows from Proposition \ref{prop:eFXeFXRcompat}. 
	
	Still assuming $Z$ and $Z'$ are truncated and geometric, let $X \cong \colim X_\al$ be a reasonable presentation. For any $\al$ we have a diagram 
	\begin{equation*}
		\begin{tikzpicture}
			[baseline=(current  bounding  box.center),thick,>=\arrtip]
			\newcommand*{\ha}{5.1}; \newcommand*{\hb}{6.1};
			\newcommand*{\va}{-1.5};
			\node (aa) at (0,0) {$ \eop_{\cF',X_\al} i^!_\al e_{\cF,X}^R$};
			\node (ab) at (\ha,0) {$(\id_{Z'} \times i_\al)^! \eop_{\cF',X} e_{\cF,X}^R$};
			\node (ac) at (\ha+\hb,0) {$(\id_{Z'} \times i_\al)^! e_{\cF, Z' \times X}^R \eop_{\cF',X \times Z}$};
			\node (ba) at (0,\va) {$ \eop_{\cF',X_\al} e_{\cF,X_\al}^R (i_\al \times \id_Z)^!$};
			\node (bb) at (\ha,\va) {$ e_{\cF,Z' \times X_\al}^R\eop_{\cF',X_\al \times Z}(i_\al \times \id_Z)^!$};
			\node (bc) at (\ha+\hb,\va) {$e_{\cF,Z' \times X_\al}^R(\id_{Z'} \times i_\al \times \id_Z)^! \eop_{\cF', X \times Z}$,};
			\draw[->] (aa) to node[above] {$ $} (ab);
			\draw[->] (ab) to node[above] {$  $} (ac);
			\draw[->] (ba) to node[above] {$ $} (bb);
			\draw[->] (bb) to node[above] {$ $} (bc);
			\draw[->] (aa) to node[below,rotate=90] {$\sim $} (ba);
			%\draw[->] (ab) to node[right] {$h' $} (bb);
			\draw[->] (ac) to node[below,rotate=90] {$\sim $} (bc);
		\end{tikzpicture} 
	\end{equation*}
	the vertical isomorphisms being given by (\ref{eq:eFXRupper!com}). Since the functors $(\id_{Z'} \times i_\al)^!$ determine an isomorphism $\IndCoh(Z' \times X) \cong \lim \IndCoh(Z' \times X_\al)$ in $\Cathatinfty$, it suffices to show the top right arrow is an isomorphism for all $\al$. But, given that $X_\al \times Z$, $Z' \times X_\al$, and $Z' \times X_\al \times Z$ are coherent by Proposition \ref{prop:cohbasechangeunderafpclimm}, the bottom right and top left arrows are isomorphisms by Proposition \ref{prop:eFXupper!compatindgeom}, and the bottom left is by the first paragraph. 
	
	Now let $Z$ and $Z'$ be ind-geometric, and write $\cF \cong i_{\al*}(\cF_\al)$, $\cF' \cong i_{\be*}(\cF'_\be)$ for some reasonable geometric substacks $i_{\al}: Z_\al \to Z$, $i_\be: Z'_\be \to Z'$ and some $\cF_\al \in \Coh(Z_\al)$, $\cF'_\be \in \Coh(Z'_\be)$. Using (\ref{eq:eFXlow*compat}) and (\ref{eq:eFXRupper!com}) the map in the statement factors as 
	\begin{align*}
		(i_\be \times \id_X)_* \eop_{\cF'_\be, X} e_{\cF_\al,X}^R (\id_X \times i_\al)^! & \to (i_\be \times \id_X)_* e_{\cF_\al,Z'_\be \times X}^R \eop_{\cF'_\be, X \times Z_\al}  (\id_X \times i_\al)^! \\
		& \to (i_\be \times \id_X)_* e_{\cF_\al,Z'_\be \times X}^R   (\id_{Z'_\be \times X} \times i_\al)^! \eop_{\cF'_\be, X \times Z} \\
		& \to e_{\cF_\al,Z' \times X}^R (i_\be \times \id_{X \times Z_\al})_*   (\id_{Z'_\be \times X} \times i_\al)^! \eop_{\cF'_\be, X \times Z} \\
		& \to e_{\cF_\al,Z' \times X}^R   (\id_{Z' \times X} \times i_\al)^! (i_\be \times \id_{X \times Z})_*  \eop_{\cF'_\be, X \times Z}. 
	\end{align*}
	But, given that $X \times Z'$ and $Y \times Z'$ are coherent by Proposition \ref{prop:cohbasechangeunderafpclimm}, the first factor is an isomorphism by the previous paragraph, the second is by Proposition \ref{prop:eFXupper!compatindgeom}, the third is by Proposition \ref{prop:eFXRlower*}, and the fourth is by Proposition \ref{sec:!pullpush}. 
\end{proof}

As with other results in this section, one can prove a weaker claim in the general case. 

\begin{Proposition}
	Let $X$, $Z$, and $Z'$ be ind-geometric stacks with $Z$ and $Z'$ reasonable. Then for all $\cF \in \Coh(Z)$, $\cF' \in \Coh(Z')$, and $\cG \in \IndCoh(X \times Z)^+$, the Beck-Chevalley map $\eop_{\cF',X} e_{\cF,X}^R(\cG) \to e_{\cF, Z' \times X}^R \eop_{\cF',X \times Z}(\cG)$ is an isomorphism.
\end{Proposition}

\bibliographystyle{amsalpha}
\bibliography{bibKoszul}

\end{document}